\documentclass[oneside, 11pt]{article}
\usepackage{setspace}
\usepackage{titlesec}
\usepackage[utf8]{inputenc}
\usepackage{tabularx}	
\usepackage{etoolbox}
\usepackage{mathtools}
\usepackage{amsmath}
\usepackage{amssymb} 
\usepackage{amsfonts}
\usepackage{mathrsfs} 
\usepackage{amsthm,pifont}
\usepackage{graphicx}
\usepackage{bbm}
\usepackage[colorlinks=true, pdfstartview=FitV, linkcolor=blue,
citecolor=blue, urlcolor=blue, pagebackref=false]{hyperref}
\usepackage[nottoc]{tocbibind}
\usepackage{xcolor}

\usepackage{tikz}
\usepackage{pgfplots,pgf}

\usepackage{cite}
\usepackage{braket}
\usepackage{leftidx,tensor}
\usepackage{dsfont}
\usepackage{calc}
\usepackage{caption}
\usepackage{subcaption}
\usepackage{color}
\usepackage[printonlyused,nohyperlinks]{acronym}
\usetikzlibrary{decorations.markings}
\usetikzlibrary{shapes.geometric}
\usetikzlibrary{patterns}
\tikzstyle{vertex}=[circle, draw, inner sep=0pt, minimum size=6pt]
\newcommand{\vertex}{\node[vertex]}

\usepackage[utf8]{inputenc}
\usepackage[english]{babel}
\usepackage{enumitem}

\allowdisplaybreaks

\newtheorem{theorem}{Theorem}[section]
\makeatletter
\patchcmd{\ttlh@hang}{\parindent\z@}{\parindent\z@\leavevmode}{}{}
\patchcmd{\ttlh@hang}{\noindent}{}{}{}
\makeatother
\titleformat*{\section}{\large\bfseries}
\titleformat*{\subsection}{\small\bfseries}
\titleformat*{\subsubsection}{\small\bfseries}
\titleformat*{\paragraph}{\small\bfseries}
\titleformat*{\subparagraph}{\small\bfseries}

\newcommand{\N}{\mathbb{N}}
\newcommand{\R}{\mathbb{R}}
\newcommand{\Z}{\mathbb{Z}}
\newcommand{\E}{\mathbb{E}}

\newcommand{\p}{\mathbb{P}}

\newcommand{\md}{\ensuremath{\mathrm{d}}}
\newcommand{\eps}{\varepsilon}
\newcommand{\cA}{\mathcal{A}}
\newcommand{\cB}{\mathcal{B}}
\newcommand{\cC}{\mathcal{C}}

\newcommand{\cE}{\mathcal{E}}
\newcommand{\cF}{\mathcal{F}}
\newcommand{\cG}{\mathcal{G}}

\newcommand{\cT}{\mathcal{T}}

\newcommand{\mz}{\mathbf{0}}

\newtheorem{lemma}[theorem]{Lemma}
\newtheorem{remark}[theorem]{Remark}
\newtheorem{proposition}[theorem]{Proposition}
\newtheorem{definition}[theorem]{Definition}
\newtheorem{notation}[theorem]{Notation}
\newtheorem{corollary}[theorem]{Corollary}
\newtheorem{claim}[theorem]{Claim}

\newtheorem{question}[theorem]{Question}

\definecolor{darkgreen}{RGB}{30,137,37}
\definecolor{lightblue}{RGB}{48,121,216}

\usepackage{geometry}
\begin{document}

	\title{\vspace{-1cm}Continuity of the critical value and a shape theorem for long-range percolation}

	\author{Johannes B\"aumler\footnote{ \textsc{Mathematical Institute, University of Koblenz}.\\ \text{  \ \ \ \ \ \ } E-Mail: \href{mailto:jbaeumler@uni-koblenz.de}{jbaeumler@uni-koblenz.de}
	}
	}

	\maketitle
	
	\begin{center}
		\parbox{12cm}{ \textbf{Abstract.} 
			Consider supercritical long-range percolation on $\Z^d$ where two vertices $x,y \in \Z^d$ are connected with probability asymptotic to $\|x-y\|^{-s}$ for some $s>2d$. Conditioned that the origin is in the infinite cluster, we prove a shape theorem for the set of points that can be reached within $n$ steps from the origin. 
			As part of the proof, we show that for long-range percolation with polynomially decaying connection probabilities in dimensions $d\geq 2$, the critical value depends continuously on the precise specifications of the model.
	}
	\end{center}
	
	\let\thefootnote\relax\footnotetext{{\sl MSC Class}: 82B43, 60K35, 05C12, 05C81 }
	\let\thefootnote\relax

%

\section{Introduction}

Consider Bernoulli bond percolation on $\Z^d$, where we include an edge between the vertices $x,y\in \Z^d$ with probability $1-e^{-\beta J(x-y)}$, independent of all other edges. The function $J:\Z^d \rightarrow \left[0,\infty\right)$ is a \textbf{symmetric} kernel, where symmetric means that for all vectors $x=(x_1, \ldots, x_d) \in \Z^d$, the value of $J\left(x\right)$ is invariant under flipping the sign of one or more coordinates and under permutations of the coordinates of $x$. We denote the resulting probability measure by $\p_{\beta,J}$ or $\p_\beta$ and its expectation by $\E_{\beta,J}$ or $\E_\beta$. Edges that are included are also referred to as open. We write $x \sim y$ if the edge between $x$ and $y$ is open. We do not allow the case $x=y$ here, i.e., we do not consider self-loops. So in particular the value $J(\mz)$ does not influence the percolation configuration and we can also consider $J$ as a function from $\Z^d \setminus \{\mz\}$ to $\left[0,\infty\right)$.
From this construction, we directly get that the resulting measure is translation invariant.  As the kernel $J$ has all symmetries of the integer lattice, so does the measure $\p_\beta$. We are interested in the case where the kernel is \textbf{integrable}, meaning that $\sum_{x\in \Z^d} J(x)< \infty$.
The integrability condition guarantees that the resulting open subgraph is almost surely locally finite under the measure $\p_\beta$. Further, we require that the kernel $J$ is \textbf{irreducible}, meaning that for all $x\in \Z^d$ there exist $\mz=a_0,a_1,\ldots,a_n=x\in \Z^d$ such that $J(a_{i}-a_{i-1}) > 0$ for all $i\in \{1,\ldots,n\}$. Bond percolation on $\Z^d$ with the measure $\p_{\beta,J}$ creates clusters, which are the connected components in the resulting random graph. Write $K_x$ for the cluster containing the vertex $x\in \Z^d$. A central question in percolation theory is the emergence of infinite clusters, for which we define the critical parameter $\beta_c$ by
\begin{equation*}
\beta_{c} = \beta_c(J) = \inf \left\{\beta \geq 0 : \p_\beta \left( |K_\mz| = \infty \right) > 0 \right\} \text.
\end{equation*}
A comparison with a subcritical branching process shows that there are no infinite clusters for $\beta < \left(\sum_{x\in \Z^d} J(x)\right)^{-1}$, which implies $\beta_{c}>0$. In dimensions $d\geq 2$ it is well known that $\beta_c < \infty$, as long as $J\neq 0$. For dimension $d=1$, Newman and Schulman showed that $\beta_c < \infty$ as long as $J$ does not decay faster than quadratic \cite{newman1986one,schulman1983long}. For (long-range) percolation there is also the question whether there can exist two or more infinite open clusters simultaneously. It was first proven by Aizenman, Kesten, and Newman \cite{aizenman1987uniqueness} that the number of infinite open clusters is almost surely either $0$ or $1$. Later, Burton and Keane \cite{burton1989density} gave a different proof of this fact by using the amenability of $\Z^d$. This proof also works for long-range percolation on $\Z^d$. Indeed, the integrability of the kernel $J$ implies that
\begin{equation*}
	\sum_{x: \|x\| \leq n} \ \sum_{y: \|y\| > n} \p_{\beta,J} (x\sim y) = o(n^d)
\end{equation*}
which implies uniqueness of the infinite component, cf. \cite{burton1989density}. So in particular, the {\sl infinite cluster} $\cC_\infty$ defined by
\begin{equation*}
	\cC_\infty \coloneqq \left\{x \in \Z^d : |K_x|= \infty\right\}
\end{equation*}
is almost surely connected.\\

Long-range percolation is mostly studied in the case where $J(x) \simeq \|x\|^{-s}$ for some parameter $s >0$, where we write $J(x)\simeq \|x\|^{-s}$ if the ratio between the two quantities satisfies $\eps < \frac{J(x)}{\|x\|^{-s}} < \eps^{-1}$ for a small enough $\eps>0$ and all $x\in \Z^d$ with $\|x\|$ large enough. If $s > d$, then the resulting kernel $J$ is integrable.  In general, it is expected that for $s > 2d$ the resulting random graph looks similar to nearest-neighbor percolation, is very well connected for $s < 2d$, and shows a self-similar behavior for $s = 2d$. See \cite{baumler2022behavior,baumler2022distances,benjamini2001diameter,berger2004lower,biskup2004scaling,biskup2011graph} for results pointing in this direction. 

In this paper, we are interested in several different properties of the {\sl supercritical} percolation cluster, i.e., in the case $\beta > \beta_c$. We are mostly interested in the {\sl chemical distance} between different points in the infinite cluster. 

For two points $x,y \in \Z^d$, we write $D(x,y)$ for the chemical distance (also called graph distance or hop-count distance) between $x$ and $y$, which is the length of the shortest open path between $x$ and $y$. For a kernel $J$ satisfying $J(x) \simeq \|x\|^{-s}$ and $\beta > \beta_c(J)$, the typical graph distance between $x$ and $y$ depends heavily on the parameter $s$. There are five different regimes, with the transitions happening at $s=d$ and $s=2d$:

\begin{itemize}
	\item For $s<d$, the graph distance between any two points $x,y\in \Z^d$ is at most $\lceil\frac{d}{d-s}\rceil$ \cite{benjamini2011geometry}.
	
	\item For $s=d$, the graph distance between $x$ and $y$ is of order $\frac{\log(\|x-y\|)}{\log \log(\|x-y\|)}$ \cite{coppersmith2002diameter}. 
	
	\item For $s\in (d,2d)$, the graph distance between $x$ and $y$ is of order $\log(\|x-y\|)^\Delta$ where $\Delta^{-1} = \log_2 \left(\frac{2d}{s}\right)$ \cite{biskup2004scaling, biskup2011graph, biskup2019sharp}.
	(Note that the graph containing the open edges is locally finite if and only if $s>d$.)
	
	\item  For $s=2d$, the graph distance between $x$ and $y$ is of order $\|x-y\|^{\theta}$ for some $\theta \in (0,1)$ \cite{baumler2022distances,baumler2022behavior,ding2013distances}.
	
	\item For $s>2d$, the graph distance between $\mz$ and $x$ satisfies
	\begin{equation}\label{eq:noam}
		\liminf_{x\to \infty} \frac{D(\mz,x)}{\|x\|} > 0 \text{ almost surely},
	\end{equation}
	see \cite{berger2004lower,luchtrath2026all}. 
	Furthermore, Berger conjectured that an analogous upper bound holds in the supercritical regime \cite[Conjecture 3]{berger2004lower}.
\end{itemize}

\subsection{Main results}

In this paper, we provide a strong affirmative answer to \cite[Conjecture 3]{berger2004lower}. We do this by proving a shape theorem for the ball of radius $n$ (in the intrinsic metric) around the origin. For the shape theorem, we need to define distances on $\R^d$ instead of $\Z^d$, which we do in the following. We identify a point $x\in \Z^d$ with the set $x+\left[-\tfrac{1}{2},\tfrac{1}{2}\right)^d$. For $x\in \R^d$, we write $x_d$ for the unique point in the integer lattice $x_d \in \Z^d$ that satisfies $x \in x_d+\left[-\tfrac{1}{2},\tfrac{1}{2}\right)^d$.
If $\sum_{x} J(x)< \infty$, the long-range percolation graph has holes for $\beta \in (\beta_c(J),\infty)$, in the sense that with probability $1$ there are infinitely many connected components, and only one of them is infinite. In particular, the graph is not connected. For $x,y \in \Z^d$, we write $x \leftrightarrow y$ if there exists an open path connecting $x$ to $y$, we write $x\leftrightarrow \infty$ if there exists an open path from $x$ to infinity, and we write $\cC_\infty = \left\{x \in \Z^d: x \leftrightarrow \infty\right\}$ for the infinite cluster. To circumvent the non-connectedness of $\Z^d$, for every $a\in \Z^d$ we define $\hat{a}$ as the point $y\in \cC_\infty$ such that $\|a-y\|_\infty$ is minimal. If there are several such points $y$, we choose the smallest in the lexicographic ordering. For $x\in \R^d$, we define $\hat{x}\coloneqq\widehat{x_d}$. So in particular the point $\widehat{x_d}$ is a point in the infinite cluster $\mathcal{C}_\infty$. (Note that this definition is consistent with the earlier definition of $\hat{x}$ when $x\in \Z^d$.) The uniqueness of the infinite open cluster implies that in the supercritical regime one has almost surely $D(\hat{x},\hat{y})<\infty$ for all $x,y\in \R^d$. We also define the pseudometric $\hat{D}$ on $\R^d$ by
\begin{equation*}
	\hat{D}:\R^d \times \R^d \to \left[0,\infty\right), \ \ \ \ \ \ \
	\hat{D}(x,y) = D(\hat{x},\hat{y}).
\end{equation*}
Note that this pseudometric satisfies the triangle inequality and $\hat{D}(x,y)=D(x,y)$ for all $x,y\in \mathcal{C}_\infty$, but does not satisfy $\hat{D}(x,y)=0$ for all distinct $x,y \in \R^d$, as one can have distinct $x,y\in \R^d$ with $\hat{x}=\hat{y}$. We define the balls in this metric by
\begin{equation*}
	\hat{B}_t(x) = \left\{z \in \R^d : \hat{D}(z,x) \leq t \right\}.
\end{equation*}

\begin{theorem}\label{theo:chemical dist}
	Let $d\geq 2$, $s>2d$, and let $J : \Z^d \to \left[0,\infty\right)$ be a symmetric and irreducible kernel satisfying $J(x) = \mathcal{O}(\|x\|^{-s})$. Let $\beta > \beta_c(J)$. Then there exists a deterministic function $\mu:\R^d \rightarrow \left[0,\infty\right)$ such that for all $x\in \R^d$
	\begin{equation}\label{eq:chemdist limsup}
		\lim_{n\to \infty} \frac{\hat{D}(\mz,nx)}{n} = \mu(x) \ \ \text{ almost surely and in $L_1$.}
	\end{equation}
	The function $\mu$ is a norm on $\R^d$, so in particular $\mu(x) > 0$ for all $x\in \R^d \setminus \{\mz\}$. The convergence is uniform in the sense that
	\begin{equation}\label{eq:strongshape1}
		\lim_{n\to \infty} \ \sup_{x\in \R^d : \|x\|\geq n} \Big| \frac{\hat{D}(\mz,x) - \mu(x)}{\|x\|} \Big| = 0 \text{ almost surely.}
	\end{equation}
	Define the $1$-ball in the $\mu$-metric by $\cB_\mu = \left\{z \in \R^d: \mu(z)\leq 1\right\}$. Then for all $\eps > 0$ there exists almost surely some $t_0 < \infty$ such that
	\begin{equation}\label{eq:strongshape2}
		(1-\eps) \cB_\mu \subset \frac{\hat{B}_t(\mz)}{t} \subset (1+\eps) \cB_\mu 
	\end{equation}
	for all $t\geq t_0$.
\end{theorem}

Chemical distances in the setting of finite-range percolation were also considered in \cite{garet2004asymptotic,cerf2016weak,antal1996chemical}; we use the corresponding result for finite-range percolation as an input and show that, typically, all points contained in the infinite percolation cluster are relatively close (both in terms of Euclidean and chemical distance) to the infinite finite-range cluster. The convergence of the distance $D(0,x)$, when properly rescaled, was also studied for $s\in (d,2d)$ in \cite{biskup2019sharp,biskup2024arithmetic} and for $s=2d$ in \cite{baumler2023polynomial,ding2026uniqueness}.

 One reason for the technical difficulties in the proof of Theorem \ref{theo:chemical dist} is that the distance $\hat{D}(x,y)= D(\hat{x},\hat{y})$ is not monotone in the edge set. Indeed, including and edge, say between $x$ and $z\in \Z^d$ with $z \in \cC_\infty$ might change the position of $\hat{x}$ and thus also the distance $\hat{D}(x,y)= D(\hat{x},\hat{y})$ can increase when adding an edge to the percolation environment.
The main technical innovations in the proof of Theorem \ref{theo:chemical dist} are Propositions \ref{prop:stretched exp} and \ref{propo:zeta} below. These propositions allow to rule out that there are points $x,y \in B_n(\mz) = \{ z\in \Z^d: \|z\|_\infty \leq n\}$ with an unusually large (but finite) graph distance $D(x,y)$.\\

In order to use known results for chemical distances for finite-range percolation, for example \cite{garet2004asymptotic,cerf2016weak,antal1996chemical}, we need to compare {\sl finite-range} percolation with {\sl long-range} percolation. This naturally leads to the following question:
\begin{question}
	Let $J:\Z^d \to \left[0,\infty\right)$ be a kernel and let $\beta>\beta_c(J)$. Does there exist $N \in \N$ so that the kernel $\tilde{J}$ defined by 
	\begin{equation*}
		\tilde{J}(x)= \begin{cases}
			J(x) & \text{ if } \|x\| \leq N\\
			0 & \text{ else }
		\end{cases}
	\end{equation*}
	satisfies $\p_{\beta, \tilde{J}}\left(|K_\mz|=\infty\right)>0$?
\end{question}
This question is known as the so-called {\sl(long-range percolation) truncation problem} and was first studied by Meester and Steif for long-range percolation with exponentially decaying tails \cite{meester1996continuity}. Their proof followed a similar strategy as the proof of the continuity of the critical value for slab percolation by Grimmett and Marstrand \cite{grimmett1990supercritical}. For long-range percolation with polynomially decaying tails, the truncation problem was studied by Berger \cite{berger2002transience} under the assumption that $J(x)\simeq \|x\|^{-s}$ for some $s \in (d,2d)$. 
With a possible application to the proof of Theorem \ref{theo:chemical dist} in mind, we study the truncation problem when $J(x)$ decays polynomially with exponent $s \geq 2d$. We will mostly work under the assumption where the kernel $J$ satisfies
\begin{equation}\label{eq:condition main}
	J(x) \leq C \|x\|^{-2d}
\end{equation}
for some constant $C<\infty$ and all $x \in \Z^d\setminus \{\mz\}$. Here we show that for $\beta > \beta_c(J)$ the truncation problem has an affirmative answer. This result provides an answer to a question posed by Berger \cite[Question 6.5]{berger2002transience}.

\begin{theorem}\label{theo:main}
	Let $d\geq 2$ and let $J:\Z^d \to \left[0,\infty\right)$ be an irreducible and symmetric kernel such that $J(x)=\mathcal{O}(\|x\|^{-2d})$. Let $\beta > \beta_c\left(J\right)$. Then there exists $N \in \N$ so that the kernel $\tilde{J}$ defined by 
	\begin{equation*}
		\tilde{J}(x)= \begin{cases}
			J(x) & \text{ if } \|x\| \leq N\\
			0 & \text{ else }
		\end{cases}
	\end{equation*}
	satisfies $\p_{\beta, \tilde{J}}\left(|K_\mz|=\infty\right)>0$.
\end{theorem}

The proof of this Theorem follows a standard ``Grimmett-Marstrand-approach" that relies on the symmetries and properties of the integer lattice. We will use the assumption $J(x)=\mathcal{O}(\|x\|^{-2d})$ only at one point in the proof, which is in Lemma \ref{lem2}. On the technical side, the other steps to obtain Lemma \ref{lem2} and to see how Lemma \ref{lem2} implies Theorem \ref{theo:main} follow using similar arguments as the proofs of Grimmett and Marstrand \cite{grimmett1990supercritical} and Meester and Steif \cite{meester1996continuity}, respectively.
Also the precise setup of the model, i.e., that $\p_\beta \left(\{x,y\} \text{ open}\right)=1-\exp(\beta J(x-y))$ is important for the proof, as ``sprinkling" thus increases the probability that edges are open for all edges. This is used in Lemma \ref{lem4}. We also consider a slightly different model of long-range percolation in Theorem \ref{theo:p1 larger} below, where the ``sprinkling" only increases the probability that nearest-neighbor edges are open. The same model was also considered by Meester and Steif \cite{meester1996continuity}.\\

The class of kernels that still percolate after removing all long enough edges is very important for this paper. As we will refer to it quite often in the rest of the paper, we give such kernels a name with the following definition.

\begin{definition}
	We call a kernel $J:\Z^d \to \left[0,\infty\right)$ \textbf{resilient} if for all $\beta > \beta_c(J)$ there exists  $N \in \N$ so that the kernel $\tilde{J}$ defined by 
	\begin{equation*}
		\tilde{J}(x)= \begin{cases}
			J(x) & \text{ if } \|x\| \leq N\\
			0 & \text{ else }
		\end{cases}
	\end{equation*}
	satisfies $\p_{\beta, \tilde{J}}\left(|K_\mz|=\infty\right)>0$.
\end{definition}

So phrased in this language, Theorem \ref{theo:main} together with the results of Berger \cite[Theorem 1.8]{berger2002transience} show the following.

\begin{remark}
	Let $d\geq 2$ and let $J:\Z^d \to \left[0,\infty\right)$ be an irreducible and symmetric kernel satisfying $J(x)= \mathcal{O}(\|x\|^{-2d})$ or $J(x) \simeq \|x\|^{-s}$ for some $s\in (d,2d)$. Then $J$ is resilient.
\end{remark}

Note that a kernel $J:\Z^d \to \left[0,\infty\right)$ can only be resilient for dimensions $d\geq 2$, as a finite-range model can never percolate in dimension $d=1$.
Resilience of kernels was previously established by Berger for long-range percolation with kernel $J(x) \simeq \|x\|^{-s}$ for some $s \in (d,2d)$ \cite{berger2002transience} and by Meester and Steif for long-range percolation with exponential decay of the connection probability \cite{meester1996continuity}. Furthermore,  several works established resilience for different kernels $J$ with $\sum_{x} J(x)=\infty$ \cite{friedli2004longrange,friedli2006truncation,sidoravicius1999truncated,menshikov2001note,baumler2024truncation}, i.e., for the case where $\beta_c(J)=0$. The general case, i.e., assuming irreducibility and $\sum_{x} J(x)=\infty$ only, is still open in dimension $d=2$ \cite[Conjecture 1.4]{baumler2024truncation}. For dependent percolation models, resilience of the kernel was shown by Mönch for inhomogeneous long-range percolation in the weak decay regime \cite{monch2023inhomogeneous} and by Dembin and Tassion for Boolean percolation \cite{dembin2022almost}.\\

Resilience of kernels (Theorem \ref{theo:main}) already has several interesting implications, which, given Theorem \ref{theo:main}, follow from relatively simple proofs. We prove some of these results for the sake of future reference. We generally divide the following results into two classes. Theorem \ref{theo:locality} and Corollary \ref{coro:percoprobconv} deal with the continuity of the functions $\beta_c\left(\cdot\right)$ and $\p_{\beta, J} \left(|K_\mz|=\infty\right)$, whereas Theorems \ref{theo:large clusters} and \ref{theo:transience} deal with structural properties of the infinite cluster for $\beta>\beta_c$.\\

The next result we present is the `locality' of the long-range percolation graph in dimensions $d\geq 2$. We say that $J_n$ converges to $J$ in $L_1$ (of $\Z^d$) if $\sum_{x\in \Z^d} |J_n(x)-J(x)|$ converges to $0$ as $n \to \infty$.

\begin{theorem}\label{theo:locality}
	Let $d\geq 2$, and let $J: \Z^d \to \left[0,\infty\right)$ be a symmetric,  irreducible and resilient kernel.
	Let $\left(J_n\right)_{n\in \N}$ be a sequence of kernels converging to $J$ in $L_1$ of $\Z^d$. Then 
	\begin{equation*}
		\beta_c(J_n) \to \beta_c(J)
	\end{equation*}
	as $n\to \infty$. In particular, this holds for symmetric $J$ satisfying $J(x)\simeq \|x\|^{-s}$ for some $s>d$.
	Further, let $d\geq 1$, let $J: \Z^d \to \left[0,\infty\right)$ be a kernel,  and let  $\left(J_n\right)_{n\in \N}$ be a sequence of kernels converging to $J$ in $L_1$ from above. Then 
	\begin{equation*}
	\beta_c(J_n) \to \beta_c(J).
	\end{equation*}
\end{theorem}

Note that Theorem \ref{theo:main} is a special case of Theorem \ref{theo:locality} and is also used in its proof. Theorem \ref{theo:locality} shows a locality-type result for long-range percolation that requires that the graph is fixed $(\Z^d)$ and only the kernel $J$ varies with $n$. 
A more general version of locality also allows the graphs to change and considers the critical parameter depending on the graph. Locality for short-range percolation graphs was previously established for slabs of $\Z^d$ by Grimmett and Marstrand \cite{grimmett1990supercritical}, for graphs of polynomial growth by Contreras, Martineau, and Tassion \cite{contreras2023locality}, and by Easo and Hutchcroft for general transitive graphs \cite{easo2023critical}. Using locality for (long-range) percolation, one can deduce that the percolation probability $\theta\left(\beta,J\right) = \p_{\beta, J}(|K_\mz|=\infty)$ is continuous outside of the critical points, i.e., at points $(\beta, J)$ for which $\beta\neq\beta_c(J)$.

\begin{corollary}\label{coro:percoprobconv}
	Let $d\geq 2$, let $J$ be an irreducible and resilient kernel, and let $\beta \neq \beta_c\left(J\right)$. Let $\left(J_n\right)_{n\in \N}$ be a sequence of kernels converging to $J$ in $L_1$, and let $(\beta_n)_{n\in \N} \subset (0,\infty)$ be such that $\lim_{n \to \infty}\beta_n = \beta$. Then
	\begin{equation*}
	\lim_{n\to \infty} \theta \left(\beta_n, J_n\right) = \theta \left(\beta,J\right).
	\end{equation*}
\end{corollary}

The next results (Theorems \ref{theo:large clusters} and \ref{theo:transience}) concern properties of the infinite percolation cluster for $\beta > \beta_c(J)$. The important connection to Theorem \ref{theo:main} is that for a resilient kernel $J$ and $\beta> \beta_c(J)$, the infinite percolation cluster $\cC_\infty = \left\{x \in \Z^d : x \leftrightarrow \infty \right\}$ sampled by $\p_{\beta,J}$ already contains an infinite percolation cluster with finite range. 
Due to this inclusion, we can use known results for finite-range percolation and then use the finite-range percolation cluster contained in $\cC_\infty$ in order to prove the corresponding statements for the cluster $\cC_\infty$. Going from the statements of the finite-range cluster to the infinite-range cluster is relatively straightforward in Theorems \ref{theo:large clusters} and \ref{theo:transience}. \\

The first result about the structure of the supercritical cluster concerns the existence of giant clusters for long-range percolation in the supercritical regime. The corresponding result for finite-range percolation was shown by Deuschel and Pisztora in \cite{deuschel1996surface}. For a set $A\subset \Z^d$, we write $|K_{\max}(A)|$ for the size of the largest open component contained in $A$. Note that this is well-defined even if the largest open component in $A$ is not unique.

\begin{theorem}\label{theo:large clusters}
	Let $d\geq 2$, let $J:\Z^d \to \left[0,\infty\right)$ be an irreducible, symmetric, and resilient kernel, and let $\beta > \beta_c(J)$. Then for all $\eps > 0$, there exists $N\in \N$ such that for all $n\geq N$
	\begin{equation*}
		\p_{\beta,J} \left(|K_{\max}\left(B_n(\mz)\right)| \geq (\theta(\beta,J)-\eps)|B_n(\mz)|\right) \geq 1-\eps .
	\end{equation*}
\end{theorem}

Further, we use the result of Theorem \ref{theo:main} to show transience of the simple random walk on the supercritical long-range percolation cluster in dimensions $d\geq 3$. This solves a conjecture by Heydenreich, Hulshof, and Jorritsma \cite{heydenreich2017structures} and S\"onmez and Rouselle \cite{sonmez2022random}.

\begin{theorem}\label{theo:transience}
	Let $d\geq 3$, let $J$ be an irreducible and resilient kernel, and let $\beta > \beta_c\left(J\right)$. Then the unique infinite component is almost surely a transient graph. In particular, if $J$ is a symmetric kernel such that
	\begin{equation}\label{eq:s cond}
		J(x) \simeq \|x\|^{-s}
	\end{equation}
	for some $s> d$, the infinite percolation cluster is almost surely transient for $\beta > \beta_c(J)$.
\end{theorem}

Note that the restriction to $d\geq 3$ is necessary, as for $d\in \{1,2\}$ and kernels $J$ satisfying condition \eqref{eq:s cond} with $s\geq 2d$, the simple random walk on the long-range percolation cluster is recurrent, as proven in \cite{berger2002transience,baumler2023recurrence}.

\subsection{Varying short edges only}

In the previous literature, also a different model of long-range percolation was considered. Let $f : \Z^d \to \left[0,1\right)$ be a \textbf{symmetric} function, i.e., a function $f$ such that $f(x)$ is invariant under sign-changes and permutations of the coordinates of $x$. We define the edge $e=\{x,y\}$ to be open with probability $f(x-y)$ if $\|x-y\|>1$, and with probability $p \in \left[0,1\right]$  if $\|x-y\|=1$. We assume that all edges are independent of each other and write $\p_{p,f}$ for the resulting probability measure. Typically, we consider the function $f$ as fixed and vary the parameter $p$. The difference to the previous setup is that here, we vary the probability that short-range edges are open, whereas, in the previous setup, all probabilities $\p_{\beta,J} \left(\{x,y\} \text{ open}\right)$ changed when varying $\beta$, as long as $J(x-y) \in (0,\infty)$. As the construction of the measures $\p_{p,f}$ is monotone in $p$, we can define the critical value 
\begin{equation*}
	p_c(f) = \inf \left\{p \in \left[0,1\right] : \p_{p,f} \left( |K_\mz| = \infty \right) > 0 \right\}.
\end{equation*}
Note that $p_c(f) \geq 0$, where equality can hold, even if the function $f$ is integrable. Furthermore, for every function $f$ one has $p_c(f) \leq p_c^d \leq 1$, where $p_c^d$ denotes the critical value for nearest-neighbor percolation on $\Z^d$; also note that $p_c^d<1$ for $d \geq 2$. For $d=1$, and for functions $f$ for which $\liminf_{x\to \infty} f(x) \|x\|^{2} > 1$, Newman and Schulman proved that $p_c(f) < 1$ \cite{newman1986one}, whereas the condition $f(x) \leq (1+o(1)) \|x\|^{-2}$ implies that $p_c(f)=1$ in dimension $d=1$ \cite{aizenman1986discontinuity}. The setup of varying the short-range probabilities in long-range percolation was often considered in previous literature \cite{aizenman1986discontinuity, meester1996continuity, newman1986one}, particularly in the work about continuity of the critical point for long-range percolation with exponential decay by Meester and Steif \cite{meester1996continuity}. One natural question is now whether the results that we stated above also hold for a supercritical long-range percolation measure $\p_{p,f}$. The answer is yes, at least under a certain regularity condition. 

\begin{theorem}\label{theo:p1 larger}
	Let $f : \Z^d  \to \left[0,1\right)$ be a symmetric function so that
	\begin{equation}\label{eq:fnecessary}
		f(x) \simeq \|x\|^{-s}
	\end{equation}
	for some $s > d$.
	Then in the supercritical regime $(p>p_c(f))$, the same results as stated in Theorems \ref{theo:main}, \ref{theo:large clusters}, and \ref{theo:transience} hold for the measure $\p_{p,f}$. If $s>2d$, then also the shape theorem as stated in Theorem \ref{theo:chemical dist} holds.
\end{theorem}

In particular, Theorem \ref{theo:p1 larger} shows that the exponential decay (respectively the ``{\sl Condition C}") required in the paper by Meester and Steif \cite{meester1996continuity} can be relaxed to polynomial decay. Our main tool for proving the results of Theorem \ref{theo:p1 larger} is a strict inequality of critical points for different kernels.

\begin{proposition}\label{prop:strict ineq}
	Let $J$ be an integrable and symmetric kernel so that there exists constants $0<a<A<\infty$ such that
	\begin{equation}\label{eq:comparability}
		0 < aJ(x+e_i) \leq J(x) \leq A J(x+e_i)
	\end{equation}
	for all $i\in \{1,\ldots,d\}$ and $x \in \Z^d$ with $\|x\|$ large enough. Define the kernel $\overline{J}$ by 
	\begin{equation*}
	\overline{J}(x)= \begin{cases}
	J(x) +1 & \text{ if } \|x\| = 1\\
	J(x) & \text{ else }
	\end{cases}.
	\end{equation*}
	Then $\beta_c(\overline{J}) < \beta_c(J)$.
\end{proposition}

To prove this result, we use the well-known technique of enhancements developed by Aizenman and Grimmett \cite{aizenman1991strict}.

\subsection{Organization}

We start with the proof of Theorem \ref{theo:main} in section \ref{sec:2}. As said above, the proof of this result follows a similar approach as used by Grimmett and Marstrand \cite{grimmett1990supercritical} and Meester and Steif \cite{meester1996continuity}. After this, we prove the shape theorem for long-range percolation, Theorem \ref{theo:chemical dist}, in section \ref{sec:shape}. This is by far the most technical part of the paper. The main technical innovations are Propositions \ref{prop:stretched exp} and \ref{propo:zeta}. Given these two propositions, the proof of Theorem \ref{theo:chemical dist} follows from similar methods as in the nearest-neighbor case \cite{garet2004asymptotic}. Section \ref{sec:shape} is completely self-contained and the only reference to the earlier sections is the use of Theorem \ref{theo:main}. In section \ref{sec:applis}, we prove Theorems \ref{theo:locality}, \ref{theo:large clusters}, and \ref{theo:transience}. The proof of these results is, given Theorem \ref{theo:main}, relatively straightforward. Theorem \ref{theo:p1 larger} is proven in section \ref{sec:short edges}.

\subsection{Notation}\label{sec:notation}

We write $\|x\|$ for the $2$-norm of $x$. We write $B_m(x)$ for the ball of radius $m$ around $x$ in the $\infty$-norm, i.e., $B_m(x)=\{y\in \Z^d : \|x-y\|_\infty \leq m\}$. We write $\mz$ for the origin of $\Z^d$ and define the annulus of thickness $\delta n$ around $B_n(\mz)$ by $S_n^{(1+\delta)n} \coloneqq B_{(1+\delta)n}(\mz) \setminus B_{n}(\mz)$.

We use the notation $x \leftrightarrow y$ if there exists an open path from $x$ to $y$ and for a set $A \subset \Z^d$ we write $x \overset{A}{\longleftrightarrow} y$ if there exists a path from $x$ to $y$ that lies entirely within the set $A$. For $x\in \Z^d$, we write $K_x = \{y \in \Z^d : x \leftrightarrow y\}$ for the open cluster containing $x$. For a set $A \subset \Z^d$, we define $K_x(A) \coloneqq \{y \in \Z^d : x \overset{A}{\longleftrightarrow} y\}$ as the open set containing $x$ within $A$. Also note that $K_x(A) = \{y \in \Z^d : x \overset{A}{\longleftrightarrow} y\} = \{y \in A : x \overset{A}{\longleftrightarrow} y\} \subseteq A$. For a set $S \subset \Z^d$, we define
\begin{equation*}
	K_S = \bigcup_{x\in S} K_x \ \ \ \ \text{ and } \ \ \ \ K_S(A) = \bigcup_{x\in S} K_x(A) = \left\{y \in A : y \overset{A}{\longleftrightarrow} S \right\}
\end{equation*}
as the set of points that can be reached from $S$ (within $A$).
For a percolation environment $\omega \in \{0,1\}^E, x,y \in \Z^d$, and $A\subseteq \Z^d$, we write $x \overset{A}{\longleftrightarrow} y$ in $\omega$ if there exists a path $(x=x_0,\ldots,x_\ell=y)$ such that $x_0,\ldots,x_{\ell} \in A$ and $\omega(\{x_{i-1},x_i\})=1$ for all $i\in \{1,\ldots, \ell\}$. We write $K_x(A;\omega)$ for the set
\begin{equation*}
	K_x(A;\omega) = \left\{y \in A: x \overset{A}{\longleftrightarrow} y \text{ in } \omega \right\} .
\end{equation*}
We say that a set $A \subset \Z^d$ is an \textbf{(open) $m$-pad} if $A=B_m(x)$ for some $x\in \Z^d$ and if $y \overset{A}{\longleftrightarrow} z$ for all $y,z\in A$. For an edge $e=\{x,y\}$, we write $|e|=|\{x,y\}|=\|x-y\|_\infty$ for the distance of its endpoints in the $\infty$-metric. For $n\in \N$ and $x,y \in \Z^d$, we write $x \overset{\leq n}{\longleftrightarrow} y$ if there exists an open path $\left(x=x_0,x_1,\ldots,x_\ell = y\right)$ with $\|x_{i}-x_{i-1}\|_\infty \leq n$ for all $i\in \{1,\ldots,\ell\}$.
For a kernel $J:\Z^d \to \left[0,\infty\right)$ and two (disjoint) sets $A,B\subset \Z^d$, we define
\begin{equation*}
	J(A,B) \coloneqq \sum_{x\in A} \sum_{y\in B} J(x-y).
\end{equation*}
In particular, this implies that for all $\beta \geq 0$ and all disjoint sets $A,B \subset \Z^d$
\begin{equation*}
	\p_{\beta,J} \left(A \nsim B\right) 
	= \prod_{x\in A} \prod_{y\in B} e^{-\beta J(x-y)}
	= e^{-\beta \sum_{x\in A} \sum_{y\in B} J(x-y)}
	= e^{-\beta  J(A,B)}.
\end{equation*}

For $x\in \Z^d$ and $A\subset \Z^d$, we also write $J(x,A)$ for $J(\{x\},A)$.
In many of our proofs, we will use {\sl sprinkling}. Let $E$ be the edge set of the complete graph upon $\Z^d$, i.e., $E=\left\{\{x,y\}: x,y\in \Z^d, x\neq y\right\}$. For an edge $e=\{x,y\} \in E$, we also write $J(e)\coloneqq J(x-y)$. We consider the percolation configuration as an element $\omega \in \{0,1\}^E$ and we regard an edge $e\in E$ as open if $\omega (e)=1$. 
To define sprinkling formally, we construct two collections of percolation environments $\left(\omega_\beta\right)_{\beta \geq 0}, (\omega_\beta^\prime )_{\beta \geq 0}$ as follows. Let $\left(U_e\right)_{e\in E}, \left(U_e^\prime\right)_{e\in E}$ be independent random variables that are uniformly distributed on the interval $\left[0,1\right]$. For all $e\in E$ and all $\beta \geq 0$, we define $\omega_\beta, \omega_\beta^\prime \in \{0,1\}^E$ by
\begin{align*}
	\omega_\beta(e)= \mathbbm{1} \left\{U_e \leq 1-\exp\left(-\beta J(e)\right)\right\} \ \ \ \text{ and } \ \ \  
	\omega_\beta^\prime(e)= \mathbbm{1} \left\{U_e^\prime \leq 1-\exp\left(-\beta J(e)\right)\right\}.
\end{align*}
From the definition it directly follows that $\p \left(\omega_\beta(e)=1\right) = 1-\exp \left(-\beta J(e)\right) = \p_{\beta,J} \left(e \text{ is open}\right)$ and the same equality also holds for $\omega^\prime_\beta$. Furthermore, this coupling is monotone in the sense that if $0 \leq \alpha \leq \beta $, then $\omega_{\alpha} \leq \omega_{\beta}$. For $\alpha,\beta \geq 0$, define $\omega = \omega_\beta \vee \omega_\alpha^\prime \in \{0,1\}^E$ by 
\begin{equation*}
	\omega(e) = \omega_\beta(e) \vee \omega_\alpha^\prime(e) = \max \left\{\omega_\beta(e) , \omega_\alpha^\prime(e)\right\}
\end{equation*}
for all $e\in E$. Then, by independence of $\omega_\beta$ and $ \omega_\alpha^\prime$,
\begin{align}
	\notag \p\left(\omega(e)=0\right) & = \p\left(\omega_\beta(e)=0 , \omega_\alpha^\prime(e)=0  \right)
	=
	\p\left(\omega_\beta(e)=0  \right) 
	\p\left( \omega_\alpha^\prime(e)=0  \right) = e^{-\alpha J(e)}  e^{-\beta J(e)}\\
	& \label{eq:maxim}
	= e^{-( \alpha +\beta ) J(e)} = 	\p\left(\omega_{\alpha + \beta}(e)=0\right)
\end{align}
which implies that $\omega = \omega_\beta \vee \omega_\alpha^\prime$ has the same distribution as $\omega_{\alpha+\beta}$. We will often consider this setup where we have a first percolation configuration $\omega_\beta$ and then sprinkle with the additional edges in $\omega_\alpha^\prime$ to obtain the new configuration $\omega = \omega_\beta \vee \omega_\alpha^\prime$.\\

For two probability measure $\p$ and $\p^\prime$ on $\{0,1\}^E$, we write $\p \gtrsim \p^\prime$ if $\p$ stochastically dominates $\p^\prime$, i.e., if $\p(A) \geq \p^\prime(A)$ for all increasing events $A \subseteq \{0,1\}^E$.

\section{The proof of Theorem \ref{theo:main}}\label{sec:2}

Throughout this section, we assume that $J$ is a symmetric and irreducible kernel satisfying the assumption of Theorem \ref{theo:main} $\left(J(x)=\mathcal{O}(\|x\|^{-2d})\right)$. We use this main assumption only at one point, namely in Lemma \ref{lem2}. Also the precise setup of the model is important for us. When increasing $\beta$, the probability $\p_{\beta}\left(\{x,y\} \text{ open}\right)$ increases for all edges $\{x,y\}$ with $J(x-y)>0$. This property is also used in the proof of Lemma \ref{lem4}. Many other arguments follow similar arguments as used by Grimmett and Marstrand \cite{grimmett1990supercritical}, respectively Meester and Steif \cite{meester1996continuity}.

\begin{lemma}\label{lem:totally connected}
	Let $J:\Z^d \to \left[0,\infty\right)$ be an irreducible and symmetric kernel. Then for all $m\in \N$ large enough, $\beta > 0$, and $x \in \{0,\ldots,m\}^d \eqqcolon A_m$
	\begin{equation*}
		\p_{\beta,J}\left(\mz \overset{A_m}{\longleftrightarrow} x\right) > 0
	\end{equation*}
	and 
	\begin{equation}\label{eq:tottaly connected}
		\p_{\beta,J}\left(x \overset{A_m}{\longleftrightarrow} y \text{ for all $x,y \in A_m$}\right) > 0.
	\end{equation}
\end{lemma}
\begin{proof}
	As the kernel $J$ is irreducible, there exists $N$ large enough so that the kernel $\tilde{J}$ defined by $\tilde{J}(v) = J(v) \mathbbm{1}_{\{\|v\|\leq N\}}$ is still irreducible. By definition, the kernel $\tilde{J}$ is also symmetric.
	Define $x_m \coloneqq \left(\lfloor m/2 \rfloor, \ldots, \lfloor m/2 \rfloor \right) \in \Z^d$. By symmetry of the kernel $\tilde{J}$, there exists $k\in \N$ such that for all $m\in \N$ large enough and $x\in A_m$ one has
	\begin{equation*}
		\p_{\beta,\tilde{J}} \left(B_k(x_m) \overset{A_m}{\longleftrightarrow} x \right) > 0 .
	\end{equation*}
	Indeed, this is possible as one can go from $x$ towards the direction of $x_m$, and with positive probability there thus exists an open path between $x$ and $B_k(x_m)$.
	Further, by the irreducibility of the kernel $\tilde{J}$, there is $K \geq k$ such that for all $y\in B_k(x_m)$
	\begin{equation*}
		\p_{\beta,\tilde{J}} \left( y \overset{B_K(x_m)}{\longleftrightarrow} x_m \right) > 0.
	\end{equation*}
	Let $m$ be large enough so that $B_K(x_m) \subset A_m$ (and thus also $B_k(x_m) \subset A_m$). Then for each $x\in A_m$, there exists with positive probability a path from $x$ to some $y \in B_{k}(x_m)$, and this path is entirely within $A_m$. Further, with positive probability, there is a path from $y$ to $x_m$ that is entirely within $B_K(x_m) \subset A_m$. So both paths have a positive probability of being open. Concatenating the two paths and using the FKG-inequality \cite[Section 2.2]{grimmett1999percolation} gives a path between $x$ and $x_m$ that is open with positive probability, i.e., $\p_{\beta,\tilde{J}} \left(x \overset{A_m}{\longleftrightarrow} x_m\right)>0$. Using the FKG-inequality once again, we see that for all $x\in A_m$, the probability that $x$ and $\mz$ are connected within $A_m$ is lower bounded by
	\begin{equation*}
		\p_{\beta,\tilde{J}} \left(\mz \overset{A_m}{\longleftrightarrow} x\right)
		\geq
		\p_{\beta,\tilde{J}} \left(\mz \overset{A_m}{\longleftrightarrow} x_m, x \overset{A_m}{\longleftrightarrow} x_m\right)
		\geq
		\p_{\beta,\tilde{J}} \left(\mz \overset{A_m}{\longleftrightarrow} x_m\right)
		\p_{\beta,\tilde{J}} \left(x \overset{A_m}{\longleftrightarrow} x_m\right)  > 0
	\end{equation*}
	and thus also $\p_{\beta,J} \left(\mz \overset{A_m}{\longleftrightarrow} x\right) \geq \p_{\beta,\tilde{J}} \left(\mz \overset{A_m}{\longleftrightarrow} x\right) > 0$.\\
	
	The proof of \eqref{eq:tottaly connected} follows by another application of the FKG-inequality:
	\begin{multline*}
		\p_{\beta,J}\left(x \overset{A_m}{\longleftrightarrow} y \text{ for all $x,y \in A_m$}\right) 
		= 
		\p_{\beta,J}\left(0 \overset{A_m}{\longleftrightarrow} x \text{ for all $x \in A_m$}\right)
		\\
		\geq
		\prod_{x \in A_m} \p_{\beta,J}\left(0 \overset{A_m}{\longleftrightarrow} x \right) > 0 .
	\end{multline*}
\end{proof}

In the next lemma, we prove that for two disjoint sets $A,B \subset \Z^d$ for which $J(A,B)$ is large, also the number of vertices $x\in B$ which are connected by an open edge to $A$, i.e., $\left|\left\{x\in B : x\sim A\right\}\right|$, is large with high probability.

\begin{lemma}\label{lem:new}
	Let $A,B \subset \Z^d$ with $A \cap B = \emptyset$. Define $\mu_\beta = \sum_{x\in \Z^d \setminus \{\mz\}} \beta J(x) = \beta J \left(\mz, \Z^d \setminus \{\mz\}\right)$. Then
	\begin{equation}\label{eq:new}
		\p_\beta \left( \big|\left\{x \in B : x \sim A\right\}\big| \leq \beta J(A,B) \frac{1\wedge \frac{1}{\mu_\beta}}{4} \right)
		\leq
		\frac{16 \left(\mu_\beta \vee 1\right) }{\beta J(A,B) }
	\end{equation}
\end{lemma}

\begin{proof}
	For each $x\in B$ we have
	\begin{align*}
		\p_\beta \left(x \sim A\right) &
		= 
		1- \prod_{y\in A} \p_\beta \left(x \nsim y\right) 
		= 
		1- \prod_{y\in A} \exp\left(-\beta J(x-y)\right) 
		=
		1- \exp \left( -\beta J(A,x) \right)
		\\
		&
		\geq
		\frac{(\beta J(A,x))\wedge 1}{2}
		=
		\beta J(A,x) \frac{1\wedge \frac{1}{\beta J(A,x)}}{2}
		\geq
		\beta J(A,x) \frac{1\wedge \frac{1}{\mu_\beta}}{2},
	\end{align*}
	where we used the elementary inequalities $1-e^{-s} \geq \frac{s \wedge 1}{2}$ and $\beta J(A,x) \leq \beta J(\mz,\Z^d \setminus \{\mz\}) = \mu_\beta$.
	Define the random variable $X \coloneqq \left|\left\{x \in B : x \sim A\right\}\right|$. Linearity of expectation implies that
	\begin{align}\label{eq:expectation X lower bound}
		\notag \E_\beta \left[X\right] & =
		\E_\beta \left[ \left|\left\{x \in B : x \sim A\right\}\right| \right] 
		= 
		\sum_{x\in B} \p_\beta \left(x \sim A\right)
		\geq
		\sum_{x\in B} \beta J(A,x) \frac{1\wedge \frac{1}{\mu_\beta}}{2}
		\\
		&
		=
		\beta
		J(A,B) \frac{1\wedge \frac{1}{\mu_\beta}}{2} .
	\end{align}
	A union bound over all $x\in B, y\in A$ implies that
	\begin{align}\label{eq:expectation bound}
		&\notag \E_\beta \left[X\right] = \E_\beta \left[ \left|\left\{x \in B : x \sim A\right\}\right| \right] 
		= 
		\sum_{x\in B} \p_\beta \left(x \sim A\right)
		\leq
		\sum_{x \in B} \sum_{y \in A} \p_\beta \left(x\sim y\right)
		\\
		&
		=
		\sum_{x \in B} \sum_{y \in A} \left(1-e^{-\beta J(x-y)}\right)
		\leq
		\sum_{x \in B} \sum_{y \in A} \beta J(x-y)
		=
		\beta J(A,B).
	\end{align}
	So we see that the expectation of $X= \left|\left\{x \in B : x \sim A\right\}\right|$ is of order $J(A,B)$. In order to say something about the typical value of $X = \left|\left\{x \in B : x \sim A\right\}\right|$, we calculate its variance. Note that the events of the form $\left\{x\sim A\right\}_{x\in B}$ are independent. Thus
	\begin{align*}
		\text{Var}\left(X\right) &
		=
		\sum_{x\in B} \text{Var}\left(\mathbbm{1}_{\{x\sim A\}}\right)
		=
		\sum_{x\in B} \left(\p_\beta \left(x \sim A\right) - \p_\beta \left(x \sim A\right)^2 \right)
		\\
		&
		\leq
		\sum_{x\in B} \p_\beta \left(x \sim A\right)
		=
		\E_\beta \left[ \left|\left\{x \in B : x \sim A\right\}\right| \right] 
		\overset{\eqref{eq:expectation bound}}{\leq} \beta J(A,B).
	\end{align*}
	Using Chebyshev's inequality, we see that
	\begin{align*}
		& \p_\beta \left( \left|\left\{x \in B : x \sim A\right\}\right| \leq \beta J(A,B) \frac{1\wedge \frac{1}{\mu_\beta}}{4} \right) 
		\overset{\eqref{eq:expectation X lower bound}}{\leq}
		\p_\beta \left( X - \E_\beta \left[X\right] \leq - \beta J(A,B) \frac{1\wedge \frac{1}{\mu_\beta}}{4}  \right)
		\\
		&
		\leq
		\frac{\text{Var}(X)}{\left(\beta J(A,B) \frac{1\wedge \frac{1}{\mu_\beta}}{4}\right)^2}
		\leq
		\frac{\beta J(A,B)}{\left(\beta J(A,B) \frac{1\wedge \frac{1}{\mu_\beta}}{4}\right)^2}
		=
		\frac{16 \left(\mu_\beta \vee 1\right) }{\beta J(A,B) } ,
	\end{align*}
	which finishes the proof.
\end{proof}

\begin{lemma}\label{lem1}
	Let $J:\Z^d \to \left[0,\infty\right)$ be a kernel and let $\beta > \beta_c(J)$. Let $\eps>0$ and let $m\in\N$ be such that $\p_{\beta}\left( B_m(\mz) \leftrightarrow \infty \right) > 1-\eps$. Then for all $L\in \N$ there exists $N\in \N$ such that for all $n\geq N$ and all sets $R$ with $B_m(\mz) \subseteq R \subseteq B_n(\mz)$
	\begin{equation*}
		\p_\beta \left(  J(K_R(B_n(\mz)), B_n(\mz)^c) > L \right) > 1-2\eps
	\end{equation*}
	where the set $K_R(B_n(\mz))= \bigcup_{x\in R} K_x(B_n(\mz))$ is the set of points that can be reached from $R$ within $B_n(\mz)$.
\end{lemma}

\begin{proof}
	Note that for all sets $R\supseteq B_m(\mz)$ we have that $K_R(B_n(\mz)) \supseteq K_{B_m(\mz)}(B_n(\mz))$, so it suffices to show the claim for $R=B_m(\mz)$. We set $R=B_m(\mz)$ for the rest of the proof. For the proof itself, we use a contraposition. So assume that there are infinitely many $n\in \N$ such that
	\begin{equation*}
		\p_\beta \left(  J(K_R(B_n(\mz)), B_n(\mz)^c) > L \right) \leq 1-2\eps
	\end{equation*}
	or equivalently
	\begin{equation*}
		\p_\beta \left(  J(K_R(B_n(\mz)), B_n(\mz)^c) \leq L \right) \geq 2\eps.
	\end{equation*}
	Then we also get by Fatou's Lemma that
	\begin{align*}
		&\p_\beta \left(  J(K_R(B_n(\mz)), B_n(\mz)^c) \leq L \text{ for infinitely many $n\in \N$} \right) \\
		& \hspace{5cm} \geq \limsup_{n \to \infty}  \p_\beta \left(  J(K_R(B_n(\mz)), B_n(\mz)^c) \leq L \right) \geq 2\eps.
	\end{align*}
	Together with $\p_\beta \left( R \leftrightarrow \infty\right) > 1-\eps$ this implies that 
	\begin{equation*}
		\p_\beta \left( R \leftrightarrow \infty \text{ and } J(K_R(B_n(\mz)), B_n(\mz)^c) \leq L \text{ for infinitely many $n\in \N$} \right) \geq \eps
	\end{equation*}
	which is a contradiction, as this probability needs to be $0$, see for example \cite[Lemma 2.6]{meester1996continuity}.
\end{proof}

\begin{definition}
	For $\delta > 0$ and $m,n \in \N$ we define the set
	\begin{equation*}
		P_{m,n}^\delta = \left\{ x \in B_{(1+\delta) n}(\mz) \setminus B_n(\mz) : x \text{ is in an open $m$-pad } A \subset B_{(1+\delta) n}(\mz) \setminus B_n(\mz)\right\}
	\end{equation*}
	as the union of all open $m$-pads within $B_{(1+\delta) n}(\mz) \setminus B_n(\mz)$. 
\end{definition}

The next lemma is the key lemma that uses the assumption on the kernel $J$ that $J(x)=\mathcal{O}(\|x\|^{-2d})$. A similar version was proven by Meester and Steif \cite[Lemma A]{meester1996continuity}.

\begin{lemma}\label{lem2}
	Let $J$ be an irreducible kernel so that $J(x)=\mathcal{O}(\|x\|^{-2d})$ and let $\beta > \beta_c(J)$. Let $\eps>0$ and let $m\in\N$ be large enough such that $\p_{\beta}\left( B_m(\mz) \leftrightarrow \infty \right) > 1-\eps$ and such that the results of Lemma \ref{lem:totally connected} hold for the box $\{0,\ldots,2m+1\}^d$. Then there exists $N\in \N$ such that for all $n\geq N$ and all sets $B_m(\mz) \subseteq R \subseteq B_n(\mz)$
	\begin{equation*}
		\p_\beta \left( K_R(B_n(\mz)) \sim P_{m,n}^\delta \right) > 1-3\eps.
	\end{equation*}
\end{lemma}

\begin{proof}
	As in the proof of Lemma \ref{lem1}, it suffices to prove this result for $R=B_m(\mz)$.
	The important observation here is that 
	\begin{align*}
		J(K_R(B_n(\mz)), S_n^{(1+\delta)n} ) & =  
		J(K_R(B_n(\mz)), B_{n}(\mz)^c ) - J(K_R(B_n(\mz)), B_{(1+\delta)n}(\mz)^c )\\
		& \geq 
		J(K_R(B_n(\mz)), B_{n}(\mz)^c ) - J(B_n(\mz), B_{(1+\delta)n}(\mz)^c )
	\end{align*}
	and that the quantity $J(B_n(\mz), B_{(1+\delta)n}(\mz)^c ) = \sum_{x \in B_n(\mz)} \sum_{y \notin B_{(1+\delta)n}(\mz)} J(x-y)$ is uniformly bounded over all $n\in \N$, as $J(x-y)=\mathcal{O}(\|x-y\|^{-2d})$. Note that this is the essential step (and the only time in the proof of Theorem \ref{theo:main}) where we use the requirement that $J(x)=\mathcal{O}(\|x\|^{-2d})$. In particular, for $L$ large enough we have that
	\begin{equation*}
		\text{ if } J(K_R(B_n(\mz)), B_{n}(\mz)^c ) > L, \text{ then } J(K_R(B_n(\mz)), S_n^{(1+\delta)n} ) > \frac{L}{2}
	\end{equation*}
	and thus, by Lemma \ref{lem1}, we get that for all large enough $L > 0$ one has for all large enough $n$ that
	\begin{equation*}
		\p_\beta \left( J(K_R(B_n(\mz)), S_n^{(1+\delta)n} ) > \frac{L}{2} \right) \geq 1-2\eps.
	\end{equation*}
	Next, we argue that there exists $L< \infty$ such that 
	\begin{equation}\label{eq:lem2toshow}
		\p_{\beta} \left( K_R(B_n(\mz)) \sim P_{m,n}^\delta \ \Big| \   J(K_R(B_n(\mz)), S_n^{(1+\delta)n} ) > \frac{L}{2} \right) > 1-\eps,
	\end{equation}
	which then implies 
	\begin{align*}
		&\p_{\beta} \left( K_R(B_n(\mz)) \sim P_{m,n}^\delta  \right)\\
		&
		\geq
		\p_{\beta} \left( K_R(B_n(\mz)) \sim P_{m,n}^\delta \ \Big| \  J(K_R(B_n(\mz)), S_n^{(1+\delta)n} ) > \frac{L}{2} \right) \p_{\beta} \left( J(K_R(B_n(\mz)), S_n^{(1+\delta)n} ) > \frac{L}{2} \right)\\
		& \geq (1-\eps)(1-2\eps) \geq 1-3\eps .
	\end{align*}
	So we are left to show that \eqref{eq:lem2toshow} holds for $L$ large enough. 
	Conditioned on the set $K_R(B_n(\mz))$, the edges $\{x,y\}$ with $x \in K_R(B_n(\mz))$ and $y \in S_n^{(1+\delta)n}$ are still open with probability $1-e^{-\beta J(x-y)}$. So by Lemma \ref{lem:new} applied with $A=K_R(B_n(\mz)), B = S_n^{(1+\delta)n}$, we get for $L$ large enough that 
	\begin{align}\label{eq:mpad to}
		\notag & \p_{\beta} \left( \left|\left\{x \in S_n^{(1+\delta)n} : x \sim K_R(B_n(\mz)) \right\}\right| > \beta \frac{L}{2} \frac{1\wedge \frac{1}{\mu_\beta}}{4} \ \Big| \  J(K_R(B_n(\mz), S_n^{(1+\delta)n}) > \frac{L}{2} \right) 
		\\
		&
		\geq 1-\frac{16 \left(\mu_\beta \vee 1\right) }{\beta \frac{L}{2} } \geq 1-\frac{\eps }{2}.
	\end{align}
	For $n$ large enough, for each $x\in S_n^{(1+\delta)n}$ there exist $u \in S_n^{(1+\delta)n}$ with $x \in B_m(u) \subset S_n^{(1+\delta)n}$. Thus,
	\begin{multline*}
		\p_\beta\left(x \in P_{m,n}^{\delta}\right) \geq \p_\beta \left(a \overset{B_m(u)}{\longleftrightarrow} b \text{ for all } a,b \in B_m(u) \right) 
		\\
		= 
		\p_\beta \left(a \overset{\{0,\ldots,2m+1\}^d}{\longleftrightarrow} b \text{ for all } a,b \in \{0,\ldots,2m+1\}^d \right) \overset{\eqref{eq:tottaly connected}}{\geq} c_m > 0,
	\end{multline*}
	where we used Lemma \ref{lem:totally connected} for the second-to-last inequality, and where $c_m$ is a positive constant. For points $x,y \in S_n^{(1+\delta)n}$ with $\|x-y\|_\infty \geq 5m$ it is independent whether they are elements of $P_{m,n}^\delta$. Using this independence one sees that
	\begin{equation*}
		 \p_{\beta} \left( K_R(B_n(\mz)) \sim P_{m,n}^\delta  \ \Big| \ \left|\left\{x \in S_n^{(1+\delta)n} : x \sim K_R(B_n(\mz)) \right\}\right| > \beta \frac{L}{2} \frac{1\wedge \frac{1}{\mu_\beta}}{4}  \right) \geq 1-\frac{\eps}{2}
	\end{equation*}
	for $L$ large enough. Together with \eqref{eq:mpad to} this implies \eqref{eq:lem2toshow} and thus finishes the proof.
\end{proof}

\begin{lemma}\label{lem3}
	Let $J:\Z^d \to \left[0,\infty\right)$ be an irreducible kernel with $J(x)=\mathcal{O}(\|x\|^{-2d})$ and let $\beta > \beta_c(J)$. For all $\eps,\delta,K>0$ there exist $m,N\in \N$ such that for all $n\geq N$ there exists $y = y(n) \in \partial B_n(\mz) = \{x:\|x\|_\infty = n\}$ such that for all sets $B_m(\mz) \subseteq R \subseteq B_n(\mz)$
	\begin{equation}\label{eq:lem4 cond1}
		\p_\beta \left( K_R(B_n(\mz)) \sim P_{m,n}^{\delta,y} \right) > 1-\eps
	\end{equation}
	and
	\begin{equation}\label{eq:lem4 cond2}
		\p_\beta \left( J (R,W_y^n) > K \right) > 1-\eps,
	\end{equation}
	where the sets $P_{m,n}^{\delta,y}$ and $W_y^n$ are defined by
	\begin{equation*}
		P_{m,n}^{\delta,y} =\left\{x \in P_{m,n}^{\delta} : x \text{ is contained in an $m$-pad $A\subset S_n^{(1+\delta)n}$ such that $A \subset B_{\delta n}(y)$ } \right\}
	\end{equation*}
	and
	\begin{align*}
		W_y^n = P_{m,n}^{\delta,y} \cup \bigcup_{x \in B_n(\mz) : x \sim  P_{m,n}^{\delta,y}} K_x\left(B_{n}(\mz)\setminus R\right) .
	\end{align*}
	See Figure \ref{fig:m pad at y} for a picture of this setup.
\end{lemma}

\begin{figure}
	\begin{center}
		\begin{tikzpicture}[scale=0.03]

			\draw[gray] (-100,-100) rectangle (100,100);
			\draw[gray] (-130,-130) rectangle (130,130);
			
			\fill[lightblue] (-9,-9) rectangle (9,9);
			
			\draw[pattern=north west lines, pattern color=darkgreen, thick] (100,25) rectangle (130,85);
			\draw[ color=darkgreen, thick] (100,25) rectangle (130,85);

			\vertex[fill ,minimum size = 3 pt, label={[left] $y$},  ]  at (100,55) {};

			\fill[lightblue] (104,57) rectangle (104+18,57+18);

			\draw[->, lightgray] (-135,0) -- (143,0) node[right] {$e_1$};
			\draw[->, color=lightgray] (0,-135) -- (0,143) node[above] {$e_2$};

			\draw[black, thick] (7,4) to[out=60, in=170] (35,18);
			\draw[black, thick] (35,18) to[out=60, in=170] (44,20);
			\draw[black, thick] (44,20) to[out=60, in=230] (60,50);
			\draw[black, thick] (60,50) to[out=60, in=170] (80,70);
			\draw[black, thick] (80,70) to[out=60, in=120] (110,70);
			
			\vertex[fill ,minimum size = 2 pt, label={},  ]  at (7,4) {};
			\vertex[fill ,minimum size = 2 pt, label={},  ]  at (35,18) {};
			\vertex[fill ,minimum size = 2 pt, label={},  ]  at (44,20) {};
			\vertex[fill ,minimum size = 2 pt, label={},  ]  at (60,50) {};
			\vertex[fill ,minimum size = 2 pt, label={},  ]  at (80,70) {};
			\vertex[fill ,minimum size = 2 pt, label={},  ]  at (110,70) {};
			
			\draw[decorate, decoration={brace, amplitude=6pt}] (-155,-130) -- node[left=0.2cm] {\( \delta n \)} (-155,-100);
			\draw[decorate, decoration={brace, amplitude=6pt}] (-155,-100) -- node[left=0.2cm] {\( 2n \)} (-155,100);
			\draw[decorate, decoration={brace, amplitude=6pt}] (-155,100) -- node[left=0.2cm] {\( \delta n \)} (-155,130);
			
			\draw[decorate, decoration={brace, amplitude=6pt}] (140,85) -- node[right=0.2cm] {\( 2\delta n \)} (140,25);
		\end{tikzpicture}
		
		\parbox{14cm}{ \caption{ An illustration of the statement of Lemma \ref{lem3}: The inner blue square $(R)$ is connected by an open path (the black edges) to an open $m$-pad (the outer blue square) in $S_n^{(1+\delta)n} \cap B_{\delta n}(y)$ (the green hatched area).}\label{fig:m pad at y}}
	\end{center}
\end{figure}

\begin{proof}
	We start with the proof of \eqref{eq:lem4 cond1}. As in the preceding lemmas, it suffices to show the claim for $R=B_m(\mz)$.
	For fixed $\delta > 0$ we can choose a family of sets $\left(\mathcal{Y}_n\right)_{n\in \N}$ such that $\mathcal{Y}_n \subset \partial B_{n}(\mz)$ for all $n\in \N$, $Y\coloneqq \sup_{n\in \N} |\mathcal{Y}_n|< \infty$, and such that for all $m < \tfrac{\delta}{3} n$, if $A\subset S_n^{(1+\delta)n}$ is an open $m$-pad, then $A \subset S_n^{(1+\delta)n} \cap B_{\delta n}(y)$ for some $y\in \mathcal{Y}_n$. The set $\mathcal{Y}_n$ can be constructed by taking points $y \in \partial B_n(\mz)$ that have an Euclidean distance of order $\delta n$. These properties of $\mathcal{Y}_n$ imply that 
	\begin{equation*}
		 \bigcup_{y\in \mathcal{Y}_n} \left\{K_R(B_n(\mz)) \sim P_{m,n}^{\delta,y}\right\}  = \left\{K_R(B_n(\mz)) \sim P_{m,n}^{\delta}\right\} .
	\end{equation*}
	As all the events $\left\{K_R(B_n(\mz)) \sim P_{m,n}^{\delta,y}\right\}$ are increasing, we get by the FKG-inequality \cite[Section 2.2]{grimmett1999percolation} (respectively the ``square-root-trick") that 
	\begin{equation*}
		\max_{y\in \mathcal{Y}_n}\p_\beta \left( K_R(B_n(\mz)) \sim P_{m,n}^{\delta,y} \right) \geq 1- \left( 1- \p_\beta \left( K_R(B_n(\mz)) \sim P_{m,n}^{\delta} \right) \right)^{1/Y} .
	\end{equation*}
	By Lemma \ref{lem2}, the expression on the right-hand side of this inequality can be arbitrarily close to $1$ for suitable choice of $m,N$ and all $n\geq N$, as $Y< \infty$.  Thus also the expression on the left-hand side of this inequality $\left(\p_\beta \left( K_R(B_n(\mz)) \sim P_{m,n}^{\delta,y} \right)\right)$ will be arbitrarily close to $1$ for an appropriate choice of $m,N$ and all $n\geq N$. In particular, for fixed $\eps, \delta, K > 0$ there exist $m,N \in \N$ such that for all $n\geq N$ there exists $y\in \mathcal{Y}_n \subset \partial B_n(\mz)$ such that
	\begin{equation}\label{eq:lem4 expo insertion}
		\p_\beta \left( K_R(B_n(\mz)) \nsim P_{m,n}^{\delta,y} \right) \leq e^{-\beta K} \eps.
	\end{equation}
	
	As $e^{-\beta K} < 1$, this directly implies \eqref{eq:lem4 cond1}.
	Next, let us go to the proof of \eqref{eq:lem4 cond2}. Let $R\subseteq B_n(\mz)$ with $R\supseteq B_m(\mz)$.
	Conditioned on the event $\left\{J (R,W_y^n) \leq K\right\}$, there is a probability of at least $e^{-\beta K}$ that all edges between $R$ and $W_y^n$ are closed. However, if all edges between $R$ and $W_y^n$ are closed, then $K_R(B_n(\mz)) \nsim P_{m,n}^{\delta,y}$, and thus we get that
	\begin{align*}
		\p_\beta \left( K_R(B_n(\mz)) \nsim P_{m,n}^{\delta,y} \right) 
		& \geq 
		\p_\beta \left( J (R,W_y^n) \leq K \right) \p_\beta \left( K_R(B_n(\mz)) \nsim P_{m,n}^{\delta,y} \big|  J (R,W_y^n) \leq K \right)
		\\
		&
		\geq \p_\beta \left( J (R,W_y^n) \leq K \right)   e^{-\beta K} ,
	\end{align*}
	which directly implies that $\p_\beta \left( J (R,W_y^n) \leq K \right) \leq \eps$, by \eqref{eq:lem4 expo insertion}.
\end{proof}

In the previous proofs, we did not use the assumption that $\beta > \beta_c$. Instead, we only used that $\theta(\beta) = \p_\beta \left( |K_\mz|=\infty \right) > 0$ (which is conjectured to be equivalent to $\beta > \beta_c$ in dimensions $d\geq 2$). However, in the next few lemmas, we will use that $\beta > \beta_c$, as the proofs of these results require sprinkling as introduced in section \ref{sec:notation}.

\begin{lemma}\label{lem4}
	Let $J:\Z^d \to \left[0,\infty\right)$ be a symmetric and irreducible kernel with $J(x)=\mathcal{O}(\|x\|^{-2d})$ and let $\beta > \beta_c(J)$. For all $\eps^\prime,\delta >0$, there exist $m,N \in \N$ such that for all even $n \geq N$ 
	\begin{equation}\label{eq:lem4 eq1}   
		\p_\beta \left( B_m(\mz) \overset{B_{(1+\delta)n}(\mz)}{\longleftrightarrow} F_{m}^{\delta}(n e_i) \right) > 1-\eps^\prime
	\end{equation}  
	where $F_{m}^{\delta}(n e_i)$ is the set 
	\begin{equation*}
		F_{m}^{\delta}(n e_i) = \left\{x \in B_{\delta n}(n e_i) : x \text{ is contained in an open $m$-pad } \tilde{A} \subset B_{\delta n}(n e_i) \right\}.
	\end{equation*}
	Further, let $\eps^\prime,\delta, K>0$. Then there exists $m,N\in \N$ such that for all even $n\geq N$ and all sets $B_m(\mz) \subseteq R \subseteq B_{(1+\delta) n}(\mz) \setminus B_{\delta n}(n e_i)$
	\begin{equation}\label{eq:lem4 eq2}
		\p_\beta \left( J \left(R,\widetilde{W}\right) > K \right) > 1-\eps^\prime
	\end{equation}
	where the set $\widetilde{W}$ is defined by
	\begin{align*}
		\widetilde{W} = F_{m}^{\delta}(n e_i) \cup \bigcup_{x \in F_{m}^{\delta}(n e_i)}  K_x\left(B_{(1+\delta) n}(\mz)\setminus R\right) .
	\end{align*}   
\end{lemma}

\begin{figure}
	\begin{center}
		\begin{tikzpicture}[x=0.023cm, y=0.023cm]
			
			\draw[gray] (-100,-100) rectangle (100,100);
			\draw[gray] (-130,-130) rectangle (130,130);
			
			\fill[lightblue] (-9,-9) rectangle (9,9);
			
			\draw[pattern=north west lines, pattern color=darkgreen, thick] (100,25) rectangle (130,85);
			\draw[ color=darkgreen, thick] (100,25) rectangle (130,85);

			\vertex[fill ,minimum size = 3 pt, label={[left] $y$},  ]  at (100,55) {};

			\fill[lightblue] (104,57) rectangle (104+18,57+18);

			\draw[->, black] (-135,0) -- (263,0) node[right] {$e_1$};
			\draw[->, color=black] (0,-135) -- (0,203) node[above] {$e_2$};

			\draw[black, thick] (7,4) to[out=60, in=170] (35,18);
			\draw[black, thick] (35,18) to[out=60, in=170] (44,20);
			\draw[black, thick] (44,20) to[out=60, in=230] (60,50);
			\draw[black, thick] (60,50) to[out=60, in=170] (80,70);
			\draw[black, thick] (80,70) to[out=60, in=120] (110,70);
			
			\vertex[fill ,minimum size = 2 pt, label={},  ]  at (7,4) {};
			\vertex[fill ,minimum size = 2 pt, label={},  ]  at (35,18) {};
			\vertex[fill ,minimum size = 2 pt, label={},  ]  at (44,20) {};
			\vertex[fill ,minimum size = 2 pt, label={},  ]  at (60,50) {};
			\vertex[fill ,minimum size = 2 pt, label={},  ]  at (80,70) {};
			\vertex[fill ,minimum size = 2 pt, label={},  ]  at (110,70) {};

			\draw[gray] (-100+113,-100+65) rectangle (100+113,100+65);
			\draw[gray] (-130+113,-130+65) rectangle (130+113,130+65);
			
			\draw[pattern=north west lines, pattern color=darkgreen, thick] (100+113,65-25) rectangle (130+113,65-85);
			\draw[ color=darkgreen, thick] (100+113,65-25) rectangle (130+113,65-85);
			\vertex[fill ,minimum size = 3 pt, label={[left] $z+\tilde{y}$},  ]  at (100+113,65-55) {};
			
			\fill[lightblue] (104+113+6,65-57+14) rectangle (104+18+113+6,65-57-18+14);
			
			\draw[black, thick]
			(115,60) to[out=70,in=160] (170, 90) to[out=100, in=0] (150, 120) to[out=60, in=110] (190, 100) to (200, 50) to[out=60, in=120] (225, 50) to[out=280,in=100] (230,16);
			
			\vertex[fill ,minimum size = 2 pt, label={},  ]  at (115,60) {};
			\vertex[fill ,minimum size = 2 pt, label={},  ]  at (170,90) {};
			\vertex[fill ,minimum size = 2 pt, label={},  ]  at (150,120) {};
			\vertex[fill ,minimum size = 2 pt, label={},  ]  at (190,100) {};
			\vertex[fill ,minimum size = 2 pt, label={},  ]  at (200,50) {};
			\vertex[fill ,minimum size = 2 pt, label={},  ]  at (225,50) {};
			\vertex[fill ,minimum size = 2 pt, label={},  ]  at (230,16) {};

			\draw[decorate, decoration={brace, amplitude=6pt}] (-155,-130) -- node[left=0.2cm] {\( \delta k \)} (-155,-100);
			\draw[decorate, decoration={brace, amplitude=6pt}] (-155,-100) -- node[left=0.2cm] {\( 2k \)} (-155,100);
			\draw[decorate, decoration={brace, amplitude=6pt}] (-155,100) -- node[left=0.2cm] {\( \delta k \)} (-155,130);
		\end{tikzpicture}
		
		\parbox{14cm}{ \caption{ Connecting $m$-pads (the blue boxes) in the proof of Lemma \ref{lem4}. We first find a path from $B_m(\mz)$ (the blue box on the left side) to $A = B_m(z)$ (the middle blue box). From $A$, we find a path to an $m$-pad $\tilde{A} \subset F_m^{\delta}(ne_1)$ (in the picture, $\tilde{A}$ is the blue box on the right side). Concatenating these two paths gives a path from $B_m(\mz)$ to $F_m^{\delta}(ne_1)$.}\label{fig:m pad at e_1}}
	\end{center}
\end{figure}

\begin{proof}
	By the symmetry of the kernel $J$, it suffices to prove the result for $e_i=e_1$. We define 
	\begin{equation*}
		\tilde{\beta}=\frac{\beta+\beta_c}{2},  \eta=\frac{\beta-\beta_c}{2}, \text{ and }	\eps = \frac{\eps^\prime}{3}
	\end{equation*}
	and take $K$ large enough so that $1-e^{-\eta K} > 1-\eps$.
	Let $m,N\in \N$ be such that for all $k\geq N$ there exists $y\in \partial B_k(\mz)$ such that for all sets $B_m(\mz) \subseteq R \subseteq B_k(\mz)$
	\begin{align}\label{eq:ychoose}
		\p_{\tilde{\beta}} \left( K_R(B_k(\mz)) \sim P_{m,k}^{\delta,y} \right) > 1-\eps \text{ and } \p_{\tilde{\beta}} \left( J (R,W_y^k) > K \right) > 1-\eps .
	\end{align}
	Such $m,N$, and $y=y(k)$ exist by the results of Lemma \ref{lem3}.
	In the following, we work with the two independent percolation configurations $\omega_{\tilde{\beta}}$ and $\omega_\eta^\prime$ as introduced in section \ref{sec:notation}. Note that their union $\omega_{\tilde{\beta}} \vee \omega_\eta^\prime$ is distributed like $\omega_{\beta}$ -- a percolation configuration sampled by the measure $\p_{\beta}$. Unless otherwise specified, all connection events of the form $x\sim v$ or $A \leftrightarrow B$ will always refer to connections in the environment $\omega_{\tilde{\beta}}$.

	By the symmetry of the kernel $J$ we can assume that the point $y= (y_1,\ldots,y_d)\in \partial B_k(\mz)$ chosen in \eqref{eq:ychoose} satisfies $y_1 = k$. So in particular we get that $B_m(\mz) \overset{B_{(1+\delta)k}(\mz)}{\longleftrightarrow} A$ for some open $m$-pad $A\subset S_k^{(1+\delta)k} \cap B_{\delta k}(y)$ with probability at least $1-\eps$. Assume that such a set $A$ exists. Let $z\in A$ be such that $A=B_m(z)$. Define the set
	\begin{equation*}
		\tilde{R} = K_A\left( B_{(1+\delta)k}(\mz) \right) = \left\{x \in B_{(1+\delta)k}(\mz) : x\overset{B_{(1+\delta)k}(\mz)}{\longleftrightarrow} A \text{ in } \omega_{\tilde{\beta}}\right\}
	\end{equation*}
	as the open set containing $A$ within $B_{(1+\delta)k}(\mz)$ in the environment $\omega_{\tilde{\beta}}$.
	Define $\tilde{y} = (\tilde{y}_1,\ldots,\tilde{y}_d) \in \Z^d$ by $\tilde{y}_1=y_1$ and $\tilde{y}_i=-y_i$ for $i\geq 2$. Furthermore, we define the sets
	\begin{multline*}
		P = \Big\{ x \in z+ S_k^{(1+\delta)k} : x \text{ is contained} \\ \text{in an $m$-pad that is a subset of $z + \left(B_{\delta k}(\tilde{y}) \cap S_k^{(1+\delta)k} \right)$} \Big\}
	\end{multline*}
	and
	\begin{align*}
		W = 
		P \cup \bigcup_{x \in B_k(z): x \sim P} K_x\left( B_k(z) \setminus \tilde{R} \right).
	\end{align*}
	The set $\tilde{R}$ can be constructed by only revealing the information whether edges with at least one end in $\tilde{R}$ are open. Contrary to that, the set $W$ only depends on edges with both ends outside of $\tilde{R}$. Also note that the set $W$ is defined as the set $W_{\tilde{y}}^k$ in Lemma \ref{lem3}, up to a translation.
	Using the translation invariance of the model and Lemma \ref{lem3}, we get that $J\left(\tilde{R},W\right) > K$ with probability at least $1-\eps$, under the measure $\p_{\tilde{\beta}}$. Conditioning on $\{J\left(\tilde{R},W\right) > K\}$, there exists an $\omega_{\eta}^\prime$-open edge between $\tilde{R}$ and $W$ with probability at least $1-\exp\left(-\eta K\right) > 1-\eps$. If there exists such an open edge between $\tilde{R}$ and $W$ in the environment $\omega_\eta^\prime$, then there exists an open path in the environment $\omega_{\tilde{\beta}} \vee \omega_\eta^\prime$ from $B_m(\mz)$ to $P$ and this open path is entirely in the set $B_{(1+\delta)k}(\mz) \cup B_{(1+\delta)k}(z) \subset B_{(2+2\delta)k}(\mz) $. The three relevant events for this $\Big($$\big\{B_m(\mz) \leftrightarrow A$ for some $m$-pad $A \subset P_{m,k}^{\delta,y} \big\}$, $\big\{J(\tilde{R},W) > K\big\}$, and $\big\{$there exists an open edge between $W$ and $\tilde{R}$ in the environment $\omega_{\eta}^\prime \big\}$$\Big)$ all have a conditional probability of at least $(1-\eps)$. Thus, all three of the events occur simultaneously with probability at least $(1-\eps)^3 \geq 1-3\eps$.\\

	\noindent
	Let $B_m(\tilde{z})=\tilde{A} \subset P$ be an open $m$-pad. We know that
	\begin{align*}
		A & = B_m(z) \subset S_k^{(1+\delta)k} \cap B_{\delta k}(y), \\
		\tilde{A} & = B_m(\tilde{z}) \subset z + \left(B_{\delta k}(\tilde{y}) \cap  S_k^{(1+\delta)k}\right) .
	\end{align*}
	As $y+\tilde{y}=2k e_1$ and $\|z-y\|_\infty \leq \delta k$, we thus get that
	\begin{equation*}
		\tilde{A} \subset z + B_{\delta k}(\tilde{y})
		= z-y +y  + B_{\delta k}(\tilde{y})
		= z-y   + B_{\delta k}(2 k e_1)
		\subset
		B_{2\delta k}(2 k e_1),
	\end{equation*}
	which also implies that $\tilde{A} \subset B_{(1+\delta)2 k} (\mz)$, and thus $\tilde{A} \subset F_m^{\delta} (2k e_1)$. See Figure \ref{fig:m pad at e_1} for the relative positions of $A,\tilde{A},y,\tilde{y}$, and $z$. So in total, we see that
	\begin{equation*}   
		\p_\beta \left( B_m(\mz) \overset{B_{(1+\delta)2k}(\mz)}{\longleftrightarrow}  F_{m}^{\delta}(2k e_1) \right) \geq (1-\eps)^3 \geq 1-3\eps = 1-\eps^\prime
	\end{equation*} 
	which finishes the proof of \eqref{eq:lem4 eq1} for $n=2k$. Given \eqref{eq:lem4 eq1}, the proof of \eqref{eq:lem4 eq2} works the same way as the proof of \eqref{eq:lem4 cond2} and we omit it.
\end{proof}


In the proof of Lemma \ref{lem4}, we first found an open $m$-pad $A\subset B_{\delta k}(y)$, and then we found an open $m$-pad $\tilde{A}\subset B_{2\delta k}(y+\tilde{y})$ such that $B_m(\mz) \leftrightarrow A \leftrightarrow \tilde{A}$. This construction of making connections by concatenating connections between $m$-pads will be extremely useful for the proof of Theorem \ref{theo:main}. In the same way as in the proof of Lemma \ref{lem4} one can prove the following result, using Lemma \ref{lem4}.

\begin{figure}
	\begin{center}
		\begin{tikzpicture}[scale=0.2]

			\draw[gray] (-10,-10) rectangle (40,10);
			
			\fill[lightblue] (-1,-1) rectangle (1,1);

			\draw[pattern=north west lines, pattern color=darkgreen, thick] (27,-3) rectangle (33,3);
			\draw[ color=darkgreen, thick] (27,-3) rectangle (33,3);
			
			\fill[lightblue] (30-1-1.3,-1+0.7) rectangle (30+1-1.3,1+0.7);
			
			\draw[->, lightgray] (-11,0) -- (43,0) node[right] {$e_1$};
			\draw[->, color=lightgray] (0,-11) -- (0,16.5) node[right] {$e_2$};

			\draw[decorate, decoration={brace, amplitude=6pt}] (-10,13) -- node[above=0.2cm] {\( n \)} (0,13);
			\draw[decorate, decoration={brace, amplitude=6pt}] (0,13) -- node[above=0.2cm] {\( 3n \)} (30,13);
			\draw[decorate, decoration={brace, amplitude=6pt}] (30,13) -- node[above=0.2cm] {\(  n \)} (40,13);
			
			\draw[decorate, decoration={brace, amplitude=6pt}] (30+3,-4) -- node[below=0.2cm] {\(  \delta n \)} (30-3,-4);

			\draw[black, thick] (0.7,0.3) to[out=10, in=130] (6,-3) to[out=80, in=250] (8,3) to[out=20, in=180] (12,4) to[out=0, in=120] (15,3) to[out=0, in=120] (20,-2) to[out=0, in=180] (22,-2) to[out=70, in=180] (26,3) to[out=0, in=150] (28,1);
			
			\vertex[fill ,minimum size = 2 pt, label={},  ]  at (0.7,0.3) {};
			\vertex[fill ,minimum size = 2 pt, label={},  ]  at (6,-3) {};
			\vertex[fill ,minimum size = 2 pt, label={},  ]  at (8,3) {};
			\vertex[fill ,minimum size = 2 pt, label={},  ]  at (12,4) {};
			\vertex[fill ,minimum size = 2 pt, label={},  ]  at (15,3) {};
			\vertex[fill ,minimum size = 2 pt, label={},  ]  at (20,-2) {};
			\vertex[fill ,minimum size = 2 pt, label={},  ]  at (22,-2) {};
			\vertex[fill ,minimum size = 2 pt, label={},  ]  at (26,3) {};
			\vertex[fill ,minimum size = 2 pt, label={},  ]  at (28,1) {};
			
		\end{tikzpicture}
		
		\parbox{14cm}{ \caption{ An illustration of the statement of Corollary \ref{coro:conn}: $B_m(\mz)$ (the left blue square) is connected by an open path (the black edges) to an open $m$-pad (right blue square) in the target area (green hatched). The path does not leave the big rectangle.}\label{fig:m pad steer 1}}
	\end{center}
\end{figure}

\begin{corollary}\label{coro:conn}
	Let $J:\Z^d \to \left[0,\infty\right)$ be a symmetric and irreducible kernel with $J(x)=\mathcal{O}(\|x\|^{-2d})$ and let $\beta > \beta_c(J)$. For all $\delta, \eps > 0$ there exist $m,N\in \N$ such that for all $n\geq N$ and $i\in \{1,\ldots,d\}$
	\begin{equation*}
		\p_\beta \left( B_m(\mz)  \overset{Z_n^i}{\longleftrightarrow} P_i  \right) > 1-\eps
	\end{equation*}
	where the sets $P^i, Z_n^i$ are defined by
	\begin{align*}
		P_i & = \left\{x : x \text{ is contained in an open $m$-pad $A\subset B_{\delta n} (3n e_i)$}\right\}\\
		Z_n^i & = \{-n,\ldots,n\}^{i-1} \times \{-n,\ldots,4n\} \times \{-n,\ldots,n\}^{d-i}.
	\end{align*}
\end{corollary}

See Figure \ref{fig:m pad steer 1} for a picture of the statement of Corollary \ref{coro:conn}.
In Corollary \ref{coro:conn} we consider boxes $(Z_n^i)$ that are no cubes. The reason why we do this is the following. Assume that we construct a path that starts at the open $m$-pad $B_m(u_1)$. From there, for $i,j\in \{1,\ldots,d\}$, we construct a path within $u_1 +Z_n^i$ to an open $m$-pad $B_m(u_2) \subset B_{\delta n}(u_1 + 3n e_i)$ and from there we construct an open path within $u_2+Z_n^j$ to an open $m$-pad $B_m(u_3) \subset B_{\delta n} (u_2 + 3n e_j)$, then the {\sl target regions} $B_{\delta n}(u_1 + 3n e_i), B_{\delta n}(u_2 + 3n e_j)$ are (at least for $\delta > 0$ sufficiently small) such that no information has been revealed about the edges in them so far.
Using this idea inductively implies the following result.

\begin{corollary}\label{coro:steer2}
	Let $J:\Z^d \to \left[0,\infty\right)$ be a symmetric and irreducible kernel with $J(x)=\mathcal{O}(\|x\|^{-2d})$ and let $\beta > \beta_c(J)$. For all $\eps > 0$ there exist $m,N \in \N$ such that for all $n\geq N$ the following result holds.
	For all $i\in \{1,\ldots,d\}$ and all sets $A=B_m(u) \subset B_n(\mz)$
	\begin{equation*}
		\p_\beta \left( A \overset{M_i}{\longleftrightarrow} P_i \right) > 1-\eps,
	\end{equation*}
	where the sets $M_i, P_i$ are defined by
	\begin{align*}
		& M_i = \left\{-3n,\ldots,3n\right\}^{i-1} \times \left\{-3n,\ldots,11n\right\} \times \left\{-3n,\ldots,3n\right\}^{d-i}  , \\
		& T_i =  \left\{-n,\ldots,n\right\}^{i-1} \times \left\{7n,\ldots,9n\right\} \times \left\{-n,\ldots,n\right\}^{d-i}, \\
		& P_i =  \left\{x \in T_i: x \text{ is contained in an open $m$-pad } B \subset T_i\right\}.
	\end{align*}
	See Figure \ref{fig:final steering} for the relative positions of these sets.
	Furthermore, let $\beta > \beta_c$ and $\eps, K > 0$. Then there exist $m,N \in \N$ such that for all $n\geq N$, all $i\in \{1,\ldots,d\}$, all sets $A=B_m(u) \subset B_n(\mz)$, and all sets $R$ with $A\subseteq R \subseteq M_i \setminus T_i$
	\begin{equation*}
		\p_\beta \left( J(R,W) > K \right) > 1-\eps
	\end{equation*}
	where the set $W$ is defined by
	\begin{equation*}
		W =  K_{P_i} \left( M_i \setminus R \right) = \left\{x \in M_i \setminus R : x \overset{M_i \setminus R}{\longleftrightarrow} P_i \right\} .
	\end{equation*}
\end{corollary}

\begin{figure}
	\begin{center}
		\begin{tikzpicture}[scale=0.08]

			\draw[gray] (-30,-30) rectangle (110,-10);
			\draw[gray] (-30,-10) rectangle (110,10);
			\draw[gray] (-30,10) rectangle (110,30);
			\draw[gray] (-30,-30) rectangle (-10,30);
			\draw[gray] (-10,-30) rectangle (10,30);
			\draw[gray] (10,-30) rectangle (30,30);
			\draw[gray] (30,-30) rectangle (50,30);
			\draw[gray] (50,-30) rectangle (70,30);
			\draw[gray] (70,-30) rectangle (90,30);
			\draw[gray] (90,-30) rectangle (110,30);
			\draw[black, very thick] (-30,-30) rectangle (110,30);

			\draw[draw=none, pattern=north west lines, pattern color=darkgreen, thick] (70,-10) rectangle (90,10);
			
			\draw[draw=none, pattern=north east lines, pattern color=orange, thick] (-10,-10) rectangle (10,10);

			\draw[decorate, decoration={brace, amplitude=6pt}] (-30,32) -- node[above=0.2cm] {\( 2n \)} (-10,32);
			\draw[decorate, decoration={brace, amplitude=6pt}] (-32,10) -- node[left=0.2cm] {\( 2n \)} (-32,30);

			\fill[lightblue] (-5,-8) rectangle (3,0);
			\fill[lightblue] (72,-2) rectangle (80,6);
			
			\draw[black, thick] (-1,-4) to[out=90, in=170] (8,8) to[out=30,in=150] (20,8) to[out =70, in= 170] (40, 20) to[out=320,in=90] (45,0) to[out=30,in=174] (55,3) to[out=10,in=170] (75,4);
			
			\vertex[fill ,minimum size = 3 pt, label={},  ]  at (-1,-4) {};
			\vertex[fill ,minimum size = 3 pt, label={},  ]  at (8,8) {};
			\vertex[fill ,minimum size = 3 pt, label={},  ]  at (20,8) {};
			\vertex[fill ,minimum size = 3 pt, label={},  ]  at (40,20) {};
			\vertex[fill ,minimum size = 3 pt, label={},  ]  at (45,0) {};
			\vertex[fill ,minimum size = 3 pt, label={},  ]  at (55,3) {};
			\vertex[fill ,minimum size = 3 pt, label={},  ]  at (75,4) {};

		\end{tikzpicture}
		
		\parbox{14cm}{ \caption{ An illustration of the statement of Corollary \ref{coro:steer2} in dimension $d=2$: For every set $B_m(u) \subset B_n(\mz)$ (the blue square on the left side as a subset of the orange hatched area) there exists, with probability at least $1-\eps$, a path (the black edges) to an open $m$-pad (the blue square on the right) in the target area $T_1$ (the green hatched area). This path does not use edges outside the big $(14n+1) \times (6n+1)$ rectangle, which is $M_1$.}\label{fig:final steering}}
	\end{center}
\end{figure}

A key idea in the proof of Theorem \ref{theo:main} is to show that a renormalized version of the truncated long-range percolation graph dominates a supercritical directed site-bond percolation model in dimension $d=2$. For this, we first define a model of directed percolation on the positive quadrant $\N_0 \times \N_0 \times \{0\}^{d-2} \eqqcolon V$. We write $V = \bigcup_{n=0}^{\infty} V_n$, where $V_n = \left\{v \in V : \|v\|_1 = n\right\}$. Vertices can be dead in this model, and directed edges of the form $(x,x+e_i)$ with $x\in V$ and $i\in \{1,2\}$ can be open or closed. We sequentially explore the open cluster containing the origin. For this, we define sets of {\sl active} vertices $A_n \subset V_n$ as follows. We start with $A_0 = V_0 = \left\{\mz\right\}$. Then, for given $A_{n-1} = \{x_1,\ldots,x_k\} \subset V_{n-1}$ we construct $A_{n} \subset V_{n}$ as follows:
\begin{itemize}
	\item[(1.)] For $i=1,\ldots,k$, define the edge $(x_i,x_i+e_1)$ to be open with conditional probability at least $q_{x_i,e_1}$ and closed with conditional probability at most $1-q_{x_i,e_1}$. If the edge $(x_i,x_i+e_1)$ is open, define $x_i+e_1$ to be active and add it to the set $A_n$. Otherwise, we say that the vertex $x_i+e_1$ is dead.
	\item[(2.)] For $i=1,\ldots,k$, if $x_i+e_2$ was declared either dead or active in step (1.), do nothing. \\
	Otherwise, define the edge $(x_i,x_i+e_2)$ to be open with conditional probability at least $q_{x_i,e_2}$ and closed with conditional probability at most $1-q_{x_i,e_2}$. If the edge $(x_i,x_i+e_2)$ is open, define $x_i+e_2$ to be active and add it to the set $A_n$.
\end{itemize}
Here, the numbers $q=\left(q_{x,e_i}\right)_{i\in \{1,2\}, x\in V}$ are real numbers in the interval $\left[0,1\right]$. The above algorithm should be read as follows. Whenever we make a choice whether we define an edge $(x,x+e_i)$ to be open, then the probability that this edge is open is, given everything that occurred so far, at least $q_{x,e_i}$.
Write $\p_{q}$ for the resulting probability measure. Using induction on $n=1,2,\ldots$ one sees that there exists a path of upward/right-directed edges from $\mz$ to all $x\in A_n$. Thus, if $|A_n| \geq 1$ for all $n\in \N$, then there exists an infinite upward/right-directed path starting at the origin. Using a Peierl's argument for $\N_0^2$, one can also see that this occurs with positive probability for $\sup_{x\in V, i\in \{1,2\}} |1-q_{x,e_i}|$ small enough. Say that $\rho \in (0,1)$ is such that
\begin{equation}\label{eq:nec cond}
	q_{x,e_i} \geq \rho \text{ for all $x\in V, i \in \{1,2\}$} \ \ \Rightarrow \ \ \p_q \left( |A_n| \geq 1 \text{ for all $n\in \N_0$}\right) > 0 .
\end{equation} 
With this, we are finally ready to go to the proof of Theorem \ref{theo:main}.

\begin{proof}[Proof of Theorem \ref{theo:main}]
	Let $\beta> \beta_c$, and let $\eta > 0$ and $\tilde{\beta} \in \left(\beta_c,\beta\right)$ be such that $\tilde{\beta}+2\eta = \beta$. Assume that  $\omega_{\tilde{\beta}}, \omega_{\eta}^{\prime}, \omega_{\eta}^{\prime \prime} \in \{0,1\}^E$ are three independent percolation configurations such that $\omega_{\eta}^{\prime}$ and $\omega_{\eta}^{\prime \prime}$ are distributed like $\omega_{\eta}$. Using the same calculation as in \eqref{eq:maxim}, we get that
	\begin{equation}\label{eq:maxvee}
		\omega \coloneqq
		\omega_{\tilde{\beta}} \vee \omega_{\eta}^{\prime} \vee \omega_{\eta}^{\prime \prime} \overset{d}{=} \omega_\beta,
	\end{equation}
	where $\omega_{\tilde{\beta}} \vee \omega_{\eta}^{\prime} \vee \omega_{\eta}^{\prime \prime} \in \{0,1\}^E$ is defined as the pointwise maximum of $\omega_{\tilde{\beta}}, \omega_{\eta}^{\prime}$, and $ \omega_{\eta}^{\prime \prime}$.
	
	Let $1-2\eps = \rho < 1$ be large enough such that \eqref{eq:nec cond} holds. Let $K > 0$ be large enough so that $e^{-\eta K} < \eps$ and let $m,n \in \N$ be large enough such that for all $i\in \{1,2\}$, all sets $A=B_m(u) \subset B_n(\mz)$, and all sets $R$ with $A\subseteq R \subseteq M_i \setminus T_i$
	\begin{equation}\label{eq:m pad condition beta tilde}
		\p_{\tilde{\beta}} \left( J(R,W) > K \right) > 1-\eps,
	\end{equation}
	where the sets $W, M_i$, and $T_i$ are defined (as in Corollary \ref{coro:steer2}) by
	\begin{align*}
		& M_i = \left\{-3n,\ldots,3n\right\}^{i-1} \times \left\{-3n,\ldots,11n\right\} \times \left\{-3n,\ldots,3n\right\}^{d-i}  , \\
		& T_i =  \left\{-n,\ldots,n\right\}^{i-1} \times \left\{7n,\ldots,9n\right\} \times \left\{-n,\ldots,n\right\}^{d-i}, \\
		& P_i =  \left\{x \in T_i: x \text{ is contained in an open $m$-pad } A \subset T_i\right\},\\
		& W =  K_{P_i} \left( M_i \setminus R \right) = \left\{x \in M_i \setminus R : x \overset{M_i \setminus R}{\longleftrightarrow} P_i \right\} .
	\end{align*}
	We now iteratively define sets of vertices $(A_k)_{k\in \N_0}$ which are subsets of $V\coloneqq \N_0 \times \N_0 \times \{0\}^{d-2}$	such that $A_k \subseteq \left\{u \in V : \|u\|_1 = k\right\}$.
	We define these sets depending on the percolation configurations $\omega_{\tilde{\beta}}, \omega_\eta^{\prime}$, and $\omega_\eta^{\prime\prime}$. A point $v\in V$ then corresponds to the box $B_n(8nv) \subset \Z^d$, and the connection of the points $v \in V$ and $u=v+e_i$ (with $i\in \{1,2\}$) depends on a connection event inside the set $8nv+M_i$. For $u\in V = \N_0 \times \N_0 \times \{0\}^{d-2}$ and $i\in\{1,2\}$, we write 
	\begin{equation*}
		M_i^u= 8n u + M_i \ \ \text{ and } \ \ T_i^u= 8n u + T_i.
	\end{equation*}
	
	Furthermore, for each vertex $u\in A_k$, there is also a set $R_1^u \subseteq B_{3n}(8nu)$ that is also associated with $u$. Let us now define the sets  $(A_k)_{k\in \N_0}$. 
	If $B_m(\mz)$ is an $\omega_{\tilde{\beta}}$-open $m$-pad, we define $A_0=\{\mz\} \subset V$ and $R_1^{\mz}=B_m(\mz)$. Otherwise we define $A_0=\emptyset$ and stop the exploration. For given $A_{k-1}=\{x_1,\ldots, x_l\} \subset V_{k-1}$, and the sets $\left(R_1^{x_i}\right)_{i\in \{1,\ldots,l\}}$, we define the set $A_k$ as follows:
	\begin{itemize}
		\item[(1.)] For $i=1,\ldots,l$: Let $u=x_i$. Define the sets
		\begin{align*}
			& R_{+}^{u} = \left\{ x \in M_1^u \setminus R_1^u : x \sim R_1^u \text{ in } \omega_{\eta}^\prime \right\} \ \ \text{ and}\\
			& X_1^u = K_{R_{+}^{u}}\left(M_1^u \setminus R_1^u; \omega_{\tilde{\beta}}\right) = \left\{x \in M_1^u \setminus R_1^u : x \overset{M_1^u \setminus R_1^u}{\longleftrightarrow} R_{+}^{u} \text{ in } \omega_{\tilde{\beta}} \right\} .
		\end{align*}
		We define the (directed) edge $(u,u+e_1)$ to be open if there exists an open $m$-pad $A \subset X_1^u\cap B_{n}(8n (u+e_1))$ in the environment $\omega_{\tilde{\beta}}$. Then we also define $u+e_1$ as active and add it to the set $A_k$. Further, we define $R_1^{u+e_1} = X_1^u\cap B_{3n}(8n (u+e_1))$.
		
		If we did not define $u+e_1$ as active so far, we say that $u+e_1$ is dead.
		
		\item[(2.)] For $i=1,\ldots,l$, if $x_i+e_2$ was declared either dead or active in step (1.), do nothing. \\
		Otherwise, let $u=x_i$ and define the sets
		\begin{align*}
			& R_2^u = \left(R_1^u  \cup X_1^u\right) \cap B_{3n}(8nu),\\
			& R_{++}^{u} = \left\{ x \in M_2^u \setminus R_2^u : x \sim R_2^u \text{ in } \omega_{\eta}^{\prime\prime} \right\}, \ \text{ and }\\
			& X_2^u = K_{R_{++}^{u}}\left(M_2^u \setminus R_2^u; \omega_{\tilde{\beta}} \right) 
			= 
			\left\{x \in M_2^u \setminus R_2^u : x \overset{M_2^u \setminus R_2^u}{\longleftrightarrow} R_{++}^{u} \text{ in } \omega_{\tilde{\beta}} \right\} .
		\end{align*}
		We define the (directed) edge $(u,u+e_2)$ to be open if there exists an open $m$-pad $A \subset X_2^u\cap B_{n}(8n (u+e_2))$ in the environment $\omega_{\tilde{\beta}}$. Then we also define $u+e_2$ as active and add it to the set $A_k$. Further, we define $R_1^{u+e_2} = X_2^u\cap B_{3n}(8n (u+e_2))$.
	\end{itemize}
	
	We will now show that if $|A_k|\geq 1$ for all $k\in \N_0$, then there exists an infinite open cluster in the percolation environment $\omega=\omega_{\tilde{\beta}} \vee \omega_{\eta}^{\prime} \vee \omega_{\eta}^{\prime\prime}$ where we erased all edges longer than $14n$ (in the $\infty$-norm). We do this by showing inductively that $\mz \overset{\leq 14 n}{\longleftrightarrow} x$ in the environment $\omega = \omega_{\tilde{\beta}} \vee \omega_\eta^{\prime} \vee \omega_\eta^{\prime\prime}$ for all $x \in R_i^u$, for all $u\in A_k$ and $i=1,2$. For $k=0$, this directly follows from the definition of $R_1^\mz$ and $R_2^\mz$.
	Assume that $v\in A_{k+1}$ is such that $v=u+e_1$ and the edge $(u,v)$ is open. By the induction assumption, there exist $\omega$-open paths between all $z \in R_1^u$ and $\mz$ and these paths only use edges $e$ with $|e|\leq 14 n$. All vertices in the set $R_{+}^{u}$ can be reached from $R_1^u$ using an $\omega_\eta^\prime$-open edge of length at most $14n$. The set $X_1^u$ can be reached from $R_{+}^{u}$ using only $\omega_{\tilde{\beta}}$-open edges of length at most $14n$. Since $R_1^v \subset X_1^u$, for each $x\in R_1^v$ there exists an $\omega$-open path from $R_1^u$ to $x$ and this path only uses edges of length at most $14n$. By the induction assumption, we also have $\mz \overset{\leq 14 n}{\longleftrightarrow} y$ in $\omega$ for each $y\in R_1^u$, and thus we get that $\mz \overset{\leq 14 n}{\longleftrightarrow} x$ in $\omega$. The case $i=2$ works analogous. Inductively, we get for all $u\in A_k$ that $x \in X_i^u$ is connected to $\mz$ by an $\omega$-open path that only uses edges of length at most $14 n$.
	So provided we can show that $\p \left( |A_k| \geq 1 \text{ for all $k\in \N_0$}\right) > 0$, this implies that the kernel $\tilde{J} : \Z^d \to \left[0,\infty\right)$ defined by
	\begin{equation*}
		\tilde{J}(x)= \begin{cases}
			J(x) & \text{ if } \|x\|_\infty \leq 14 n\\
			0 & \text{ else }
		\end{cases}
	\end{equation*}
	satisfies $\theta(\beta, \tilde{J}) > 0$ and thus $\beta\geq \beta_c(\tilde{J})$.\\	
	
	In the remainder of the proof, we show that $\p \left( |A_k| \geq 1 \text{ for all $k\in \N_0$}\right) > 0$. In the light of condition \eqref{eq:nec cond}, it suffices to show that the conditional probability of forming an open edge between an active vertex $u\in A_{k}$ and $u+e_i$ is at least $\rho$ for $i=1,2$.
	
	Let $k\in \N_0$ and let $u \in A_k$. Assume that we are in step (1.) of the above exploration, i.e., we explore the rectangle $M_1^u$ to create a connection between $m$-pads in $B_n(8nu)$ and $B_n(8n(u+e_1)) = T_1^u$. Condition on the event that $R_1^u = R$ for some set $R\subset B_{3n}(8nu)$. The event $R_1^u = R$ is independent of all random variables $\omega_\eta^\prime(\{a,b\})$ with $a,b \in M_1^u$ and all random variables $\omega_{\tilde{\beta}}(\{a,b\})$ with $a,b \in M_1^u \setminus R$. Define the set
	\begin{equation*}
		P_1^u = \left\{x\in T_1^u: x \text{ is contained in an $\omega_{\tilde{\beta}}$-open $m$-pad } A \subset T_1^u\right\}.
	\end{equation*}
	As the set $R=R_1^u$ contains an $\omega_{\tilde{\beta}}$-open $m$-pad $B \subset B_n(8nu)$, this implies, by \eqref{eq:m pad condition beta tilde}, that $\p \left(J(R,W)>K\right)>1-\eps$, where the set $W\subset M_1^u$ is defined by
	\begin{equation*}
		W= K_{P_1^u} \left(M_1^u\setminus R; \omega_{\tilde{\beta}}\right) = \left\{x\in M_1^u\setminus R : x \overset{M_1^u\setminus R}{\longleftrightarrow} P_1^u \text{ in } \omega_{\tilde{\beta}} \right\} .
	\end{equation*}
	Provided that $J(R,W)>K$, there exists an $\omega_\eta^{\prime}$-open edge between $R$ and $W$ with probability at least $1-e^{-\eta K} \geq 1-\eps$. If there exists such an $\omega_\eta^{\prime}$-open edge $\{a,b\}$ with $a\in R, b\in W$, then $b\in R_{+}^u$ and there exists an $\omega_{\tilde{\beta}}$-open path from $b$ to an $\omega_{\tilde{\beta}}$-open $m$-pad $A\subset T_1^u$, and thus we define the edge $(u,u+e_1)$ as open and $u+e_1$ as active. So in particular we see that, conditioned that $u$ is active and $R_1^u=R$, we get that $(u,u+e_1)$ is open with probability at least $(1-\eps)^2 \geq  \rho$. As this holds uniformly over all sets $R_1^u=R$ containing and $\omega_{\tilde{\beta}}$-open $m$-pad $B \subset B_n(8nu)$, this implies that if $u$ is active, the edge $(u,u+e_1)$ is open with probability at least $\rho$, and thus also $u+e_1$ is active with probability at least $\rho$.\\
	
	Next, assume that we are in step (2.) of the above exploration, i.e., we explore the rectangle $M_2^u$ to create a connection between $m$-pads in $B_n(8nu)$ and $B_n(8n(u+e_2)) = T_2^u$. As the vertex $u+e_2$ was not declared active or dead previously, there is no information revealed so far about edges with both endpoints in $M_2^u\setminus B_{3n}(8nu)$. Condition on the event that $R_2^u = R$ for some set $R\subset B_{3n}(8nu)$. The event $R_2^u = R$ is independent of all random variables $\omega_\eta^{\prime\prime}(\{a,b\})$ with $a,b \in M_2^u$ and all random variables $\omega_{\tilde{\beta}}(\{a,b\})$ with $a,b \in M_2^u \setminus R$. Define the set
	\begin{equation*}
		P_2^u = \left\{x\in T_2^u: x \text{ is contained in an $\omega_{\tilde{\beta}}$-open $m$-pad } A \subset T_2^u\right\}.
	\end{equation*}
	As the set $R=R_2^u$ contains an $\omega_{\tilde{\beta}}$-open $m$-pad $B \subset B_n(8nu)$, this implies, by \eqref{eq:m pad condition beta tilde}, that $\p \left(J(R,W)>K\right)>1-\eps$, where the set $W\subset M_1^u$ is defined by
	\begin{equation*}
		W= K_{P_2^u} \left(M_2^u\setminus R; \omega_{\tilde{\beta}}\right) = \left\{x\in M_2^u\setminus R : x \overset{M_2^u\setminus R}{\longleftrightarrow} P_2^u \text{ in } \omega_{\tilde{\beta}} \right\} .
	\end{equation*}
	Provided that $J(R,W)>K$, there exists an $\omega_\eta^{\prime\prime}$-open edge between $R$ and $W$ with probability at least $1-e^{-\eta K} \geq 1-\eps$. If there exists such an $\omega_\eta^{\prime\prime}$-open edge $\{a,b\}$ with $a\in R, b\in W$, then $b\in R_{++}^u$ and there exists an $\omega_{\tilde{\beta}}$-open path from $b$ to an $\omega_{\tilde{\beta}}$-open $m$-pad $A\subset T_2^u$, and thus we define the edge $(u,u+e_2)$ as open and $u+e_2$ as active. So in particular we see that, conditioned that $u$ is active and $R_2^u=R$ we get that $(u,u+e_2)$ is open with probability at least $(1-\eps)^2 \geq  \rho$. As this holds uniformly over all sets $R_2^u=R$ containing and $\omega_{\tilde{\beta}}$-open $m$-pad $B \subset B_n(8nu)$, this implies that if $u$ is active, the edge $(u,u+e_2)$ is open with probability at least $\rho$, and thus also $u+e_2$ is active with probability at least $\rho$.
\end{proof}

\section{A shape theorem for the long-range percolation metric}\label{sec:shape}

The next property of the supercritical long-range percolation cluster that we consider is the growth of chemical distances. The study of chemical distances and shape theorems is a classical question in percolation theory, and shape theorems for various settings were previously established \cite{kesten2003first,cerf2016weak,garet2004asymptotic}. 
Theorem \ref{theo:main} above says that the infinite supercritical long-range percolation cluster contains a supercritical finite range percolation cluster when $J\left(x\right) = \mathcal{O}\left(\|x\|^{-2d}\right)$. In a finite-range percolation cluster, the chemical distance between two points is comparable to the Euclidean distance due to the following result by Cerf and Théret \cite[Theorem 6]{cerf2016weak}, following earlier work of Antal and Pisztora \cite{antal1996chemical}. (They proved the result for nearest-neighbor percolation, but the same proof works for finite-range percolation.)

\begin{theorem}[Cerf-Théret\cite{cerf2016weak}] \label{theo:theret}
	Let $d\geq 2$, let $J:\Z^d \to \left[0,\infty\right)$ be a symmetric and irreducible kernel with finite range, and let $\beta > \beta_c(J)$. Then there exist positive constants $A_1,A_2,A_3$ such that
	\begin{equation}\label{eq:theret}
		\forall x,y \in \Z^d, \forall l \geq A_3 \|x-y\|, \ 
		\p_{\beta,J} \left(x \leftrightarrow y, D(x,y)\geq l\right) \leq A_1 e^{-A_2 l} .
	\end{equation}
\end{theorem}

Using this result and Theorem \ref{theo:main}, it is straightforward to show that the chemical distance between two {\sl typical} points $x,y \in \mathcal{C}_\infty$ in the  long-range percolation model is of the same order as the Euclidean distance. However, there might be points $x\in \Z^d$ that are several steps away from the infinite finite-range cluster.
To circumvent such problems, we introduce Propositions \ref{prop:stretched exp} and \ref{propo:zeta} below. These are the main technical innovations in this paper. The proof of Theorem \ref{theo:chemical dist} given these propositions follows from relatively standard techniques, like the subadditive ergodic theorem. The resilience of kernels (Theorem \ref{theo:main}) is also an essential input in the proof of Propositions \ref{prop:stretched exp} and \ref{propo:zeta}.

\begin{proposition}\label{prop:stretched exp}
	Let $d\geq 2$, let $J:\Z^d \rightarrow \left[0,\infty\right)$ be a symmetric and irreducible kernel satisfying $J(x)=\mathcal{O}(\|x\|^{-s})$ for some $s>2d$, and let $\beta > \beta_c(J)$. Then there exists $\kappa > 0$ and $C^\prime < \infty$ such that
	\begin{equation*}
		\p_{\beta} \Big( \exists x,y \in \mathcal{C}_\infty(\omega) \cap B_n(\mz): D \left( x , y \right) > C^\prime n \Big) \leq n^{-\kappa}
	\end{equation*}
	for all large enough $n\in \N$.
\end{proposition}

\begin{proposition}\label{propo:zeta}
	Let $d\geq 2$, let $J:\Z^d \rightarrow \left[0,\infty\right)$ be a symmetric and irreducible kernel satisfying $J(x)=\mathcal{O}(\|x\|^{-s})$ for some $s>2d$, and let $\beta>\beta_c(J)$. Let $\zeta = \frac{1}{8d}$. Then
	\begin{equation}\label{1.1}
		\p_{\beta} \left( \exists x,y \in B_{n^\zeta}(\mz) : n < D(x,y) < \infty  \right) \leq n^{-1.1}
	\end{equation}
	for all large enough $n$.
\end{proposition}
\noindent
We postpone the proofs of these results to section \ref{sec:proof of propos}. First, we discuss how these two propositions imply Theorem \ref{theo:chemical dist}.

\subsection{Proof of Theorem \ref{theo:chemical dist}}

The proof of Theorem \ref{theo:chemical dist} given the two propositions follows from classical techniques, as demonstrated in \cite{cerf2016weak, garet2004asymptotic}. The same moment properties that follow from the exponential decay results (Theorem \ref{theo:theret}) in the finite-range setting do not hold for long-range percolation, but Proposition \ref{propo:zeta} still provides strong enough estimates on the moments.

Before going to the proof, we also introduce the following result, which states that it is exponentially unlikely for a box $B_n(\mz)$ to not intersect the infinite cluster. We will use that for all symmetric finite-range kernels $J$ and all $\beta > \beta_c(J)$ there exist $C<\infty$ and $\eta > 0$ such that 
\begin{equation}\label{eq:roberto}
	\p_{\beta,J} \left( B_n(\mz) \cap \mathcal{C}_\infty = \emptyset \right) \leq C e^{-n^\eta}
\end{equation}
for all $n\in \N$.
This follows from the results of Durrett and Schonmann in dimension $d=2$ \cite{durrett1988large} and from the results of Chayes, Chayes, and Newman in dimensions $d\geq 3$ \cite{chayes1987bernoulli}.
Note that \eqref{eq:roberto} directly implies the following result: If $J$ is a symmetric and resilient kernel, and $\beta > \beta_c(J)$, then there exists $N$ such that the kernel $\tilde{J}$ defined by $\tilde{J}(x) = J(x) \mathbbm{1}_{\|x\|\leq N}$ satisfies $\beta> \beta_c(\tilde{J})$. So in particular inequality \eqref{eq:roberto} holds for the kernel $\tilde{J}$. Thus we get that
\begin{equation}\label{eq:roberto2}
	\p_{\beta,J} \left( B_n(\mz) \cap \mathcal{C}_\infty = \emptyset \right) \leq \p_{\beta,\tilde{J}} \left( B_n(\mz) \cap \mathcal{C}_\infty = \emptyset \right) \leq C e^{-n^\eta}
\end{equation}
for some $C<\infty, \eta > 0$, and all $n\in \N$. So in particular, it is (stretched-exponentially) unlikely that the infinite cluster $\cC_\infty$ sampled from $\p_{\beta,J}$ does not intersect $B_n(\mz)$.

\begin{proof}[Proof of \eqref{eq:chemdist limsup} given Propositions \ref{prop:stretched exp} and \ref{propo:zeta}]
	Let $\zeta = \frac{1}{8d}$. We start by showing that
	\begin{equation}\label{eq:L1 of chem dist}
		\E_\beta \left[ \hat{D}(\mz, x) \right] < \infty
	\end{equation}
	for all $x \in \Z^d$. If $\hat{D}(\mz, x)>n$, then either there are $u,v \in B_{n^\zeta}(\mz) \cap \cC_\infty$ with $D(u,v) > n$, or (at least) one of $\hat{\mz}, \hat{x}$ is not in $B_{n^\zeta}(\mz)$; Remember that for $x\in \Z^d$ we write $\hat{x}$ for $\arg\min_{y \in \cC_\infty} \|x-y\|$, with lexicographic tie-breaking. Thus we get that
	\begin{align*}
		& \p_\beta \left( \hat{D}(\mz, x) >n \right) 
		\\
		& 
		\leq 
		\p_{\beta} \left( \exists u,v \in B_{n^\zeta}(\mz) \cap \mathcal{C}_\infty : D(u,v) > n \right) +  \p_\beta \left( \hat{\mz} \notin B_{n^\zeta}(\mz) \right)
		+
		\p_\beta \left( \hat{x} \notin B_{n^\zeta}(x) \right) 
	\end{align*}
	The first term in this sum can be bounded from above using Proposition \ref{propo:zeta}. For the second and third term, observe that if $n$ is large enough, then the two implications 
	\begin{align*}
		\left\{\hat{\mz} \notin B_{n^\zeta}(\mz)\right\} \Rightarrow \left\{\mathcal{C}_\infty \cap B_{n^{\zeta/2}}(\mz) = \emptyset\right\} \text{ and } \left\{\hat{x} \notin B_{n^\zeta}(\mz)\right\} \Rightarrow \left\{\mathcal{C}_\infty \cap B_{n^{\zeta/2}}(x) = \emptyset\right\}
	\end{align*}
	hold. So in particular we see that
	\begin{align}\label{roberto impli}
		\notag \p_\beta \left( \hat{\mz} \notin B_{n^\zeta}(\mz) \right)
		+
		\p_\beta \left( \hat{x} \notin B_{n^\zeta}(x) \right) &
		\leq
		\p_\beta \left( \mathcal{C}_\infty \cap B_{n^{\zeta/2}}(\mz) = \emptyset \right)
		+
		\p_\beta \left( \mathcal{C}_\infty \cap B_{n^{\zeta/2}}(x) = \emptyset \right)
		\\
		&
		\leq
		2 C\exp \left(-n^{\frac{\zeta \eta}{2}}\right)
	\end{align}
	for some $\eta>0$, by \eqref{eq:roberto2}. Thus we get that for $n$ large enough
	\begin{align*}
		& \p_\beta \left( \hat{D}(\mz, x) >n \right) 
		\\
		& 
		\leq 
		\p_{\beta} \left( \exists u,v \in B_{n^\zeta}(\mz) \cap \mathcal{C}_\infty : D(u,v) > n \right) +  \p_\beta \left( \hat{\mz} \notin B_{n^\zeta}(\mz) \right)
		+
		\p_\beta \left( \hat{x} \notin B_{n^\zeta}(x) \right) \\
		&
		\overset{\eqref{roberto impli}}{\leq}
		\p_{\beta} \left( \exists u,v \in B_{n^\zeta}(\mz) \cap \mathcal{C}_\infty : D(u,v) > n \right) + 2 C\exp \left(-n^{\frac{\zeta \eta}{2}}\right)
		\overset{\eqref{1.1}}{\leq} 
		n^{-1.1} + 2 C\exp \left(-n^{\frac{\zeta \eta}{2}}\right)
	\end{align*}
	for some constants $C< \infty$ and $\eta>0$. Summing this over all $n\in \N_0$ shows that the expectation in \eqref{eq:L1 of chem dist} is indeed finite. For $z \in \R^d$ and $z_d \in \Z^d$ with $z \in z_d + \left[-\tfrac{1}{2}, \tfrac{1}{2}\right)^d$ one has $\hat{D}(\mz,z) = \hat{D}(\mz,z_d)$ and thus we also get that $\E_{\beta}\left[\hat{D}(\mz,z)\right] < \infty$ for all $z \in \R^d$.
	With this, we can construct the norm $\mu$. We do this very briefly and follow the same arguments as presented by Cerf and Theret \cite{cerf2016weak}. $\hat{D}$ is a pseudometric on $\Z^d$, and thus $\hat{D}(x,z) \leq \hat{D}(x,y) + \hat{D}(y,z)$ for all $x,y,z\in \Z^d$. As the percolation process is stationary and ergodic, we can use Kingman's subadditive ergodic theorem \cite{kingman1973subadditive}. For $x\in \Z^d$, we define
	\begin{equation*}
		\mu(x) \coloneqq \lim_{n \to \infty} \frac{\hat{D}(\mz,nx)}{n}
	\end{equation*}
	where the limit exists almost surely and in $L_1$ under the measure $\p_{\beta}$. Similarly, for $x\in \mathbb{Q}^d$, let $N \in \N_{>0}$ be such that $Nx \in \Z^d$. Then we define
	\begin{equation*}
		\mu(x) \coloneqq \lim_{n \to \infty} \frac{\hat{D}(\mz,nNx)}{nN}.
	\end{equation*}
	The function $\mu$ is a semi-norm on $\mathbb{Q}^d$ and the limit above does not depend on the choice of $N$. As $\mu(x) \leq \|x\|_\infty (\mu(e_1)+\ldots+\mu(e_d))$, the function $\mu$ is Lipschitz-continuous on $\mathbb{Q}^d$ and thus we can extend it to $\R^d$. It directly follows from the asymptotic lower-bound on the chemical distance \eqref{eq:noam} that $\mu(x)>0$ for all $x\in \R^d \setminus\{\mz\}$, so $\mu$ is a norm on $\R^d$.\\	
\end{proof}

Next, we show the following result for long-range percolation, which was, for finite-range percolation previously proven in \cite[Lemma 1]{cerf2016weak}. The proof for long-range percolation uses Propositions \ref{prop:stretched exp} and \ref{propo:zeta}. Because of the polynomial decay in the statements of the propositions (compared to stretched exponential decay for analogous statements for finite-range percolation), we need slightly different tools compared to the proof for finite-range percolation.

\begin{lemma}\label{lem:theret2}
	Let $J$ be a kernel satisfying $J(x)=\mathcal{O}(\|x\|^{-s})$ for some $s>2d$, and let $\beta>\beta_c(J)$. Then there exists a constant $C$ such that for any $\eps > 0$ there exists almost surely $R> 0$ such that for all $u,v\in \Z^d$
	\begin{equation*}
		\left.\begin{array}{r}
			\|u\|_\infty \geq R \\
			\|u-v\|_\infty \leq \eps\|u\|_\infty
		\end{array}\right\} \Longrightarrow \hat{D}(u,v) \leq C \varepsilon\|u\|_\infty .
	\end{equation*}
\end{lemma}
\begin{proof}[Proof of Lemma \ref{lem:theret2} given Propositions \ref{prop:stretched exp} and \ref{propo:zeta}]
	Let $\eps > 0$. For each $k\in \N$, let $\cT_k \subset \Z^d$ be a deterministic set so that
	\begin{equation*}
		\cT_k \subset B_{2^{k}}(\mz)  \ \text{ , } \ \bigcup_{x \in \cT_k} B_{\eps 2^k}(x) \supseteq  B_{2^{k}}(\mz) , \  \text{ and } \ T \coloneqq \sup_k |\cT_k| < \infty ,
	\end{equation*}
	which is possible for any fixed $\eps > 0$.
	Indeed, the set $\cT_k$ needs to have a cardinality of order at least $\eps^{-d}$. Define the event $\cE_k$ by 
	\begin{align*}
		\cE_k =  
		\bigcap_{x\in \cT_k} \left\{  D \left( u , v \right) \leq C^\prime 4 \eps 2^k \ \forall u,v \in \mathcal{C}_\infty(\omega) \cap B_{4\eps 2^{k}}(x) \right\} \cap \bigcap_{x\in B_{2^k}(\mz)} \left\{B_{\sqrt{2^k}}(x) \cap \cC_\infty \neq \emptyset \right\}
	\end{align*}
	where $C^\prime$ is the constant from Proposition \ref{prop:stretched exp}. Using the result of Proposition \ref{prop:stretched exp} and inequality \eqref{eq:roberto2}, it directly follows from a union bound that for $k$ large enough
	\begin{align*}
		\p_\beta \left(\cE_k^c\right)&
		\leq
		\sum_{x\in \cT_k} \p_\beta \left( D \left( u , v \right) > C^\prime 4 \eps 2^k \ \text{ for some } u,v \in \mathcal{C}_\infty(\omega) \cap B_{4\eps 2^{k}}(x) \right)
		\\
		&
		\hspace{7cm}
		+
		\sum_{x\in B_{2^k}(\mz)} \p_\beta \left(B_{\sqrt{2^k}}(x) \cap \cC_\infty = \emptyset \right) 
		\\
		&
		 \leq T (4 \eps 2^k)^{-\kappa} + (2^k+1)^d C \exp\left(-2^\frac{k\eta}{2}\right)  \leq T^\prime 2^{-k\kappa}
	\end{align*}
	for some constant $T^\prime < \infty$. So in particular we get that $\sum_{k=1}^{\infty} \p_\beta \left( \cE_k^c \right) < \infty$ and thus the event $\cE_k^c$ only occurs for finitely many $k$ almost surely. Say that the event $\cE_k$ holds for all $k\geq K$. Now let $u,v \in \Z^d$ with $\|u\|_\infty > 2^K$ and $\|u-v\|_\infty  \leq \eps \|u\|_\infty$. Say that $\|u\|_\infty \in \left(2^{k-1}, 2^{k}\right]$. Then there exists $x\in \cT_k$ with $u \in B_{\eps 2^k}(x)$. Further, as $\|u-v\|_\infty \leq \eps \|u\|_\infty \leq \eps 2^k$, and $\|u-\hat{u}\|_\infty \leq \sqrt{2^k}$, $\|v-\hat{v}\|_\infty \leq \sqrt{2^k}$, we get that
	\begin{align*}
		&\|\hat{u}-x\|_\infty \leq \|\hat{u}-u\|_\infty+\|u-x\|_\infty  \leq \sqrt{2^k} + \eps 2^k \text{ and } \\ & \|\hat{v}-x\|_\infty \leq \|\hat{v}-v\|_\infty +  \|v-u\|_\infty +  \|u-x\|_\infty \leq \sqrt{2^k} + \eps \|u\|_\infty + \eps 2^k ,
	\end{align*}
	which implies that $\hat{u},\hat{v} \in  B_{4\eps 2^k}(x)$ for $k$ large enough such that $\eps 2^k > \sqrt{2^k}$. Since the event $\cE_k$ holds for our choice of $k$ and because $\hat{u},\hat{v} \in B_{4\eps 2^k}(x)$ for some $x\in \cT_k$, we get by the definition of the event $\cE_k$ that
	\begin{equation*}
		\hat{D}(u,v) = D(\hat{u},\hat{v}) \leq C^\prime 4 \eps 2^k \leq C^\prime 8 \eps \|u\|_\infty.
	\end{equation*}
	This finishes the proof, as $u,v \in \Z^d$ with $\|u\|_\infty > 2^K$ and $\|u-v\|_\infty  \leq \eps \|u\|_\infty$ were arbitrary.
\end{proof}

Given Lemma \ref{lem:theret2}, the proof of the remaining items of Theorem \ref{theo:chemical dist} does not make any use of the long-range percolation properties of the graphs and is completely analogous to finite-range percolation. It was carried out in full detail in the proofs of \cite[Lemma 2]{cerf2016weak} and \cite[Theorem 3]{cerf2016weak}, respectively, and we do not repeat their proofs here.

\subsection{Proofs of the Propositions}\label{sec:proof of propos}

In this section, we prove Propositions \ref{prop:stretched exp} and \ref{propo:zeta}.
Before going to the proofs, we first need to introduce and prove several intermediate claims.

\begin{notation}\label{notation}
	For $x,y\in \Z^d$ and $\omega\in \{0,1\}^E$, we write $D(x,y;\omega)$ for the graph distance between $u$ and $v$ in the environment $\omega$.
	For $x\in \Z^d$, we write $B_k(x,\omega) \coloneqq \{y\in \Z^d : D(x,y;\omega)\leq k\}$ for the ball of radius $k$ around $x$ in the graph distance in the environment $\omega$. We define the spheres in the graph distance metric $\left(S_k(x,\omega)\right)_{k\in \N_0}$ by
	\begin{equation*}
		S_k(x,\omega) = \{y\in \Z^d : D(x,y;\omega) =  k\} = B_k(x,\omega) \setminus B_{k-1}(x,\omega).
	\end{equation*}
	For an environment $\omega \in \{0,1\}^E$, we define the environment $\omega_{\leq N} \in \{0,1\}^E$ by
	\begin{align*}
		\omega_{\leq N}(e) = \begin{cases}
			\omega (e) & \text{ if } |e|\leq N\\
			0 & \text{ if } |e|>N
		\end{cases}.
	\end{align*}
	So $\omega_{\leq N}$ is the environment after deleting all edges of length longer than $N$.
	We write $\Delta(\omega_{\leq N})$ for the set of connected sets in the environment $\omega_{\leq N}$, i.e.,
	\begin{equation*}
		\Delta(\omega_{\leq N}) = \left\{S \subseteq \Z^d : x \overset{\leq N}{\longleftrightarrow} y \text{ in } S \text{ for all } x,y\in S\right\} ,
	\end{equation*}
	where we say that $x \overset{\leq N}{\longleftrightarrow} y \text{ in } S$ if there exist $a_0,a_1,\ldots,a_\ell \in S$ with $x=a_0, y=a_\ell$, $a_i \sim a_{i-1}$ and $\|a_i-a_{i-1}\|_\infty \leq N$ for $i=1,\ldots,\ell$.
	Note that we do not require that elements of $\Delta(\omega_{\leq N})$ are maximally connected, i.e., elements of $\Delta(\omega_{\leq N})$ are not necessarily open clusters, but they are connected subsets of open clusters. 
	For $u\in \Z^d$ and $K\in \N$, we define $V_u^K \coloneqq Ku + \{0,\ldots, K-1\}^d$ as the box with side length $K$ and base point $Ku$. 
\end{notation}

Note that by local finiteness of the graph, we have that almost surely $|K_x|=\infty$ if and only if $|S_j(x)|\geq 1$ for all $j\in \N$. One other result that we use is the existence of giant components in finite boxes. Let $J$ be a symmetric and irreducible kernel with finite range and let $\beta > \beta_c(J)$. Then the largest cluster inside the box $|V_\mz^K|$ has size comparable to the volume of $|V_\mz^K|$. In particular, there exists some $\rho > 0$ such that for all $K$ large enough
\begin{equation}\label{eq:quadrant}
	\min_{x \in V_\mz^K} \p_{\beta,J} \left(|K_x \left(V_\mz^K\right) | \geq K\right) \geq \rho.
\end{equation}
This result follows directly from  \cite{barsky1991percolation,grimmett1990supercritical}. Using this, we can prove the first intermediary statement.

\begin{claim}\label{claim:Lkdefin}
	Let $J$ be a symmetric and irreducible kernel satisfying $J(x)=\mathcal{O}(\|x\|^{-2d})$, and let $\beta>\beta_c(J)$. Let $N\in \N$ be large enough so that $\beta>\beta_c(J_N)$, where $J_N$ is the kernel defined by $J_N(x)=J(x)\mathbbm{1}_{\|x\|_\infty \leq N}$. (Such an $N$ exists by Theorem \ref{theo:main}).
	For $x\in \Z^d$ and $r\in \left(N,+\infty\right]$, we define the event $L_k(x)$
	\begin{equation}\label{eq:Lkdefin}
		L_k^r(x) \coloneqq \left\{|B_k(x,\omega_{\leq r})|\geq k \text{ and there is no } Z \in \Delta(\omega_{\leq N}) \text{ s.t. } Z \subset B_k(x,\omega_{\leq r}) \text{ and } |Z| \geq  k^{\frac{1}{4d}}   \right\}.
	\end{equation}
	Phrased differently, $L_k^r(x)$ is the event that the cluster containing $x$ in the environment $\omega_{\leq r}$ has `size' at least $k$, but $B_k(x,\omega_{\leq r})$ does not contain a finite-range cluster $Z$ with a size of at least $k^{\frac{1}{4d}}$. Then
	\begin{equation}\label{eq:Lkinequ}
		\p_{\beta,J} \left(L_k^r(x)\right) \leq e^{-\sqrt{k}}
	\end{equation}
	for all $k$ large enough and all $r \in \left(N,+\infty\right]$.
\end{claim}

\begin{proof}
	By translational invariance, it suffices to consider $x=\mz$. Set $K=\lceil k^{\frac{1}{4d}} \rceil$.
	We explore the ball $B_k(\mz,\omega_{\leq r})$ via a variant of breadth first search. For this we proceed as follows.
	\begin{enumerate}
		\item[0.] Start with $A_0 = \left\{\mz\right\}$, $U_0=\{\mz\}$.
		\item[1.] For $i=1,\ldots,\big\lfloor \frac{k}{2} \big\rfloor + 1$:
		\begin{itemize}
			\item[$(a)$] For each $u\in U_{i-1}$, let $y_u \in V_u^K \cap A_{i-1}$; if there exist multiple such vertices, choose the smallest in the lexicographic ordering.
			\item[$(b)$] For each $u\in U_{i-1}$, reveal the set $K_{y_u} \left(V_u^K; \omega_{\leq N}\right)$ and check whether the size of this set satisfies $ \left| K_{y_u} \left(V_u^K; \omega_{\leq N}\right) \right| \geq k^{\frac{1}{4d}}$. 
			\item[$(c)$] Define $A_i \coloneqq S_i(\mz,\omega_{\leq r})=B_i(\mz,\omega_{\leq r})\setminus B_{i-1}(\mz,\omega_{\leq r})$ and \\ $U_i = \Big\{u \in \Z^d: u~\notin~\bigcup_{j=0}^{i-1} U_{j},  B_i(\mz,\omega_{\leq r}) \cap V_u^K \neq \emptyset \Big\}$.
		\end{itemize}
	\end{enumerate}
	In words, the set of {\sl active vertices} $A_i$ after step $i$ is the set $S_i(\mz,\omega_{\leq r})$. For each $u\in \Z^d$ for which $V_u^K \cap B_i(\mz,\omega_{\leq r}) \neq \emptyset$ but $V_u^K \cap B_{i-1}(\mz,\omega_{\leq r}) = \emptyset$, we explore the clusters inside $V_u^K$ in the environment $\omega_{\leq N}$ and check for large clusters in these boxes. 
	
	We define the $\sigma$-algebra $\cF_i$ as the $\sigma$-algebra that contains all the information after step $i$ in the above algorithm, i.e., $\cF_i = \sigma \left( A_0,\ldots,A_i, K_{y_u}\left(V_u^K;\omega_{\leq N}\right) \text{ for all } u \in U_0,\ldots,U_{i-1} \right)$. Also note that the set $U_i$ is measurable with respect to $\cF_i$.
	For each of the sets of the form $V_u^K$ that intersect $S_0(\mz,\omega_{\leq r}) \cup \ldots \cup S_{\lfloor k/2 \rfloor}(\mz,\omega_{\leq r})$, there is a probability of at least $\rho$ that $|K_{y_u}\left(V_u^K\right)| \geq k^{\frac{1}{4d}}$ by \eqref{eq:quadrant}. Also note that for $u\in U_i$, this holds conditionally on the $\sigma$-algebra $\cF_i$, as the sets $A_0,\ldots,A_{i-1},A_i, K_{y_{\tilde{u}}}\left(V_{\tilde{u}}^K; \omega_{\leq N}\right) \text{ for } \tilde{u} \in U_0,\ldots,U_{i-1}$ do not reveal any information about the status of edges $\left\{\{x,y\} : x,y \in V_u^K \right\}$. Also, conditional on $(y_u)_{u\in U_i}$, the different sets $\left(K_{y_u} \left(V_u^K; \omega_{\leq N}\right)\right)_{u\in U_i}$ are independent. Thus we see that
	\begin{align}\label{layer}
		&\notag \p_{\beta,J} \left( \bigcap_{u \in U_i} \left\{\left|K_{y_u} \left(V_u^K; \omega_{\leq N}\right)\right| < k^{\frac{1}{4d}} \right\} \Big| \cF_i \right)
		\\
		&
		\hspace{35mm}
		\leq 
		\left(1 - \min_{x \in V_\mz^K} \p_{\beta,J} \left(|K_x \left(V_\mz^K\right) | \geq K \right)\right)^{|U_i|}
		\leq (1-\rho)^{|U_i|} .
	\end{align}
	Let $X_i=\left|\left\{u\in U_i  :  \left| K_{y_u} \left(V_u^K; \omega_{\leq N}\right) \right| \geq k^{\frac{1}{4d}}\right\}\right|$. Inequality \eqref{layer} then says that the random variable $X_i$ satisfies $\p_{\beta,J} (X_i=0 |\cF_i)\leq (1-\rho)^{|U_i|}$. From this, we can deduce the following claim, which we will prove below.
	
	\begin{claim}\label{claim:expo explo}
		For all $j, \ell \in \N_0$ one has
		\begin{equation}\label{eq:expo explo}
			\p_{\beta,J} \left( \sum_{i=0}^{j} X_i = 0, \sum_{i=0}^{j} |U_i| \geq \ell \right)
			\leq (1-\rho)^{\ell}
		\end{equation}
	\end{claim}
	
	If $|B_{k}(\mz,\omega_{\leq r})|\geq k$, then $B_{\lfloor k/2 \rfloor}(\mz,\omega_{\leq r})=S_0(\mz,\omega_{\leq r}) \cup \ldots \cup S_{\lfloor k/2 \rfloor}(\mz,\omega_{\leq r})$ has size at least $k/2$ and thus, for $k$ large enough, the set $S_0(\mz,\omega_{\leq r}) \cup \ldots \cup S_{\lfloor k/2 \rfloor}(\mz,\omega_{\leq r})$ intersects at least $k^{\frac{3}{5}}$ many sets of the form $V_u^K$ with $u\in \Z^d$. This holds as each set $V_u^K$ contains $K^d = \lceil k^{\frac{1}{4d}} \rceil^d$ many elements, and $k^{\frac{3}{5}} \lceil k^{\frac{1}{4d}} \rceil^d < \frac{k}{2}$ for large enough $k$. Thus we get that $|U_0|+|U_1|+\ldots+|U_{\lfloor k/2 \rfloor}| \geq k^{\frac{3}{5}}$ on the event where $|B_{k}(\mz,\omega_{\leq r})|\geq k$. Thus
	\begin{align}\label{eq:exp sqrtk}
		\notag &\p_{\beta,J} \left( \bigcap_{u \in U_0,\ldots,U_{\lfloor k/2 \rfloor}} \left\{\left|K_{y_u} \left(V_u^K; \omega_{\leq N}\right)\right| < k^{\frac{1}{4d}} \right\}, |B_{k}(\mz,\omega_{\leq r})|\geq k  \right) \\
		&
		\hspace{30mm} \leq
		\p_{\beta,J} \left( \sum_{i=0}^{\lfloor k/2 \rfloor} X_i = 0, \sum_{i=0}^{\lfloor k/2 \rfloor} |U_i| \geq k^{\frac{3}{5}} \right)
		\overset{\eqref{eq:expo explo}}{\leq} (1-\rho)^{k^{\frac{3}{5}}} \leq e^{-\sqrt{k}},
	\end{align}
	where the last inequalities hold for $k$ large enough, and where we used the results of Claim \ref{claim:expo explo} in the second-to-last inequality.\\

	Next, we argue that
	\begin{multline}\label{inclu}
		\bigcap_{u \in U_0,\ldots,U_{\lfloor k/2 \rfloor}} \left\{\left|K_{y_u} \left(V_u^K; \omega_{\leq N}\right)\right| < k^{\frac{1}{4d}} \right\} \\ \supset 
		\left\{ \text{there is no } Z \in \Delta(\omega_{\leq N}) \text{ s.t. } Z \subset B_k(\mz,\omega_{\leq r}) \text{ and } |Z| \geq  k^{\frac{1}{4d}} \right\} .
	\end{multline}
	Indeed, if there exists $u\in \Z^d$ with $\left|K_{y_u}\left(V_u^K; \omega_{\leq N}\right)\right| \geq k^{\frac{1}{4d}}$ and $u\in U_\ell$ for $\ell \leq \lfloor k/2 \rfloor$, then $y_u \in S_{\ell}(\mz,\omega_{\leq r})$ and thus $K_{y_u}\left(V_u^K; \omega_{\leq N}\right) \subset B_{K^d}(y_u,\omega_{\leq r}) \subset B_{\ell + K^d}(\mz,\omega_{\leq r})  \subset B_k(\mz,\omega_{\leq r})$. So in particular, the set $Z$ defined by $Z \coloneqq K_{y_u}\left(V_u^K; \omega_{\leq N}\right)$ satisfies $Z\in \Delta \left(\omega_{\leq N}\right)$ and $Z \subset B_k(\mz,\omega_{\leq r})$. This shows the subset-relation \eqref{inclu}. By the definition of the event $L_k^r(\mz)$, we thus get that
	\begin{equation*}
		\p_{\beta,J} \left(L_k^r (\mz)\right) \leq \p_{\beta,J} \left( \bigcap_{u \in U_0,\ldots,U_{\lfloor k/2 \rfloor}} \left\{\left|K_{y_u} \left(V_u^K; \omega_{\leq N}\right)\right| < k^{\frac{1}{4d}} \right\}, |B_{k}(\mz,\omega_{\leq r})|\geq k  \right) \overset{\eqref{eq:exp sqrtk}}{\leq} e^{-\sqrt{k}} .
	\end{equation*}
	for all large enough $k$, as required.
\end{proof}

We are left to prove Claim \ref{claim:expo explo}.
\begin{proof}[Proof of Claim \ref{claim:expo explo}]
	Inequality \eqref{layer} implies that for all $u_0,\ldots,u_m \in \N_0$ one has
	\begin{equation*}
		\p_{\beta,J} \left(X_m=0 \big| |U_0|=u_0, X_0=0, |U_1|=u_1, X_1=0 \ldots, |U_m|=u_m\right) \leq (1-\rho)^{u_m}.
	\end{equation*}
	Define the set 
	\begin{equation*}
		\mathcal{I_\ell}= \left\{(u_1,\ldots,u_j) \in (\N_0)^j : u_0+\ldots+u_j \geq \ell, \p_{\beta,J} \left( \bigcap_{i=0}^{j} \{|U_i|=u_i, X_i=0\} \right) > 0 \right\}.
	\end{equation*}
	Then we get that
	\begin{align*}
		&\p_{\beta,J} \left( \sum_{i=0}^{j} X_i = 0, \sum_{i=0}^{j} |U_i| \geq \ell \right)
		=
		\sum_{(u_1,\ldots,u_j) \in \mathcal{I}_{\ell}} \p_{\beta,J} \left( \bigcap_{i=0}^{j} \{|U_i|=u_i, X_i=0\} \right)
		\\
		&
		\hspace{7mm} =
		\sum_{(u_1,\ldots,u_j) \in \mathcal{I}_{\ell}}
		\prod_{m=0}^{j} \Bigg( \p_{\beta,J} \left( |U_m|=u_m \Big| \bigcap_{i=0}^{m-1} \{|U_i|=u_i, X_i=0\} \right)
		\\
		& \hspace{47mm} \cdot \p_{\beta,J} \left( X_m=0 \Big| \bigcap_{i=0}^{m-1} \{|U_i|=u_i, X_i=0\}, |U_m|=u_m \right) \Bigg)
		\\
		& \hspace{7mm}
		\leq
		\sum_{(u_1,\ldots,u_j) \in \mathcal{I}_{\ell}}
		\prod_{m=0}^{j} \Bigg( \p_{\beta,J} \left( |U_m|=u_m \Big| \bigcap_{i=0}^{m-1} \{|U_i|=u_i, X_i=0\} \right) (1-\rho)^{u_m} \Bigg)
		\\
		&
		\hspace{7mm} \leq (1-\rho)^{\ell}
		\sum_{(u_1,\ldots,u_j) \in \mathcal{I}_{\ell}}
		\prod_{m=0}^{j} \Bigg( \p_{\beta,J} \left( |U_m|=u_m \Big| \bigcap_{i=0}^{m-1} \{|U_i|=u_i, X_i=0\} \right)  \Bigg) \leq (1-\rho)^{\ell},
	\end{align*}
	which proves the Claim.
\end{proof}

Finally, we are able to go to the proof of Proposition \ref{prop:stretched exp}. One important property of the supercritical finite-range percolation cluster that we will use is the upper bound on the probability of a large finite cluster. Let $J_N$ be the symmetric kernel with finite range defined by $J_N(x)=J(x)\mathbbm{1}_{\|x\|_{\infty}\leq N}$. Then for all $\beta > \beta_c(J_N)$ there exists a constant $c_\beta > 0$ such that
\begin{equation}\label{eq:kesten zhang}
	\p_{\beta,J_N} \left(n \leq |K_\mz| < \infty \right) \leq \exp\left(-c_\beta n^{\frac{1}{d}}\right).
\end{equation}
This result goes back to Kesten and Zhang \cite{kesten1990probability}, improving an earlier result of Chayes, Chayes, and Newman \cite{chayes1987bernoulli}. See also the discussion in Grimmet's book \cite[(8.64)]{grimmett1999percolation} for a proof, and a result of Contreras, Martineau, and Tassion for an improvement of this result to general groups of polynomial growth \cite{contreras2021supercritical}.

\begin{proof}[Proof of Proposition \ref{prop:stretched exp}]
	Let $\eta \in (0,1)$ so that $d+ \eta(d-s) < 0$, which is possible since $s > 2d$. Choose $N \in \N$ (as in Claim \ref{claim:Lkdefin}) such that $\beta>\beta_c(J_N)$, which is possible since $\lim_{N\to \infty} \beta_c(J_N) = \beta_c(J) < \beta$, by Theorem \ref{theo:main}. We will first argue that
	\begin{align}\label{A1-A4}
		& \notag \left\{\exists x \in \mathcal{C}_\infty(\omega) \cap B_n(\mz): D \left( x , \mathcal{C}_\infty(\omega_{\leq N}) \cap B_{2n}(\mz)  \right) > n \right\}\\
		& \notag \hspace{5mm}
		\subset
		\left\{ \exists u \in B_{2n}(\mz), y\in \Z^d : u\sim y, \|u-y\|_\infty \geq n^\eta  \right\} 
		\cup 
		\left\{ \exists z \in B_{n}(\mz) : L_{\lfloor n^{1-\eta} \rfloor}^\infty(z) \text{ holds } \right\}\\
		& \hspace{40mm}
		\cup
		\left\{ \exists u \in B_{2n}(\mz) : \lfloor n^{1-\eta} \rfloor^{\frac{1}{4d}} \leq |K_u(\omega_{\leq N})| < \infty \right\}.
	\end{align}
	For the general structure of the argument, we will show that $\cA_1 \subset \cA_2 \cup \cA_3 \cup \cA_4$, with $\cA_1,\ldots,\cA_4$ the events listed above. We will show that if $\omega \in \cA_1$, $\omega \notin \cA_2$, and $\omega \notin \cA_3$, then $\omega \in \cA_4$. So assume that $\omega \in \cA_1$. Then there exists $x \in \mathcal{C}_\infty(\omega) \cap B_n(\mz)$ so that $x$ is connected to $B_{2n}(\mz)^c$ but $D \left( x , \mathcal{C}_\infty(\omega_{\leq N}) \cap B_{2n}(\mz)  \right) > n$. If $\omega \notin \mathcal{A}_2$, then there is no edge of length at least $n^\eta$ with an endpoint in $B_{2n}(\mz)$ and thus  $B_{n^{1-\eta}}(x,\omega) \subset B_{2n}(\mz)$. If $L_{\lfloor n^{1-\eta} \rfloor }^\infty(x)$ does not occur, then $B_{\lfloor n^{1-\eta} \rfloor }(x,\omega)$ contains a set $Z \in \Delta \left( \omega_{\leq N} \right)$ with $|Z| \geq \lfloor n^{1-\eta} \rfloor^\frac{1}{4d} $. As $Z\subset B_{n^{1-\eta}}(x,\omega) \subset B_{2n}(\mz)$, we can pick $u\in Z \cap B_{2n}(\mz)$.
	
	As $Z\in \Delta \left(\omega_{\leq N}\right)$ and $u\in Z$, we get that $Z \subset K_u(\omega_{\leq N})$, so in particular $|K_u(\omega_{\leq N})| \geq |Z| \geq \lfloor n^{1-\eta} \rfloor^\frac{1}{4d}$.
	Since
	\begin{equation*}
		\text{$D \left( x , \mathcal{C}_\infty(\omega_{\leq N}) \cap B_{2n}(\mz)  \right) > n$ but $D \left( x , Z \right) \leq n^{1-\eta}$,}
	\end{equation*}
	we also get that $Z \nsubseteq \mathcal{C}_\infty(\omega_{\leq N})$. As the set $Z$ is connected in the environment $\omega_{\leq N}$, i.e., $Z \in \Delta\left(\omega_{\leq N}\right)$, we also get that $u \notin \mathcal{C}_\infty(\omega_{\leq N})$, which says that the set $K_u(\omega_{\leq N})$ is finite.
	So in particular the point $ u \in B_{2n}(\mz)$ satisfies $\lfloor n^{1-\eta} \rfloor^{\frac{1}{4d}} \leq |K_u(\omega_{\leq N})| < \infty$, which shows that $\omega \in \cA_4$. Thus we proved the inclusion $\cA_1 \subset \cA_2 \cup \cA_3 \cup \cA_4$, with $\cA_1,\ldots,\cA_4$ the events in \eqref{A1-A4}. By a union bound, we get that
	\begin{align}\label{poly uni}
		& \notag \p_{\beta, J}\left(\exists x \in \mathcal{C}_\infty(\omega) \cap B_n(\mz): D \left( x , \mathcal{C}_\infty(\omega_{\leq N}) \cap B_{2n}(\mz)  \right) > n \right)\\
		& \notag
		\leq 
		\p_{\beta, J} \left( \exists u \in B_{2n}(\mz), y\in \Z^d : u\sim y, \|u-y\|_\infty \geq n^\eta  \right) 
		+ 
		\p_{\beta, J} \left( \exists u \in B_{n}(\mz) : L_{\lfloor n^{1-\eta} \rfloor}^\infty(u) \text{ holds } \right)\\
		&  \hspace{30mm}
		+
		\p_{\beta, J} \left( \exists u \in B_{2n}(\mz) : \lfloor n^{1-\eta} \rfloor^{\frac{1}{4d}} \leq |K_u(\omega_{\leq N})| < \infty \right)
	\end{align}
	and we are only left to show that each of the three summands is polynomially small in $n$.
	For the first summand, we have that
	\begin{align}\label{poly1}
		& \notag \p_{\beta, J} \left( \exists u \in B_{2n}(\mz), y\in \Z^d : u\sim y, \|u-y\|_\infty \geq n^\eta  \right) 
		\leq (4n+1)^d \sum_{y : \|y\|_\infty \geq n^\eta } \p_{\beta, J}(\mz\sim y)\\
		&
		\leq C^\prime n^d \sum_{y : \|y\|_\infty \geq n^\eta } \|y\|^{-s}
		\leq C^{\prime \prime} n^d \left(n^\eta\right)^{d-s}
		\leq C^{\prime \prime} n^{d + \eta (d-s)} 
	\end{align}
	for some constants $C^\prime, C^{\prime\prime} < \infty$ that depend on $d,\beta$, and $J$. The last expression in \eqref{poly1} is polynomially small in $n$ by our assumption on $\eta$ (we assumed that $d+ \eta(d-s) < 0$). The remaining two terms can be upper bounded using a union bound over the possible values of $z\in B_n(\mz)$ or $u \in B_{2n}(\mz)$, respectively, for $n$ large enough:
	\begin{align}
		\notag \p_{\beta, J} & \left( \exists z \in B_{n}(\mz) : L_{\lfloor n^{1-\eta} \rfloor}^{\infty}(z) \text{ holds } \right)
		\leq
		(2n+1)^d \p_{\beta, J} \left(  L_{\lfloor n^{1-\eta} \rfloor}^{\infty}(\mz) \right)\\
		\label{poly2}& \overset{\eqref{eq:Lkinequ}}{\leq} (2n+1)^d \exp \left(-\sqrt{ \lfloor n^{1-\eta} \rfloor }\right), \text{ and } 
		\\
		\notag \p_{\beta, J} & \left( \exists u \in B_{2n}(\mz) : \lfloor n^{1-\eta} \rfloor^{\frac{1}{4d}} \leq |K_u(\omega_{\leq N})| < \infty \right) 
		\leq 
		(4n+1)^d \p_{\beta, J_N}  \left( \lfloor n^{1-\eta} \rfloor^{\frac{1}{4d}} \leq |K_{\mz}| < \infty \right)
		\\
		& \label{poly3}
		\overset{\eqref{eq:kesten zhang}}{\leq}
		(4n+1)^d
		\exp \left(-c_\beta \left(\lfloor n^{1-\eta} \rfloor^{\frac{1}{4d}}\right)^{\frac{1}{d} } \right) .
	\end{align}
	Both quantities in the above expressions are (sub)polynomially small in $n$. So inserting inequalities \eqref{poly1}, \eqref{poly2}, and \eqref{poly3} into \eqref{poly uni}, we see that
	\begin{equation}\label{eq:kappabar}
		\p_{\beta,J} \left( \exists x \in \mathcal{C}_\infty(\omega) \cap B_n(\mz): D \left( x , \mathcal{C}_\infty(\omega_{\leq N}) \cap B_{2n}(\mz)  \right) > n  \right) \leq n^{-\bar{\kappa}}
	\end{equation}
	for some $\bar{\kappa} > 0$ and all $n$ large enough. \\
	
	If there exist $u,v \in B_n(\mz) \cap \cC_\infty(\omega)$ such that $D(u,v) > (C+2)n$, then either 
	$D \left( u , \mathcal{C}_\infty(\omega_{\leq N}) \cap B_{2n}(\mz)  \right) > n$, $D \left( v , \mathcal{C}_\infty(\omega_{\leq N}) \cap B_{2n}(\mz)  \right) > n$, or there are $x,y \in \mathcal{C}_\infty(\omega_{\leq N}) \cap B_{2n}(\mz)$ with $D(x,y; \omega_{\leq N}) > Cn$. Thus we get that for some $C$ large enough
	\begin{align*}
		& \p_{\beta, J} \Big( \exists u,v \in \mathcal{C}_\infty(\omega) \cap B_n(\mz): D \left( u,v \right) > (C+2) n \Big)\\
		& \hspace{10mm}
		\leq 
		\p_{\beta,J} \left( \exists x \in \mathcal{C}_\infty(\omega) \cap B_n(\mz): D \left( x , \mathcal{C}_\infty(\omega_{\leq N}) \cap B_{2n}(\mz)  \right) > n  \right) \\
		& \ \
		\hspace{45mm}+ 
		\p_{\beta, J} \Big( \exists x,y \in \mathcal{C}_\infty(\omega_{\leq N}) \cap B_{2n}(\mz): D \left( x , y; \omega_{\leq N} \right) > C n \Big)\\
		& \hspace{10mm}
		\overset{\eqref{eq:kappabar}}{\leq} 
		n^{-\bar{\kappa}} + \sum_{x,y \in B_{2n}(\mz)} 
		\p_{\beta,J} \Big( \infty > D \left( x , y; \omega_{\leq N} \right) > C n \Big)
		\\
		&
		\hspace{10mm} \overset{\eqref{eq:theret}}{\leq} n^{-\bar{\kappa}} + \left((4n+1)^d\right)^2 e^{-cn} \leq  n^{-\kappa},
	\end{align*}
	where $c>0$ and $\kappa > 0$ are positive constants, and where the last two inequalities hold for $n\in \N$ large enough.
\end{proof}

Finally, we go to the proof of Proposition \ref{propo:zeta}.

\begin{proof}[Proof of Proposition \ref{propo:zeta}]
	Fix $N\in \N$ such that the kernel $J_N$ defined by $J_N(x) = J(x)\mathbbm{1}_{\|x\|\leq N}$ satisfies $\beta> \beta_c(J_N)$. Such a $N\in \N$ exists by Theorem \ref{theo:main}.
	Let $x,y \in B_{n^\zeta}(\mz)$ be arbitrary, where $\zeta = \frac{1}{8d}$. First, explore the two clusters $K_x \left( \omega_{\leq n^{3/4}} \right)$ and $K_y \left( \omega_{\leq n^{3/4}} \right)$. We will now argue that the two inequalities
	\begin{align}\label{eq:argue1}
		& \p_{\beta,J} \left( n < D(x,y) < \infty, \ |K_x \left( \omega_{\leq n^{3/4}} \right)|\leq n^\zeta \text{ or } |K_y \left( \omega_{\leq n^{3/4}} \right)|\leq n^\zeta \right) \leq n^{\zeta-1.5}, \\
		&
		\label{eq:argue2}
		\p_{\beta,J} \left( n < D(x,y) < \infty, \ |K_x \left( \omega_{\leq n^{3/4}} \right)|> n^\zeta, |K_y \left( \omega_{\leq n^{3/4}} \right)|> n^\zeta \right) \leq n^{\zeta-1.5}
	\end{align}
	hold for all large enough $n\in \N$ and arbitrary $x,y \in B_{n^{\zeta}}(\mz)$. Note that the two probabilities in \eqref{eq:argue1} and \eqref{eq:argue2} add up to $\p_{\beta,J} \left( n < D(x,y) < \infty \right)$.
	Together with a union bound, inequalities \eqref{eq:argue1} and \eqref{eq:argue2} imply that
	\begin{align*}
		& \p_{\beta, J} \left( \exists x,y \in B_{n^\zeta}(\mz) : n < D(x,y) < \infty  \right)
		\leq 
		\sum_{ x,y \in B_{n^\zeta}(\mz)}
		\p_{\beta, J} \left(  n < D(x,y) < \infty  \right)\\
		&
		\leq \sum_{ x,y \in B_{n^\zeta}(\mz)} 2 n^{\zeta-1.5} \leq C n^{2d\zeta+\zeta-1.5} = C n^{\frac{1}{4} + \frac{1}{8d}-1.5} \leq n^{-1.1}
	\end{align*}
	for some constant $C<\infty$ and all $n\in \N$ large enough. Thus we are left to show that \eqref{eq:argue1} and \eqref{eq:argue2} hold. We start with \eqref{eq:argue1}. If $x \in K_y(\omega_{\leq n^{3/4}})$ and $|K_y(\omega_{\leq n^{3/4}})| \leq n^\zeta$, there is nothing to show as the chemical distance between $x$ and $y$ can be at most $n^\zeta$ in this case. Using this observation and the symmetry between $x$ and $y$ we get that
	\begin{align}
		\notag & \p_{\beta,J} \left( n < D(x,y) < \infty, \ |K_x \left( \omega_{\leq n^{3/4}} \right)|\leq n^\zeta \text{ or } |K_y \left( \omega_{\leq n^{3/4}} \right)|\leq n^\zeta \right) 
		\\
		\notag & \hspace{22mm}
		\leq 2 \p_{\beta,J} \left( n < D(x,y) < \infty, |K_y \left( \omega_{\leq n^{3/4}} \right)|\leq n^\zeta \right) \\
		&\notag  \hspace{22mm}
		= 2 \p_{\beta,J} \left( n < D(x,y) < \infty, |K_y \left( \omega_{\leq n^{3/4}} \right)|\leq n^\zeta, x \notin K_y \left( \omega_{\leq n^{3/4}} \right) \right) \\
		\label{eq:two step}& \hspace{22mm}
		\leq 2 \p_{\beta,J} \left( D(x,y) < \infty, |K_y \left( \omega_{\leq n^{3/4}} \right)|\leq n^\zeta, x \notin K_y \left( \omega_{\leq n^{3/4}} \right) \right) .
	\end{align}
	If $ D(x,y) < \infty$, but $x \notin K_y \left( \omega_{\leq n^{3/4}} \right)$, there needs to exist at least one open edge $\{u,v\}$ with $\|u-v\|_\infty > n^{3/4}$ and $u \in K_y(\omega_{\leq n^{3/4}}), v \notin K_y(\omega_{\leq n^{3/4}})$. As these edges are independent from the configuration $K_y \left( \omega_{\leq n^{3/4}} \right)$, we get that
	\begin{align*}
		& \p_{\beta,J} \left( D(x,y) < \infty, |K_y \left( \omega_{\leq n^{3/4}} \right)|\leq n^\zeta, x \notin K_y \left( \omega_{\leq n^{3/4}} \right) \right) \\
		& \leq \p_{\beta,J} \left( |K_y \left( \omega_{\leq n^{3/4}} \right)|\leq n^\zeta, \omega(\{u,v\})=1 \text{ for some } u\in K_y \left( \omega_{\leq n^{3/4}} \right) \text{ and } v\notin B_{n^{3/4}}(u) \right) \\
		& \leq n^\zeta \p_{\beta,J} \left( \omega(\{\mz,v\})=1 \text{ for some } v\notin B_{n^{3/4}}(\mz) \right)
		\leq  n^\zeta \sum_{v\notin B_{n^{3/4}}(\mz)} \p_{\beta,J} \left(\mz \sim v \right)\\
		&
		\leq  n^\zeta \sum_{v\notin B_{n^{3/4}}(\mz)} C_1 \|v\|^{-s}
		\leq C_2 n^\zeta n^{\frac{3}{4}(d-s)} \leq \frac{1}{2} n^{\zeta-1.5}
	\end{align*}
	for some constants $C_1,C_2$ and $n\in \N$ large enough. In the last inequality, we used that $d-s<-d\leq -2$. Plugging this into \eqref{eq:two step} finishes the proof of \eqref{eq:argue1}.\\
	
	\noindent
	Next, we prove \eqref{eq:argue2}. As in Notation \ref{notation}, we define sets $\left(S_{j}\left(x,\omega_{\leq n^{3/4}}\right)\right)_{j=0}^{\lfloor n^{\zeta} \rfloor}$ by $S_0\left(x,\omega_{\leq n^{3/4}}\right) = \{x\}$ and
	\begin{equation*}
		S_{j+1}\left(x,\omega_{\leq n^{3/4}}\right) = B_{j+1}\left(x,\omega_{\leq n^{3/4}}\right) \setminus B_{j}\left(x,\omega_{\leq n^{3/4}}\right).
	\end{equation*}
	Analogously, we define the sets $\left(S_{j}(y , \omega_{\leq n^{3/4}} )\right)_{j=0}^{\lfloor n^{\zeta} \rfloor}$.
	We only explore these spheres for $\lfloor n^{\zeta} \rfloor$ many steps to ensure that all these sets are contained in $B_{n^{0.9}}(\mz)$, where the exponent $0.9$ is arbitrary and any other exponent close enough to $1$ would also work. Note that 
	\begin{equation*}
		\text{$|K_x(\omega_{\leq n^{3/4}})| > n^\zeta$ if and only if }
		\left| \bigcup_{i=0}^{\lfloor n^{\zeta} \rfloor} S_i\left(x,\omega_{\leq n^{3/4}}\right) \right| > n^\zeta.
	\end{equation*}
	Define the events $\tilde{L}(x) $ and $\widetilde{Q}(x) $ by
	\begin{align}\label{eq:Lktildedefin}
		& \tilde{L}(x) \coloneqq \left\{\left|  B_{ n^{\zeta}}\left(x,\omega_{\leq n^{3/4}}\right) \right| > n^\zeta , \cC_\infty(\omega_{\leq N}) \cap   B_{ n^{\zeta}}\left(x,\omega_{\leq n^{3/4}}\right) = \emptyset \right\} \text{ and }
		\\
		&\notag
		\widetilde{Q}(x) \coloneqq
		\left\{ \left| B_{n^{\zeta}}\left(x,\omega_{\leq n^{3/4}}\right) \right|> n^{\zeta},  \nexists Z \in \Delta(\omega_{\leq N}) \text{ s.t. } Z \subset B_{n^{\zeta}}\left(x,\omega_{\leq n^{3/4}}\right) \text{ and } |Z| \geq  n^{\frac{\zeta}{4d}} \right\}
	\end{align}
	Claim \ref{claim:Lkdefin} applied with $r=n^{\frac{3}{4}}, k=n^{\zeta}$ shows that
	\begin{align}\label{eq:sqrtzeta}
		\notag & \p_{\beta,J}\Big( |B_{n^{\zeta}}\left(x,\omega_{\leq n^{3/4}}\right)|> n^{\zeta},  \nexists Z \in \Delta(\omega_{\leq N}) \text{ s.t. } Z \subset B_{n^{\zeta}}\left(x,\omega_{\leq n^{3/4}}\right) \text{ and } |Z| \geq  n^{\frac{\zeta}{4d}}  \Big) \\
		& \ \ = \p_{\beta,J} \left(\widetilde{Q}(x)\right) \leq e^{-\sqrt{n^{\zeta}}}
	\end{align}
	for all $n\in \N$ large enough. Assume that the event $\tilde{L}(x)$ holds but the event $\widetilde{Q}(x)$ does not hold. Let $Z \in \Delta \left(\omega_{\leq N}\right)$ be such that $Z \subset B_{n^{\zeta}}\left(x,\omega_{\leq n^{3/4}}\right)$ and $|Z| \geq n^{\frac{\zeta}{4d}}$. Note that the conditions $Z \subset B_{n^{\zeta}}\left(x,\omega_{\leq n^{3/4}}\right)$ and $x\in B_{n^\zeta}(\mz)$ already imply that $Z \subset B_n(\mz)$. Let $z\in Z$. As we assumed that the event $\tilde{L}(x)$ holds and $z \in B_{n^\zeta}\left(x, \omega_{\leq n^{3/4}}\right)$, we get that $z \notin \mathcal{C}_\infty(\omega_{\leq N})$ and thus
	\begin{equation*}
		\infty > \left|K_z(\omega_{\leq N})\right| \geq |Z| \geq n^{\frac{\zeta}{4d}} .
	\end{equation*}
	So in particular, we see that on the event $\tilde{L}(x) \cap \widetilde{Q}(x)^c$, there needs to exist $z\in B_n(\mz)$ for which $\infty > \left|K_z(\omega_{\leq N})\right| \geq n^{\frac{\zeta}{4d}}$. We thus get that
	\begin{align}\label{eq:gammaeq}
		\notag \p_{\beta, J}(\tilde{L}(x)) & \leq   
		\p_{\beta, J}(\widetilde{Q}(x)) + \p_{\beta,J}\left( \tilde{L}(x) \cap \widetilde{Q}(x)^c  \right)
		\\
		& \notag \leq   
		\p_{\beta, J}(\widetilde{Q}(x)) + \sum_{z\in B_n(\mz)} \p_{\beta,J} \left( n^{\frac{\zeta}{4d}} < |K_z(\omega_{\leq N})| < \infty \right)
		\\
		& \notag 
		\overset{\eqref{eq:sqrtzeta}}{\leq} e^{-\sqrt{n^{\zeta}}} + \sum_{z\in B_n(\mz)} \p_{\beta,J} \left( n^{\frac{\zeta}{4d}} < |K_z(\omega_{\leq N})| < \infty \right) \\
		&
		\overset{\eqref{eq:kesten zhang}}{\leq} e^{-\sqrt{n^{\zeta}}} + (2n+1)^d \exp\left(-c_\beta n^{\frac{\zeta}{4d} \cdot \frac{1}{d}}\right)  \leq  \exp\left(-n^\gamma\right)
	\end{align}
	for some $\gamma > 0$ and all large enough $n$. Similarly, we also get that $\p_{\beta, J}(\tilde{L}(y)) \leq \exp\left( - n^\gamma \right)$, where the event $\tilde{L}(y)$ is defined as in \eqref{eq:Lktildedefin}, with $x$ replaced by $y$. If both events $\tilde{L}(x)$ and $\tilde{L}(y)$ do not hold, then both $x$ and $y$ are connected by open paths of length at most $n^\zeta$ in the environment $\omega_{\leq n^{3/4}}$ to points $u_x \in \cC_\infty\left( \omega_{\leq N} \right)$ and $u_y \in \cC_\infty\left( \omega_{\leq N} \right)$, respectively. Furthermore, for $n$ large enough, we can choose the points $u_x, u_y$ such that $u_x,u_y \in  B_{n^{0.9}}(\mz)$. This holds, as the paths between $x$ and $u_x$, or between $y$ and $u_y$, respectively, have a length of at most $n^\zeta \leq n^{1/8}$, and the edges in the path have a length of at most $n^{3/4}$. In particular, we see that if both events $\tilde{L}(x)$ and $\tilde{L}(y)$ do not hold, but $|K_x \left( \omega_{\leq n^{3/4}} \right)|> n^\zeta, |K_y \left( \omega_{\leq n^{3/4}} \right)|> n^\zeta$, and $n<D(x,y)<\infty$, then there exist $u_x, u_y \in B_{n^{0.9}(\mz)} \cap \cC_\infty (\omega_{\leq N})$, such that $D(u_x,u_y) > n-2n^\zeta$. Thus we get that
	\begin{align*}
		& \p_{\beta,J} \left( n < D(x,y) < \infty, \ |K_x \left( \omega_{\leq n^{3/4}} \right)|> n^\zeta, |K_y \left( \omega_{\leq n^{3/4}} \right)|> n^\zeta \right)\\
		& \hspace{12mm}
		\leq
		\p_{\beta,J} \left( n < D(x,y) , \ |K_x \left( \omega_{\leq n^{3/4}} \right)|> n^\zeta, |K_y \left( \omega_{\leq n^{3/4}} \right)|> n^\zeta \right)\\
		& \hspace{12mm} \leq \p_{\beta,J} \left( \tilde{L}(x) \cup \tilde{L}(y) \cup \left\{\exists u_x,u_y \in  B_{n^{0.9}}(\mz) \cap \cC_\infty \left(\omega_{\leq N}\right) : D(u_x,u_y) > n-2n^\zeta \right\} \right)\\
		& \hspace{12mm}
		\overset{\eqref{eq:gammaeq}}{\leq} 2 \exp\left(-n^\gamma\right) + \sum _{u_x,u_y \in  B_{n^{0.9}}(\mz) } \p_{\beta,J} \left( \infty >  D(u_x,u_y;\omega_{\leq N}) > n-2n^\zeta  \right)\\
		& \hspace{12mm}
		\overset{\eqref{eq:theret}}{\leq } 2 \exp\left(-n^\gamma\right) + (2n^{0.9}+1)^d C e^{-cn} 
	\end{align*}
	for some constants $c>0$ and $C<\infty$, where the last two inequalities hold for all large enough $n$. So in particular this implies that \eqref{eq:argue2} holds.
\end{proof}

\section{Applications of Theorem \ref{theo:main}}\label{sec:applis}

Next, we discuss various applications of Theorem \ref{theo:main}.

\subsection{Locality of long-range percolation}

We start with the proof of Theorem \ref{theo:locality}. Note that Theorem \ref{theo:main} also provides a locality result of the type $\beta_c(J_n) \to \beta_c(J)$ if one defines the kernels $J_n$ by
\begin{equation*}
	J_n(x)= \begin{cases}
		J(x) & \text{ if } \|x\|_\infty \leq n \\
		0 & \text{ else }
	\end{cases}.
\end{equation*}

\begin{proof}[Proof of Theorem \ref{theo:locality}]
	Let $J_n \to J$ in $L_1$ on $\Z^d$. In order to show the result, we will show that
	\begin{align}
		&\label{eq:liminfconv} \liminf_{n \to \infty} \beta_c(J_n) \geq \beta_c(J) \text{ and } \\
		&\label{eq:limsupconv} \limsup_{n \to \infty} \beta_c(J_n) \leq \beta_c(J).
	\end{align}
	We need to show both inequalities for the case where $J_n$ converges to a resilient kernel $J$, and for the case where $J_n$ converges to a general kernel from above.
	
	We start with the proof of \eqref{eq:liminfconv}; here we do not make a distinction whether the kernel $J$ is resilient or not. Let $\beta < \beta_c(J)$. By the proof of sharpness of the phase transition by Duminil-Copin and Tassion \cite{duminil2016new,duminil2017new}, there exists a finite set $S \subset \Z^d$ such that $\mz \in S \subset \Z^d$ and
	\begin{equation}\label{eq:duminilcopin tassion}
		\phi_{\beta,J} \left( S \right) \coloneqq
		\sum_{x \in S} \sum_{y \notin S} \p_{\beta,J} \left( \mz \overset{S}{\longleftrightarrow} x\right)
		\left(1-e^{-\beta J(x-y)}\right) 
		<  1 .
	\end{equation}
	As $J_n \to J$ pointwise and $S$ is a finite set, this implies that $\lim_{n \to \infty}  \p_{\beta,J_n} \left( \mz \overset{S}{\longleftrightarrow} x\right) =  \p_{\beta,J} \left( \mz \overset{S}{\longleftrightarrow} x\right)$ for all $x\in S$.
	As $J_n$ converges to $J$ in $L_1$ of $\Z^d$, we also get that for each $x\in S$, the sum $\sum_{y \notin S} 	\left(1-e^{-\beta J_n(x-y)}\right)$ converges to $\sum_{y \notin S} 	\left(1-e^{-\beta J(x-y)}\right)$. Thus we also get that $\lim_{n \to \infty} \phi_{\beta,J_n} \left( S \right) = \phi_{\beta,J} \left( S \right)$. So in particular, by \eqref{eq:duminilcopin tassion}, one has
	\begin{equation*}
		\phi_{\beta,J_n} \left( S \right) =
		\sum_{x \in S} \sum_{y \notin S} \p_{\beta,J_n} \left( \mz \overset{S}{\longleftrightarrow} x\right)
		\left(1-e^{-\beta J_n(x-y)}\right) 
		<  1 
	\end{equation*}
	for all large enough $n$. The condition $\phi_{\beta,J_n} \left( S \right) < 1$ implies that $\beta \leq \beta_c(J_n)$ \cite{duminil2016new,duminil2017new} and thus $\beta \leq \liminf_{n \to \infty} \beta_c(J_n) $. As $\beta < \beta_c(J)$ was arbitrary, this finishes the proof of \eqref{eq:liminfconv}. \\
	
	We still need to show that \eqref{eq:limsupconv} holds. We start with the case where $J$ is a general kernel and $J_n$ converges to $J$ from above. If $J_n \geq J$, then $\beta_c(J_n) \leq \beta_c(J)$, which implies that $\eqref{eq:limsupconv}$ holds for this case.\\
	
	Next, let $J$ be a resilient kernel and let $(J_n)_{n\in \N}$ be a sequence of kernels such that $J_n$ converges to $J$ in $L_1$, not necessarily from above. Let $\eps > 0$. Take $N \in \N$ large enough so that the kernel $\tilde{J}$ defined by
	\begin{equation*}
		\tilde{J}(x)= \begin{cases}
			J(x) & \text{ if } \|x\|_\infty \leq N\\
			0 & \text{ else }
		\end{cases}
	\end{equation*}
	satisfies $\beta_c(\tilde{J}) < \beta_c(J) + \eps$. Such $N\in \N$ exists as $J$ is a resilient kernel. Define the kernel $\tilde{J_n}$ by
	\begin{equation*}
		\tilde{J}_n(x)= \begin{cases}
			J_n(x) & \text{ if } \|x\|_\infty \leq N\\
			0 & \text{ else }
		\end{cases}.
	\end{equation*}
	Then $\tilde{J_n}\to \tilde{J}$ pointwise. By construction one also has $\tilde{J_n} \leq J_n$ and thus $\beta_c(J_n) \leq \beta_c(\tilde{J_n})$. As $\tilde{J}$ has finite range and $\tilde{J}_n \to \tilde{J}$ pointwise, we have, for every $\delta \in (0,1)$, that $\tilde{J}_n(x)\geq (1-\delta)\tilde{J}(x)$ for all $x \in \Z^d$, for all large enough $n \in \N$. Thus, for such large enough $n \in \N$ we get that $\p_{\beta, \tilde{J_n}}$ stochastically dominates $\p_{\beta, (1-\delta)\tilde{J}} = \p_{(1-\delta)\beta,\tilde{J}}$, which directly implies that $\beta_c(\tilde{J}_n) \leq \frac{1}{1-\delta}\beta_c(\tilde{J})$ for $n\in \N$ large enough.
	Combining the previous inequalities we get for $n$ large enough that
	\begin{equation*}
		\beta_c(J_n) \leq \beta_c(\tilde{J_n}) \leq \frac{1}{1-\delta}\beta_c(\tilde{J}) < \frac{1}{1-\delta} \left(\beta_c(J) + \eps\right) ,
	\end{equation*}
	which implies that $\limsup_{n \to \infty} \beta_c(J_n) \leq \frac{1}{1-\delta} \left(\beta_c(J) + \eps\right)$. As $\eps > 0$ and $\delta \in (0,1)$ were arbitrary, this finishes the proof.
\end{proof}

\begin{remark}
	Note that the proof of \eqref{eq:liminfconv} used the $L_1$-convergence of the kernel $J_n$, but did not use any other property of the limiting kernel $J$. Contrary to that, the proof of \eqref{eq:limsupconv} used pointwise convergence of the kernels only, but also required resilience of the kernel $J$, or that the approximating sequence converges from above.
\end{remark}

\subsection{Continuity of the percolation probability outside criticality}

Next, we go to the proof of Corollary \ref{coro:percoprobconv}. Note that the restriction to $\beta \neq \beta_c(J)$ in the statement is essential. Indeed, proving the result of Corollary \ref{coro:percoprobconv} for $\beta=\beta_c(J)$ would imply continuity of the percolation phase transition, which is an important open problem even for finite-range percolation in intermediate dimensions.
A similar statement to that of Corollary \ref{coro:percoprobconv} for long-range percolation with exponentially decaying connection probabilities was already proven by Meester and Steif in \cite[Theorem 1.4]{meester1996continuity} and we follow a similar strategy of the proof as they did. 

\begin{proof}[Proof of Corollary \ref{coro:percoprobconv}]
	In order to show the corollary, we need to show that
	\begin{align}
		&\label{eq:limsup} \limsup_{n\to \infty} \theta \left(\beta_n, J_n\right) \leq \theta \left(\beta,J\right) \text{ and }\\
		&\label{eq:liminf} \liminf_{n\to \infty} \theta \left(\beta_n, J_n\right) \geq \theta \left(\beta,J\right).
	\end{align}
	We start with the proof of \eqref{eq:limsup}. Let $\eps > 0$. As the number of finite subsets of $\Z^d$ is countable, there exists a finite collection of different finite sets $\left(A_i\right)_{i\in \{1,\ldots,m\}}$ such that $\mz \in A_i \subset \Z^d$ for all $i\in \{1,\ldots,m\}$ and 
	\begin{equation*}
		\p_{\beta,J} \left(K_\mz \in \{A_1,\ldots,A_m\}\right) = \sum_{i=1}^{m} \p_{\beta,J} \left(K_\mz = A_i\right) \geq 1-\theta(\beta,J)-\eps.
	\end{equation*}
	As $A_i$ is a finite set, $\beta_n \to \beta$, and $J_n \to J$ in $L_1$, we get that
	\begin{equation*}
		\lim_{n\to \infty} \sum_{i=1}^{m} \p_{\beta_n,J_n} \left(K_\mz = A_i\right) = \sum_{i=1}^{m} \p_{\beta,J} \left(K_\mz = A_i\right) \geq 1-\theta(\beta,J)-\eps,
	\end{equation*}
	so in particular
	\begin{equation*}
		 \sum_{i=1}^{m} \p_{\beta_n,J_n} \left(K_\mz = A_i\right) \geq  1-\theta(\beta,J)-2\eps
	\end{equation*}
	for all $n$ large enough and thus also
	\begin{equation*}
		\p_{\beta_n,J_n } \left(|K_\mz|=\infty\right) \leq 1 - \sum_{i=1}^{m} \p_{\beta_n,J_n} \left(K_\mz = A_i\right) \leq \theta(\beta,J) + 2 \eps
	\end{equation*}
	for all $n$ large enough, which finishes the proof of \eqref{eq:limsup}, as $\eps > 0$ was arbitrary.
	
	Next, let us prove \eqref{eq:liminf}. We first assume that $\beta < \beta_c(J)$. As $\beta_n \to \beta$ and $\beta_c(J_n) \to \beta_c(J)$ for $n\to \infty$ (by Theorem \ref{theo:locality}), we have that $\beta_n < \beta_c(J_n)$ for all $n$ large enough. So in particular
	\begin{equation*}
		\theta(\beta_n,J_n) = 0 = \theta(\beta,J)
	\end{equation*}
	for all large enough $n\in \N$. Next, let us turn to the case $\beta > \beta_c(J)$. For $N\in \N$, define the kernel $I_N$ by $I_N(x) = J(x)\mathbbm{1}_{\{\|x\|_\infty\leq N\}}$. For a percolation environment $\omega \in \{0,1\}^E$, we define $\omega_{\leq N} \in \{0,1\}^E$ by
	\begin{align*}
		\omega_{\leq N}(e) = \begin{cases}
			\omega (e) & \text{ if } |e|\leq N\\
			0 & \text{ if } |e|>N
		\end{cases}.
	\end{align*}
	 As $J$ is a resilient kernel by assumption, we know from Theorem \ref{theo:locality} that $\beta_c(I_N) \to \beta_c(J)$ as $N\to \infty$, so in particular we can fix $M \in \N$ large enough so that $\beta > \beta_c(I_M)$. We couple the measures $\left(\p_{\beta,I_k}\right)_{k\in \N}$ for different values of $k \in \N$ using the Harris coupling (see e.g. \cite{heydenreich2017progress}), and we write  $\mathcal{C}_{\infty}(\omega_{\leq M})$ for the (almost surely unique) infinite cluster sampled by $\p_{\beta,I_M}$.
	We write $\mz \leftrightarrow \mathcal{C}_{\infty}(\omega_{\leq M})$ if the origin is connected to the infinite finite-range percolation cluster and we write $\mz \overset{\leq N}{\longleftrightarrow} \mathcal{C}_{\infty}(\omega_{\leq M})$ if the origin is connected to $\mathcal{C}_{\infty}(\omega_{\leq M})$ using only edges $\{x,y\}$ with $\|x-y\|_\infty \leq N$.
	Note that the almost sure uniqueness of the infinite cluster implies that the events $\left\{\mz \leftrightarrow \mathcal{C}_{\infty}(\omega_{\leq M})\right\}$ and $\left\{\mz \leftrightarrow \mathcal{C}_{\infty}(\omega)\right\}$ are almost surely identical and that for $N\geq M$ also the two events $\left\{ \mz \overset{\leq N}{\longleftrightarrow} \mathcal{C}_{\infty}(\omega_{\leq N})\right\}$ and $\left\{ \mz \overset{\leq N}{\longleftrightarrow} \mathcal{C}_{\infty}(\omega_{\leq M})\right\}$ are almost surely the same. Thus we get that
	\begin{align*}
		\theta(\beta,J) - \theta(\beta,I_N) & = \p\left( \left\{\mz \leftrightarrow \mathcal{C}_{\infty}(\omega)\right\} \cap \left\{ \mz \overset{\leq N}{\longleftrightarrow} \mathcal{C}_{\infty}(\omega_{\leq N})\right\}^c \right)
		\\
		&
		 = \p\left( \left\{\mz \leftrightarrow \mathcal{C}_{\infty}(\omega_{\leq M})\right\} \cap \left\{ \mz \overset{\leq N}{\longleftrightarrow} \mathcal{C}_{\infty}(\omega_{\leq M})\right\}^c \right)
	\end{align*}
	which converges to $0$ as $N \to \infty$ by the uniqueness of the infinite open cluster.
	So for each $\eps > 0$ we can find $N \geq M$ large enough so that
	\begin{equation*}
		\theta \left(\beta, I_N\right) \geq \theta(\beta,J) - \eps \ \text{ and } \ \beta > \beta_c(I_N)
	\end{equation*}
	The function  $\tilde{\beta} \mapsto \theta \left(\tilde{\beta}, I_N\right)$ is continuous at $\tilde{\beta} = \beta$, since $\beta > \beta_c(I_M) \geq \beta_c(I_N)$, see \cite[Lemma 8.10]{grimmett1999percolation}. Thus we can pick $\delta > 0$ small enough so that 
	\begin{equation*}
		\theta \left(\beta - \delta, I_N\right) \geq \theta(\beta,J) - 2\eps.
	\end{equation*}
	As $\beta_n \to \beta$ and $J_n \to J$ in $L_1$ (and thus also pointwise) this implies that $\beta_n J_n(x) \geq (\beta-\delta) I_N(x)$ for all $n\in \N$ large enough and $x\in \Z^d$ (Remember that $I_N$ has finite range). Thus one has for all edges $\{x,y\} \subset \Z^d$ and all $n$ large enough that $\p_{\beta_n, J_n} \left(\{x,y\} \text{ open}\right) \geq \p_{(\beta-\delta),I_N} \left(\{x,y\} \text{ open}\right)$. As different edges are independent, this pointwise bound already implies the corresponding stochastic dominance for the percolation measures, i.e., $\p_{\beta_n, J_n} \gtrsim \p_{(\beta-\delta),I_N}$ for all large enough $n$. For such sufficiently large $n$, we thus get by the stochastic domination that
	\begin{equation*}
		\theta(\beta_n,J_n) \geq \theta(\beta-\delta, I_N) \geq \theta(\beta,J) -2\eps,
	\end{equation*}
	which finishes the proof as $\eps > 0$ was arbitrary.
\end{proof}

\begin{remark}
	Note that the proof of \eqref{eq:limsup} did not use any previous results and holds without any further assumptions on the kernel $J$. Contrary to that, the proof of inequality \eqref{eq:liminf} heavily uses the resilience of the kernel $J$ and it can be easily seen that inequality \eqref{eq:liminf} does not hold in dimension $d=1$. However, the proof of \eqref{eq:liminf} does not use the $L_1$-convergence of $J_n$ to $J$, but requires pointwise convergence only.
\end{remark}

\subsection{Existence of large clusters}

For the proof of Theorem \ref{theo:large clusters}, we need the following claim for finite-range percolation. It says that with high probability all points $x,y$ in the infinite cluster of a box are connected in a slightly bigger box.

\begin{claim}\label{claim:connection in boxes}
	Let $J:\Z^d \to \left[0,\infty\right)$ be an irreducible and symmetric kernel with finite range, and let $\beta > \beta_c(J)$. Then
	\begin{align*}
		\lim_{n\to \infty} \p_{\beta,J} \left(\forall x,y \in \mathcal{C}_\infty \cap B_{n-\sqrt{n}}(\mz) : x \overset{B_n(\mz)}{\longleftrightarrow} y\right) = 1 .
	\end{align*}
\end{claim}

We will prove this result later; let us first see how it implies Theorem \ref{theo:large clusters}.

\begin{proof}[Proof of Theorem \ref{theo:large clusters} given Claim \ref{claim:connection in boxes}]
	By Theorem \ref{theo:main} we know that we can pick $N\in \N$ large enough so that the kernel $\tilde{J}$ defined by $\tilde{J}(x)=J(x)\mathbbm{1}_{\|x\|_\infty\leq N}$ satisfies $\beta > \beta_c(\tilde{J})$. So in particular there almost surely exists an infinite open cluster using only short edges. Define this cluster as
	\begin{align*}
		K = \left\{x \in \Z^d : x \overset{\leq N}{\longleftrightarrow} \infty  \right\}.
	\end{align*}
	The set $K$ is a random set and a subset of the infinite cluster $\mathcal{C}_\infty$ whose distribution is invariant under translations.
	For a point $y \in \Z^d$, define the random variable $Z_y$ as the distance to the set $K$:
	\begin{align*}
		Z_y= \inf \Big\{m \geq 0 : & \text{ There exist } y_0,y_1,\ldots,y_k \subset B_m(y) \text{ s.t. }\\
		& \ \  \{y_i,y_{i+1}\} \text{ open for $i=0,\ldots,k-1$}, \ y_0 =y, \text{ and } y_k\in K   \Big\} .
	\end{align*}
	So in particular $Z_y=0$ if and only if $y\in K$, and $Z_y < \infty$ if and only if $y\in \cC_\infty$ almost surely, by uniqueness of the infinite cluster. The probability of the event $\{\infty > Z_\mz >\sqrt{n}\}$ converges to $0$ as $n\to \infty$.
	So in particular, by stationarity, 
	\begin{equation*}
		|B_n(\mz)|^{-1} \sum_{y\in B_n(\mz) \cap \mathcal{C}_\infty} \mathbbm{1}_{\{Z_y > \sqrt{n}\}}
	\end{equation*}
	converges to $0$ in expectation, and thus also in probability, as $n\to \infty$. Now fix $\eps > 0$. Assume that
	\begin{align}
		\label{eq:assume1} & |\{x \in B_{n-3\sqrt{n}}(\mz) : x \in \mathcal{C}_\infty\}| \geq (\theta(\beta,J)-\eps) |B_{n-3\sqrt{n}}(\mz)|, \\
		\label{eq:assume2} & |B_n(\mz)|^{-1} \sum_{y\in B_n(\mz) \cap \mathcal{C}_\infty} \mathbbm{1}_{Z_y > \sqrt{n}} \leq \eps, \text{ and that } \\
		\label{eq:assume3} & B_{n-\sqrt{n}}(\mz) \cap K \text{ is connected within $B_n(\mz)$}.
	\end{align}
	All these three events hold with high probability in $n$. The event $\eqref{eq:assume1}$ holds with high probability because of ergodicity, the event \eqref{eq:assume2} holds with high probability as the sum converges to $0$ in probability, and the third event \eqref{eq:assume3} holds with high probability by Claim \ref{claim:connection in boxes}. Thus all three events hold simultaneously with high probability. Let $x,y \in B_{n-3\sqrt{n}}(\mz)$ be such that $Z_x, Z_y \leq \sqrt{n}$. Then there exist $a \in B_{\sqrt{n}}(x) \subset B_{n-\sqrt{n}}(\mz)$ and $b \in B_{\sqrt{n}}(y) \subset B_{n-\sqrt{n}}(\mz)$ such that $a,b \in K$, $a$ and $x$ are connected within $B_{n-\sqrt{n}}(\mz)$, and  $b$ and $y$ are connected within $B_{n-\sqrt{n}}(\mz)$. So if the event in \eqref{eq:assume3} holds, then for all $x,y \in B_{n-3\sqrt{n}}(\mz)$ with $Z_x, Z_y \leq \sqrt{n}$  there exists a path between them that stays entirely within $B_n(\mz)$. So if all three events \eqref{eq:assume1}, \eqref{eq:assume2}, and \eqref{eq:assume3} hold simultaneously, then
	\begin{align*}
		|K_{\max}\left(B_n(\mz)\right)| & \geq |\{x \in B_{n-3\sqrt{n}}(\mz) : x \in \mathcal{C}_\infty, Z_x \leq \sqrt{n}\}| \\
		& \geq (\theta(\beta,J)-\eps) |B_{n-3\sqrt{n}}(\mz)| - \eps |B_n(\mz)| \geq (\theta(\beta,J)-3\eps) |B_{n}(\mz)|
	\end{align*}
	where the last inequality holds for $n$ large enough. This shows Theorem \ref{theo:large clusters}, as all three events \eqref{eq:assume1}, \eqref{eq:assume2}, and \eqref{eq:assume3} hold with high probability in $n$.
\end{proof}

Finally, we prove Claim \ref{claim:connection in boxes}.

\begin{proof}[Proof of Claim \ref{claim:connection in boxes}]
	Define the event $\cG_n$ by
	\begin{align*}
		\cG_n = \bigcap_{x\in B_n(\mz)} \left\{B_{n^{1/4}}(x) \cap \cC_\infty \neq \emptyset\right\} 
		\cap
		\bigcap_{x,y\in B_n(\mz) \cap \cC_\infty } \left\{D(x,y) \leq A_3 (\|x-y\| \vee n^{1/4}) \right\} .
	\end{align*}
	Note that the condition $x,y\in \cC_\infty$ implies that $x\leftrightarrow y$ by uniqueness of the infinite open cluster. So using \eqref{eq:theret} and \eqref{eq:roberto} and a union bound over all possible values of $x,y \in B_n(\mz)$ one sees that
	\begin{align*}
		&\p_\beta \left(\cG_n^c\right) \\ & \leq \sum_{x\in B_n(\mz)} \p_\beta \left(B_{n^{1/4}}(x) \cap \cC_\infty = \emptyset\right)
		+
		\sum_{x,y\in B_n(\mz) } \p_\beta \left(x,y \in \mathcal{C}_\infty, D(x,y) > A_3 (\|x-y\| \vee n^{1/4}) \right)\\
		&
		\leq
		\sum_{x\in B_n(\mz)} C e^{-\lfloor n^{1/4} \rfloor^\eta}
		+
		\sum_{x,y\in B_n(\mz) } A_1 e^{-A_2 \left(\|x-y\| \vee n^{1/4}\right) }
	\end{align*}
	and thus $\p_{\beta}\left(\cG_n\right) \geq 1-\tfrac{1}{n}$ for all large enough $n$. We finish the proof by showing that the event $\cG_n$ implies that $x\overset{B_n(\mz)}{\longleftrightarrow} y$ for all $x,y\in B_{n-\sqrt{n}}(\mz) \cap \cC_\infty$.
	Let $x,y\in B_{n-\sqrt{n}}(\mz) \cap \cC_\infty$. Then we can pick $a_0,a_1,\ldots,a_k \in B_{n-\sqrt{n}}(\mz)$ such that $\|a_i-a_{i-1}\|_{\infty} \leq n^{1/4}$ for all $i=1,\ldots,k$ and $x \in B_{n^{1/4}}(a_0), y \in B_{n^{1/4}}(a_k)$. By the definition of the event $\cG_n$, for all $i\in\{0,\ldots,k\}$ there exists $x_i \in B_{n^{1/4}}(a_i) \cap \cC_\infty$. The Euclidean distance between $x_i$ and $x_{i-1}$ is bounded by
	\begin{equation*}
		\|x_i-x_{i-1}\|\leq \|x_i-a_i\| + \|a_i-a_{i-1}\| + \|a_{i-1}-x_{i-1}\| \leq 3d n^{1/4}
	\end{equation*}
	and thus the graph distance between $x_{i}$ and $x_{i-1}$ is bounded by $A_3 3d n^{1/4}$, by the definition of $\cG_n$. The same holds for the graph distance between $x$ and $x_0$ and the graph distance between $x_k$ and $y$. As $J$ is a kernel with finite range and $a_{i-1}, a_{i} \in B_{n-\sqrt{n}}(\mz)$, the shortest path between $x_{i-1}$ and $x_{i}$ stays inside the box $B_n(\mz)$ for large enough $n$, and the same holds for the shortest path between $x_0$ and $x$, or between $x_k$ and $y$, respectively. Thus we get that 
	\begin{equation*}
		x \overset{B_n(\mz)}{\longleftrightarrow} x_0 \overset{B_n(\mz)}{\longleftrightarrow} x_1
		\overset{B_n(\mz)}{\longleftrightarrow} \ldots 
		\overset{B_n(\mz)}{\longleftrightarrow} x_k
		\overset{B_n(\mz)}{\longleftrightarrow} y
	\end{equation*}
	which shows that $x \overset{B_n(\mz)}{\longleftrightarrow} y$. As $x,y \in B_{n-\sqrt{n}}\cap \cC_\infty$ were arbitrary, this finishes the proof.
\end{proof}

\subsection{Transience of random walks}

Next, we prove transience of the simple random walk on supercritical long-range percolation clusters in dimensions $d\geq 3$. Our main tool here is transience of the simple random walk on finite-range percolation clusters in dimensions $d\geq 3$. This was shown by Grimmett, Kesten, and Zhang for nearest-neighbor percolation \cite{grimmett1993random}. The proof for finite-range percolation works analogous and we will not pursue this here.

\begin{proof}[Proof of Theorem \ref{theo:transience}]
	Let $\beta>\beta_c(J)$ and let $N \in \N$ be large enough such that the kernel $\tilde{J}$ defined by 
	\begin{equation*}
		\tilde{J}(x)= \begin{cases}
			J(x) & \text{ if } \|x\|_\infty \leq N\\
			0 & \text{ else }
		\end{cases}
	\end{equation*}
	satisfies $\beta > \beta_c(\tilde{J})$. Such an $N\in \N$ exists, as the kernel $J$ was assumed to be resilient. We can sample the percolation configuration under the measure $\p_{\beta, J}$ by first sampling the percolation configuration under the measure $\p_{\beta, \tilde{J}}$ and then including the edges $e=\{x,y\}$ with $\|x-y\| > N$ with the corresponding probabilities. As $\tilde{J}$ is a kernel with finite range, the infinite cluster sampled by the measure $\p_{\beta, \tilde{J}}$ is almost surely transient \cite{grimmett1993random}. By Rayleigh's monotonicity principle \cite[Chapter 2.4]{lyons2017probability}, this implies that also the infinite cluster sampled by the percolation configuration $\p_{\beta, J}$ is almost surely transient.
\end{proof}

\section{Varying short edges only}\label{sec:short edges}

In this section, we prove Theorem \ref{theo:p1 larger}. The main tool for this is Proposition \ref{prop:strict ineq}, which we prove in section \ref{sec:strict}. After this, we conclude with the proof of Theorem \ref{theo:p1 larger} in section \ref{sec:p1}.

\subsection{Strict inequality of critical points}\label{sec:strict}

In this section, we prove Proposition \ref{prop:strict ineq}. 
In order to prove the strict inequality of critical points, we use the technique of enhancements developed by Aizenman and Grimmett \cite{aizenman1991strict}. The main item to prove here is the differential inequality \eqref{eq:inequality strict}.
For an integrable and translation-invariant kernel $J$ and $\beta,s \geq 0$, we define the combined measure $\p_{\beta,s,J}$ as the measure of independent bond percolation where an edge $\{x,y\}$ is open with probability
\begin{equation*}
	\p_{\beta,s,J} \left(\{x,y\} \text{ open}\right) = p(\beta,s,\{x,y\}) = 
	\begin{cases}
		1-\exp \left(-\beta J(\{x,y\}) - s \right) & \text{ if } \|x-y\| = 1\\
		1-\exp \left(-\beta J(\{x,y\})     \right) & \text{ otherwise } 
	\end{cases}.
\end{equation*}

\begin{proposition}\label{propo:strict}
	For every kernel $J$ satisfying condition \eqref{eq:comparability}, there exists a continuous function $g: \R_{>0} \times \R_{>0} \to \R_{>0}$ and $N \in \N$ such that for all $\beta,s>0$ and all $n \geq N$
	\begin{equation}\label{eq:inequality strict}
		\frac{\md}{\md \beta} \p_{\beta,s,J} \left(\mz \leftrightarrow B_n(\mz)^c \right)
		\leq g(\beta,s)
		\frac{\md}{\md s} \p_{\beta,s,J} \left(\mz \leftrightarrow B_n(\mz)^c \right).
	\end{equation}
\end{proposition}

Assuming this proposition, we can directly prove Proposition \ref{prop:strict ineq}.

\begin{proof}[Proof of Proposition \ref{prop:strict ineq} assuming Proposition \ref{propo:strict}]
	Let $J$ be a kernel and assume that $0 < \beta_c(J) < \infty$. Let $K \in \N$ be large enough so that $\frac{1}{K} < \frac{\beta_c(J)}{4}$. Let $M$ be a large enough constant so that $g(\beta,s) \leq M$ for all $\beta \in \left[\tfrac{1}{2} \beta_c(J), 2 \beta_c(J)\right]$ and $s \in \left[\frac{1}{K}, \beta_c(J)\right]$. For abbreviation, we write $\beta_c=\beta_c(J)$ in the rest of the proof. Let $\eps \in (0,0.1)$ be small enough so that
	\begin{equation*}
		2\eps M < \frac{\beta_c}{4} \text{ and } \eps < \frac{\beta_c}{4}.
	\end{equation*}
	For $r \in \left[0,2\eps\right]$ define 
	\begin{equation*}
		\text{$\beta(r) \coloneqq \beta_c + \eps - r$ and $s(r) \coloneqq \frac{1}{K} + rM$.}
	\end{equation*}
	So in particular 
	\begin{equation*}
		\left(\beta(r),s(r)\right)\in \left[\frac{\beta_c}{2},2\beta_c\right] \times \left[\frac{1}{K},\frac{1}{K}+2\eps M\right] \subset \left[\frac{\beta_c}{2},2\beta_c\right] \times \left[\frac{1}{K},\beta_c(J)\right] \text{ for all } r\in \left[0,2\eps\right],
	\end{equation*}
	and thus $g(\beta(r),s(r))\leq M$ for all $r\in \left[0,2\eps\right]$. Differentiating yields that
	\begin{align*}
		& \frac{\md}{\md r} \p_{\beta(r), s(r), J} \left(\mz \leftrightarrow B_n(\mz)^c \right)\\ & 
		=
		- \frac{\md}{\md \beta(r)} \p_{\beta(r), s(r), J} \left(\mz \leftrightarrow B_n(\mz)^c \right) + M \frac{\md}{\md s(r)} \p_{\beta(r), s(r), J} \left(\mz \leftrightarrow B_n(\mz)^c \right) 
		\\
		&
		\geq
		- \frac{\md}{\md \beta(r)} \p_{\beta(r), s(r), J} \left(\mz \leftrightarrow B_n(\mz)^c \right) + g(\beta(r),s(r)) \frac{\md}{\md s(r)} \p_{\beta(r), s(r), J} \left(\mz \leftrightarrow B_n(\mz)^c \right) 
		\geq 0
	\end{align*}
	for $n$ large enough. Thus
	\begin{align*}
		& \p_{\beta_c-\eps, \frac{1}{K} + 2\eps M, J} \left(\mz \leftrightarrow B_n(\mz)^c \right)
		=
		\p_{\beta(2\eps), s(2\eps), J} \left(\mz \leftrightarrow B_n(\mz)^c \right)
		\geq 
		\p_{\beta(0), s(0), J} \left(\mz \leftrightarrow B_n(\mz)^c \right)\\
		& \ \ = \p_{\beta_c+\eps, \frac{1}{K}, J} \left(\mz \leftrightarrow B_n(\mz)^c \right)
		\geq
		\p_{\beta_c+\eps,0, J} \left(\mz \leftrightarrow B_n(\mz)^c \right)>0
	\end{align*}
	for $n$ large enough. Taking $n\to \infty$ shows that $\p_{\beta_c-\eps, \frac{1}{K} + 2\eps M, J} \left(\mz \leftrightarrow \infty\right) > 0$. Finally, we will prove that $\p_{\beta_c-\eps, \overline{J}}$ stochastically dominates $\p_{\beta_c-\eps, \frac{1}{K} + 2\eps M, J}$, i.e., 
	\begin{equation}\label{eq:stochastic domination}
		\p_{\beta_c-\eps, \frac{1}{K} + 2\eps M, J} \lesssim \p_{\beta_c-\eps, \overline{J}} ,
	\end{equation}
	which implies that $\p_{\beta_c-\eps, \overline{J}} \left(\mz \leftrightarrow \infty\right) > 0$ and thus $\beta_{c}(\overline{J}) \leq  \beta_c - \eps < \beta_c(J)$. In order to show \eqref{eq:stochastic domination}, we just need to show that for each edge the marginal probability of being open under the measure $\p_{\beta_c-\eps, \overline{J}}$ is at least the marginal probability of being open under the measure $\p_{\beta_c-\eps, \frac{1}{K} + 2\eps M, J}$. This is clear for edges $\{x,y\}$ with $\|x-y\|>1$, since $J(x-y)=\overline{J}(x-y)$ for such edges, and thus they have the same probability of being open under both measures. For edges $\{x,y\}$ with $\|x-y\|=1$, we need to show that
	\begin{equation*}
		(\beta_c-\eps)J(\{x,y\}) + \frac{1}{K} + 2\eps M \leq (\beta_c-\eps)\overline{J}(\{x,y\}),
	\end{equation*}
	which is true, since
	\begin{align*}
		&(\beta_c-\eps)J(\{x,y\}) + \frac{1}{K} + 2\eps M \leq (\beta_c-\eps)\overline{J}(\{x,y\})
		\Leftrightarrow
		\frac{1}{K} + 2\eps M \leq \beta_c-\eps\\
		&
		\Leftrightarrow
		\frac{1}{K} + 2\eps M + \eps \leq \beta_c
	\end{align*}
	and the last line follows from the assumptions on $K$ and $\eps$, as $\frac{1}{K} , 2\eps M , \eps \leq \tfrac{\beta_c}{4}$.
\end{proof}

\begin{proof}[Proof of Proposition \ref{propo:strict}]
	For an edge $e=\{x,y\}$, we write $p(\beta,s,e)$ for the probability that this edge is open under the measure $\p_{\beta,s,J}$. We write $E_n$ for the set of edges with at least one endpoint in $B_n(\mz)$, and we write $E_n^s$ for the set of edges $\left\{\{x,y\} \in E_n : \|x-y\| = 1\right\}$, i.e., the {\sl short} edges. We define the event $A_n=\{\mz \leftrightarrow B_n(\mz)^c\}$. Using Russo's formula, respectively a straight-forward modification for long-range percolation, and applying it for the two derivatives in \eqref{eq:inequality strict}, we need to show that there exists a continuous function $g$ such that for $n$ large enough
	\begin{align}
		&\notag \sum_{e\in E_n} \p_{\beta,s,J} \left(e \text{ is pivotal for the event } A_n\right) \frac{\md}{\md \beta} p(\beta,s,e)\\
		& \hspace{15mm} \label{eq:strict g(beta,s) 1} \leq g(\beta,s)
		\sum_{e\in E_n} \p_{\beta,s,J} \left(e \text{ is pivotal for the event } A_n\right) \frac{\md}{\md s} p(\beta,s,e).
	\end{align}
	From the definition of $p(\beta,s,e)$, we see that $\frac{\md}{\md \beta} p(\beta,s,e)$ is of order $J(e)$, whereas $\frac{\md}{\md s} p(\beta,s,e)$ is $0$ for $e \notin E_n^s$ and of constant order for $e\in E_n^s$.
	Thus inequality \eqref{eq:strict g(beta,s) 1} holds, provided we can show that there exists a continuous function $\overline{g}:\R_{>0} \times \R_{>0} \to \R_{>0}$ so that
	\begin{align}\label{eq:strict g(beta,s) 2}
		\sum_{e\in E_n} J(e) \p_{\beta,s,J} \left(e \text{ is pivotal for } A_n\right) 
		\leq \overline{g}(\beta,s)
		\sum_{e\in E_n^s} \p_{\beta,s,J} \left(e \text{ is pivotal for } A_n\right).
	\end{align}
	So we need to study the probability that edges are pivotal for the event $A_n$ for both short and long edges. We write $x \geq \mz$ if all coordinates of $x$ are non-negative. For fixed $\beta>0$, the probability that an edge $e$ is open is proportional to $J(e)$. Thus there exist constants $C_1,C_2 < \infty$ depending on the kernel $J$ and, in a continuous way, on the parameters $\beta$ and $s$ so that
	\begin{align}
		& \notag \sum_{e\in E_n} J(e) \p_{\beta,s,J} \left(e \text{ is pivotal for } A_n\right) \leq C_1
		\sum_{e\in E_n} \p_{\beta,s,J} \left(e \text{ is open and pivotal for } A_n\right) \\
		& \hspace{22mm} \notag
		\leq C_1
		\sum_{x\in B_n(\mz)} \ \sum_{y \in \Z^d\setminus\{x\}}
		\p_{\beta,s,J} \left(\{x,y\} \text{ is open and pivotal for } A_n\right) \\
		& \hspace{22mm} \leq  \label{eq:general}
		C_2
		\sum_{x\in B_n(\mz): x \geq \mz} \ \sum_{y \in \Z^d\setminus\{x\}}
		\p_{\beta,s,J} \left(\{x,y\} \text{ is open and pivotal for } A_n\right) .
	\end{align}
	The last inequality follows by symmetry of the model and the symmetry of the event $A_n$. We only restrict to $x\geq \mz$ as we want $x - e_1$ to be well-defined inside the box $B_n(\mz)$ in the following.
	
	For long enough edges $\{x,y\}$, the probability $\p_{\beta,s,J}\left(\{x,y\} \text{ open}\right)$ is of the same order as the probability $\p_{\beta,s,J}\left(\{x-e_1,y\} \text{ open}\right)$, by condition \eqref{eq:comparability}. So for long enough edges $\{x,y\}$ we can `reroute' the edge $\{x,y\}$ to start at $x-e_1$ instead of $x$. The probability of the corresponding new event only differs by a constant multiplicative factor. For short edges, we can replace the open edge $\{x,y\}$ with two edges $\{x-e_1,u\}$ and $\{u,y\}$ for some $u \in B_n(\mz)$ with $\|u-x\|=\mathcal{O}(1)$ and $J(\{x-e_1,u\}), J(\{u,y\}) > 0$. Using such local modifications, we see that there exists a constant $C_3$ (that does not depend on $n$) such that for $n$ large enough and for all $x\in B_n(\mz)$ with $x\geq \mz$ one has
	\begin{align}\label{eq:local modify}
		\notag & \sum_{y \in \Z^d\setminus\{x\}} \p_{\beta,s,J} \left(\{x,y\} \text{ is open and pivotal for } A_n\right)\\
		& \hspace{7mm}
		\leq
		C_3 \sum_{y \in \Z^d\setminus\{x\}} \p_{\beta,s,J} \left(\{x,x-e_1\} \text{ and } \{x-e_1,y\} \text{ are both open and pivotal for } A_n\right).
	\end{align} 
	Next, we argue that
	\begin{multline}
		 \sum_{y \in \Z^d\setminus\{x\}} \p_{\beta,s,J} \left(\{x,x-e_1\} \text{ and } \{x-e_1,y\} \text{ are open and pivotal for } A_n\right)\\
		 \leq \label{eq:argue to show}
		\p_{\beta,s,J} \left(\{x,x-e_1\} \text{ is open and pivotal for } A_n\right).
	\end{multline} 
	To show inequality \eqref{eq:argue to show}, first note that the events of the form 
	\begin{equation*}
		\Big\{\{x,x-e_1\} \text{ and } \{x-e_1,y\} \text{ are open and pivotal for } A_n\Big\}
	\end{equation*}
	are disjoint for distinct $y \in \Z^d\setminus \{x\}$. This holds, as there can never be three or more open edges with $x-e_1$ as an endpoint that are pivotal for a connection event like $A_n$. Thus we get that 
	\begin{align*}
		& \sum_{y \in \Z^d\setminus\{x\}} \p_{\beta,s,J} \big(\{x,x-e_1\} \text{ and } \{x-e_1,y\} \text{ are both open and pivotal for } A_n\big)\\
		& \hspace{8mm}
		=
		\p_{\beta,s,J} \left( \bigcup_{y \in \Z^d\setminus\{x\}} \big\{\{x,x-e_1\} \text{ and } \{x-e_1,y\} \text{ are both open and pivotal for } A_n\big\}\right)\\
		& \hspace{8mm}
		\leq
		\p_{\beta,s,J} \left( \{x,x-e_1\} \text{ open and pivotal for } A_n \right).
	\end{align*}
	This shows \eqref{eq:argue to show}. Inserting inequalities \eqref{eq:local modify} and \eqref{eq:argue to show} into \eqref{eq:general}, we get that
	\begin{align*}
		&\sum_{e\in E_n} J(e) \p_{\beta,s,J} \left(e \text{ is pivotal for } A_n\right) \\
		& \hspace{22mm}
		\overset{\eqref{eq:general}}{\leq} C_2
		\sum_{x\in B_n(\mz): x \geq \mz} \ \sum_{y \in \Z^d\setminus\{x\}}
		\p_{\beta,s,J} \left(\{x,y\} \text{ is open and pivotal for } A_n\right)\\
		& \hspace{22mm}
		\overset{\eqref{eq:local modify}, \eqref{eq:argue to show}}{\leq} C_2 C_3
		\sum_{x\in B_n(\mz): x \geq \mz} 
		\p_{\beta,s,J} \left(\{x,x-e_1\} \text{ is open and pivotal for } A_n\right)\\
		& \hspace{22mm}
		\leq
		 C_2 C_3
		\sum_{e \in E_n^s} 
		\p_{\beta,s,J} \left(e \text{ is pivotal for } A_n\right)
	\end{align*}
	which finishes the proof of \eqref{eq:strict g(beta,s) 2} and thus the proof of Proposition \ref{prop:strict ineq}.
\end{proof}

\subsection{The proof of Theorem \ref{theo:p1 larger}}\label{sec:p1}

We are left to prove Theorem \ref{theo:p1 larger}. This follows from the strict inequality of critical points, Proposition \ref{prop:strict ineq}, as we show in the following.

\begin{proof}[Proof of Theorem \ref{theo:p1 larger}]
	Let $f:\Z^d \to \left[0,1\right)$ and let $p \in \left(p_c(f),1\right)$. Define $\bar{p} = \frac{p+p_c(f)}{2} \in \left(p_c(f),p\right)$. Choose $\beta > 0$ so that
	\begin{equation}\label{eq:betadef}
		e^{-\beta} = \frac{1-p}{1-\bar{p}}
	\end{equation}
	which is possible since $1-\bar{p} > 1-p$. Define a kernel $J : \Z^d \setminus \{\mz\}\to \left[0,\infty\right]$ by
	\begin{align*}
		1-e^{-\beta J(x)} = \begin{cases}
			\bar{p} & \text{ if } \|x\| = 1\\
			f(x) & \text{ if } \|x\| > 1
		\end{cases}.
	\end{align*}
	Thus we get that the two measures $\p_{\beta,J}$ and $\p_{\bar{p},f}$ agree. As there is an infinite open cluster under the measure $\p_{\bar{p},f}=\p_{\beta,J}$, this directly implies that $\beta \geq \beta_c(J)$. Define the kernel $\overline{J}$ by 
	\begin{align*}
		\overline{J}(x) = \begin{cases}
			J(x) +1 & \text{ if } \|x\| = 1\\
			J(x) & \text{ else }
		\end{cases}.
	\end{align*}
	By construction we have for all edges $\{x,y\}$ with $\|x-y\|> 1$ that
	\begin{equation*}
		\p_{\beta,\overline{J}}(\{x,y\} \text{ closed})
		=
		\p_{\beta,J}(\{x,y\} \text{ closed})
		=
		1-f(x-y)
		=
		\p_{p,f}(\{x,y\} \text{ closed}) .
	\end{equation*}
	For nearest-neighbor edges $\{x,y\}$ with $\|x-y\|=1$ we have by the definition of $\beta$ \eqref{eq:betadef} that
	\begin{equation*}
		\p_{\beta,\overline{J}}(\{x,y\} \text{ closed}) = e^{-\beta\overline{J}(x-y)} = e^{-\beta J(x-y)} e^{-\beta} = (1-\bar{p}) \frac{1-p}{1-\bar{p}} = \p_{p,f}(\{x,y\} \text{ closed})
	\end{equation*}
	and thus the two measures $\p_{p,f}$ and $\p_{\beta,\overline{J}}$ agree. As $f(x) \simeq \|x\|^{-s}$ for some $s>d$ by assumption \eqref{eq:fnecessary}, this directly implies that condition \eqref{eq:comparability} is satisfied. Thus we can apply Proposition \ref{prop:strict ineq} for the kernel $J$, and in particular, we get that	
	\begin{equation*}
		\beta \geq \beta_c(J) > \beta_c(\overline{J}).
	\end{equation*}
	So the measure $\p_{p,f}$ equals the measure $\p_{\beta,\overline{J}}$, which is a measure for supercritical long-range percolation on $\Z^d$. From here, one can easily verify that the different results stated in Theorem \ref{theo:p1 larger} hold for the percolation configuration sampled from the measure $\p_{p,f}$.
\end{proof}

\end{document}